\documentclass[12pt]{amsart}
\usepackage{amssymb}
\title[A characterization of Shimura varieties]{Stability of Hodge bundles and a numerical characterization of Shimura varieties}
\author[M. M\"oller]{Martin M\"oller}
\address{Max-Planck-Institut f\"ur Mathematik,
Vivatsgasse 7,
53111  Bonn, Germany}
\email{moeller@mpim-bonn.mpg.de  }
\author[E. Viehweg]{Eckart Viehweg}
\address{Universit\"at Duisburg-Essen, Mathematik, 45117 Essen, Germany}
\email{viehweg@uni-due.de}
\thanks{This work has been supported by the DFG-Leibniz program and by the SFB/TR 45
`Periods, moduli spaces and arithmetic of algebraic varieties'.}
\author[K. Zuo]{Kang Zuo}
\address{Universit\"at Mainz,
Fachbereich 17, Mathematik,
55099 Mainz, Germany}
\email{kzuo@mathematik.uni-mainz.de}
 \fussy
\begin{document}
\theoremstyle{plain}
\newtheorem{thm}{Theorem}[section]
\newtheorem{theorem}[thm]{Theorem}
\newtheorem{lemma}[thm]{Lemma}
\newtheorem{corollary}[thm]{Corollary}
\newtheorem{proposition}[thm]{Proposition}
\newtheorem{construction}[thm]{Construction}
\newtheorem{addendum}[thm]{Addendum}
\newtheorem{variant}[thm]{Variant}
\newtheorem{lemmadef}[thm]{Lemma and Definition}
\newtheorem{sect}[thm]{}
\theoremstyle{definition}
\newtheorem{notations}[thm]{Notations}
\newtheorem{notation}[thm]{Notation}
\newtheorem{exdef}[thm]{Example and Definition}
\newtheorem{lemdef}[thm]{Definition and Lemma}
\newtheorem{ass-not}[thm]{Assumptions and Notations}
\newtheorem{con-not}[thm]{Conclusion and Notations}
\newtheorem{lem-not}[thm]{Lemma and Notations}
\newtheorem{problem}[thm]{Problem}
\newtheorem{remark}[thm]{Remark}
\newtheorem{remarks}[thm]{Remarks}
\newtheorem{definition}[thm]{Definition}
\newtheorem{claim}[thm]{Claim}
\newtheorem{assumption}[thm]{Assumption}
\newtheorem{assumptions}[thm]{Assumptions}
\newtheorem{properties}[thm]{Properties}
\newtheorem{property}[thm]{Property}
\newtheorem{example}[thm]{Example}
\newtheorem{examples}[thm]{Examples}
\newtheorem{conjecture}[thm]{Conjecture}
\newtheorem{constr}[thm]{Construction}
\newtheorem{aconstr}[thm]{Allowed Constructions}
\newtheorem{condition}[thm]{Condition}
\newtheorem{questions}[thm]{Questions}
\newtheorem{question}[thm]{Question}
\newtheorem{setup}[thm]{Set-up}
\newtheorem{Variant}[thm]{Variant}
\numberwithin{equation}{section}
\catcode`\@=11
\def\opn#1#2{\def#1{\mathop{\kern0pt\fam0#2}\nolimits}}
\def\bold#1{{\bf #1}}%
\def\underrightarrow{\mathpalette\underrightarrow@}
\def\underrightarrow@#1#2{\vtop{\ialign{$##$\cr
 \hfil#1#2\hfil\cr\noalign{\nointerlineskip}%
 #1{-}\mkern-6mu\cleaders\hbox{$#1\mkern-2mu{-}\mkern-2mu$}\hfill
 \mkern-6mu{\to}\cr}}}
\let\underarrow\underrightarrow
\def\underleftarrow{\mathpalette\underleftarrow@}
\def\underleftarrow@#1#2{\vtop{\ialign{$##$\cr
 \hfil#1#2\hfil\cr\noalign{\nointerlineskip}#1{\leftarrow}\mkern-6mu
 \cleaders\hbox{$#1\mkern-2mu{-}\mkern-2mu$}\hfill
 \mkern-6mu{-}\cr}}}
\let\amp@rs@nd@\relax
\newdimen\ex@
\ex@.2326ex
\newdimen\bigaw@
\newdimen\minaw@
\minaw@16.08739\ex@
\newdimen\minCDaw@
\minCDaw@2.5pc
\newif\ifCD@
\def\minCDarrowwidth#1{\minCDaw@#1}
\newenvironment{CD}{\@CD}{\@endCD}
\def\@CD{\def\A##1A##2A{\llap{$\vcenter{\hbox
 {$\scriptstyle##1$}}$}\Big\uparrow\rlap{$\vcenter{\hbox{%
$\scriptstyle##2$}}$}&&}%
\def\V##1V##2V{\llap{$\vcenter{\hbox
 {$\scriptstyle##1$}}$}\Big\downarrow\rlap{$\vcenter{\hbox{%
$\scriptstyle##2$}}$}&&}%
\def\={&\hskip.5em\mathrel
 {\vbox{\hrule width\minCDaw@\vskip3\ex@\hrule width
 \minCDaw@}}\hskip.5em&}%
\def\verteq{\Big\Vert&&}%
\def\noarr{&&}%
\def\vspace##1{\noalign{\vskip##1\relax}}\relax\let\amp@rs@nd@&\iffalse}\fi
 \CD@true\vcenter\bgroup\relax\let\\=\cr\iffalse}\fi\tabskip\z@skip\baselineskip20\ex@
 \lineskip3\ex@\lineskiplimit3\ex@\halign\bgroup
 &\hfill$\m@th##$\hfill\cr}
\def\@endCD{\cr\egroup\egroup}
\def\>#1>#2>{\amp@rs@nd@\setbox\z@\hbox{$\scriptstyle
 \;{#1}\;\;$}\setbox\@ne\hbox{$\scriptstyle\;{#2}\;\;$}\setbox\tw@
 \hbox{$#2$}\ifCD@
 \global\bigaw@\minCDaw@\else\global\bigaw@\minaw@\fi
 \ifdim\wd\z@>\bigaw@\global\bigaw@\wd\z@\fi
 \ifdim\wd\@ne>\bigaw@\global\bigaw@\wd\@ne\fi
 \ifCD@\hskip.5em\fi
 \ifdim\wd\tw@>\z@
 \mathrel{\mathop{\hbox to\bigaw@{\rightarrowfill}}\limits^{#1}_{#2}}\else
 \mathrel{\mathop{\hbox to\bigaw@{\rightarrowfill}}\limits^{#1}}\fi
 \ifCD@\hskip.5em\fi\amp@rs@nd@}
\def\<#1<#2<{\amp@rs@nd@\setbox\z@\hbox{$\scriptstyle
 \;\;{#1}\;$}\setbox\@ne\hbox{$\scriptstyle\;\;{#2}\;$}\setbox\tw@
 \hbox{$#2$}\ifCD@
 \global\bigaw@\minCDaw@\else\global\bigaw@\minaw@\fi
 \ifdim\wd\z@>\bigaw@\global\bigaw@\wd\z@\fi
 \ifdim\wd\@ne>\bigaw@\global\bigaw@\wd\@ne\fi
 \ifCD@\hskip.5em\fi
 \ifdim\wd\tw@>\z@
 \mathrel{\mathop{\hbox to\bigaw@{\leftarrowfill}}\limits^{#1}_{#2}}\else
 \mathrel{\mathop{\hbox to\bigaw@{\leftarrowfill}}\limits^{#1}}\fi
 \ifCD@\hskip.5em\fi\amp@rs@nd@}
\newenvironment{CDS}{\@CDS}{\@endCDS}
\def\@CDS{\def\A##1A##2A{\llap{$\vcenter{\hbox
 {$\scriptstyle##1$}}$}\Big\uparrow\rlap{$\vcenter{\hbox{%
$\scriptstyle##2$}}$}&}%
\def\V##1V##2V{\llap{$\vcenter{\hbox
 {$\scriptstyle##1$}}$}\Big\downarrow\rlap{$\vcenter{\hbox{%
$\scriptstyle##2$}}$}&}%
\def\={&\hskip.5em\mathrel
 {\vbox{\hrule width\minCDaw@\vskip3\ex@\hrule width
 \minCDaw@}}\hskip.5em&}
\def\verteq{\Big\Vert&}
\def\novarr{&}
\def\noharr{&&}
\def\SE##1E##2E{\slantedarrow(0,18)(4,-3){##1}{##2}&}
\def\SW##1W##2W{\slantedarrow(24,18)(-4,-3){##1}{##2}&}
\def\NE##1E##2E{\slantedarrow(0,0)(4,3){##1}{##2}&}
\def\NW##1W##2W{\slantedarrow(24,0)(-4,3){##1}{##2}&}
\def\slantedarrow(##1)(##2)##3##4{%
\thinlines\unitlength1pt\lower 6.5pt\hbox{\begin{picture}(24,18)%
\put(##1){\vector(##2){24}}%
\put(0,8){$\scriptstyle##3$}%
\put(20,8){$\scriptstyle##4$}%
\end{picture}}}
\def\vspace##1{\noalign{\vskip##1\relax}}\relax\let\amp@rs@nd@&\iffalse}\fi
 \CD@true\vcenter\bgroup\relax\let\\=\cr\iffalse}\fi\tabskip\z@skip\baselineskip20\ex@
 \lineskip3\ex@\lineskiplimit3\ex@\halign\bgroup
 &\hfill$\m@th##$\hfill\cr}
\def\@endCDS{\cr\egroup\egroup}
\newdimen\TriCDarrw@
\newif\ifTriV@
\newenvironment{TriCDV}{\@TriCDV}{\@endTriCD}
\newenvironment{TriCDA}{\@TriCDA}{\@endTriCD}
\def\@TriCDV{\TriV@true\def\TriCDpos@{6}\@TriCD}
\def\@TriCDA{\TriV@false\def\TriCDpos@{10}\@TriCD}
\def\@TriCD#1#2#3#4#5#6{%
\setbox0\hbox{$\ifTriV@#6\else#1\fi$}
\TriCDarrw@=\wd0 \advance\TriCDarrw@ 24pt
\advance\TriCDarrw@ -1em
\def\SE##1E##2E{\slantedarrow(0,18)(2,-3){##1}{##2}&}
\def\SW##1W##2W{\slantedarrow(12,18)(-2,-3){##1}{##2}&}
\def\NE##1E##2E{\slantedarrow(0,0)(2,3){##1}{##2}&}
\def\NW##1W##2W{\slantedarrow(12,0)(-2,3){##1}{##2}&}
\def\slantedarrow(##1)(##2)##3##4{\thinlines\unitlength1pt
\lower 6.5pt\hbox{\begin{picture}(12,18)%
\put(##1){\vector(##2){12}}%
\put(-4,\TriCDpos@){$\scriptstyle##3$}%
\put(12,\TriCDpos@){$\scriptstyle##4$}%
\end{picture}}}
\def\={\mathrel {\vbox{\hrule
   width\TriCDarrw@\vskip3\ex@\hrule width
   \TriCDarrw@}}}
\def\>##1>>{\setbox\z@\hbox{$\scriptstyle
 \;{##1}\;\;$}\global\bigaw@\TriCDarrw@
 \ifdim\wd\z@>\bigaw@\global\bigaw@\wd\z@\fi
 \hskip.5em
 \mathrel{\mathop{\hbox to \TriCDarrw@
{\rightarrowfill}}\limits^{##1}}
 \hskip.5em}
\def\<##1<<{\setbox\z@\hbox{$\scriptstyle
 \;{##1}\;\;$}\global\bigaw@\TriCDarrw@
 \ifdim\wd\z@>\bigaw@\global\bigaw@\wd\z@\fi
 \mathrel{\mathop{\hbox to\bigaw@{\leftarrowfill}}\limits^{##1}}
 }
 \CD@true\vcenter\bgroup\relax\let\\=\cr\iffalse}\fi
 \tabskip\z@skip\baselineskip20\ex@
 \lineskip3\ex@\lineskiplimit3\ex@
 \ifTriV@
 \halign\bgroup
 &\hfill$\m@th##$\hfill\cr
#1&\multispan3\hfill$#2$\hfill&#3\\
&#4&#5\\
&&#6\cr\egroup%
\else
 \halign\bgroup
 &\hfill$\m@th##$\hfill\cr
&&#1\\%
&#2&#3\\
#4&\multispan3\hfill$#5$\hfill&#6\cr\egroup
\fi}
\def\@endTriCD{\egroup}
\newcommand{\sA}{{\mathcal A}}
\newcommand{\sB}{{\mathcal B}}
\newcommand{\sC}{{\mathcal C}}
\newcommand{\sD}{{\mathcal D}}
\newcommand{\sE}{{\mathcal E}}
\newcommand{\sF}{{\mathcal F}}
\newcommand{\sG}{{\mathcal G}}
\newcommand{\sH}{{\mathcal H}}
\newcommand{\sI}{{\mathcal I}}
\newcommand{\sJ}{{\mathcal J}}
\newcommand{\sK}{{\mathcal K}}
\newcommand{\sL}{{\mathcal L}}
\newcommand{\sM}{{\mathcal M}}
\newcommand{\sN}{{\mathcal N}}
\newcommand{\sO}{{\mathcal O}}
\newcommand{\sP}{{\mathcal P}}
\newcommand{\sQ}{{\mathcal Q}}
\newcommand{\sR}{{\mathcal R}}
\newcommand{\sS}{{\mathcal S}}
\newcommand{\sT}{{\mathcal T}}
\newcommand{\sU}{{\mathcal U}}
\newcommand{\sV}{{\mathcal V}}
\newcommand{\sW}{{\mathcal W}}
\newcommand{\sX}{{\mathcal X}}
\newcommand{\sY}{{\mathcal Y}}
\newcommand{\sZ}{{\mathcal Z}}
\newcommand{\A}{{\mathbb A}}
\newcommand{\B}{{\mathbb B}}
\newcommand{\C}{{\mathbb C}}
\newcommand{\D}{{\mathbb D}}
\newcommand{\E}{{\mathbb E}}
\newcommand{\F}{{\mathbb F}}
\newcommand{\G}{{\mathbb G}}
\newcommand{\BH}{{\mathbb H}}
\newcommand{\I}{{\mathbb I}}
\newcommand{\J}{{\mathbb J}}
\newcommand{\BL}{{\mathbb L}}
\newcommand{\M}{{\mathbb M}}
\newcommand{\N}{{\mathbb N}}
\newcommand{\BP}{{\mathbb P}}
\newcommand{\Q}{{\mathbb Q}}
\newcommand{\R}{{\mathbb R}}
\newcommand{\BS}{{\mathbb S}}
\newcommand{\T}{{\mathbb T}}
\newcommand{\U}{{\mathbb U}}
\newcommand{\V}{{\mathbb V}}
\newcommand{\W}{{\mathbb W}}
\newcommand{\X}{{\mathbb X}}
\newcommand{\Y}{{\mathbb Y}}
\newcommand{\Z}{{\mathbb Z}}
\newcommand{\fg}{{\mathfrak g }}
\newcommand{\fk}{{\mathfrak k }}
\newcommand{\fa}{{\mathfrak a }}
\newcommand{\rk}{{\rm rk}}
\newcommand{\ch}{{\rm c}}
\newcommand{\rounddown}[1]{\llcorner{#1}\lrcorner}
\newcommand{\fM}{{\mathfrak M}}
\newcommand{\bsl}{{\backslash \,}}
\newcommand{\Mon}{{\rm Mon}}
\newcommand{\MT}{{\rm MT}}
\newcommand{\Hg}{{\rm Hg}}
\newcommand{\tr}{{\rm tr}}
\newcommand{\Sl}{{\rm Sl}}
\newcommand{\Gl}{{\rm Gl}}
\newcommand{\Sp}{{\rm Sp}}
\def\bigtimes{\mathop{\mbox{\Huge$\times$}}}
\newcommand{\can}{{\rm can}}
\newcommand{\Aut}{{\rm Aut}}
\newcommand{\Res}{{\rm Res}}
\newcommand{\Lie}{{\rm Lie}}
\newcommand{\der}{{\rm der}}
\newcommand{\ad}{{\rm ad}}
\newcommand{\pr}{{\rm pr}}
\newcommand{\nc}{{\rm nc}}
\newcommand{\Hom}{{\rm Hom}}
\newcommand{\GL}{{\rm GL}}
\newcommand{\SU}{{\rm SU}}
\newcommand{\SO}{{\rm SO}}
\newcommand{\OO}{{\rm O}}
\newcommand{\CSp}{{\rm CSp}}
\newcommand{\Stab}{{\rm Stab}}
\newcommand{\mov}{{\rm mov}}

\begin{abstract}
Let $U$ be a connected non-singular quasi-projective variety and $f:A \to U$ 
a family of abelian varieties of dimension $g$. Suppose that the induced 
map $U \to \sA_g$ is generically finite and there is a compactification $Y$ 
with complement $S = Y \setminus U$ a normal crossing divisor such that
 $\Omega^1_Y(\log S)$ is nef and $\omega_Y(S)$ is ample with respect to $U$. 

We characterize whether $U$ is a Shimura variety by numerical data 
attached to the variation of Hodge structures, rather than by properties of the 
map $U \to \sA_g$ or by the existence of CM points. 

More precisely, we show that $f:A\to U$ is a Kuga fibre space, if and only
if two conditions hold. First, each irreducible local subsystem $\V$ of 
$R^1 f_* \C_A$ is either unitary or satisfies the Arakelov equality.
Second, for each factor $M$ in the universal cover of $U$ whose tangent 
bundle behaves like the one of a complex ball, an iterated Kodaira-Spencer map associated
with $\V$ has minimal possible length in the direction of $M$. If in addition $f:A\to U$
is rigid, it is a connected Shimura subvariety of $\sA_g$ of Hodge type.
\end{abstract}
\maketitle

\tableofcontents

Let $Y$ be a non-singular complex projective variety of dimension $n$, and 
let $U$ be the complement of a normal crossing divisor $S$. 
We are interested in families $f:A\to U$ of polarized abelian varieties, up to isogeny,
and we are looking for numerical invariants which take the minimal possible value if and only if $U$ is a Shimura variety of certain type, or to be more precise, if $f:A\to U$ is a Kuga fibre space as recalled in Section~\ref{Kugadef}. Those invariants will be attached to $\C$-subvariations of Hodge structures $\V$ of 
$R^1f_*\C_A$. We will always assume that the family has semistable reduction in codimension one,
hence that the local system $R^1f_*\C_A$ has unipotent monodromy around the components of $S$.

In \cite{VZ04} we restricted ourselves to curves $Y$, and we gave a characterization of Shimura curves in terms of the degree of $\Omega^1_Y(\log S)$ and the degree of the Hodge bundle $f_*\Omega^1_{X/Y}(\log f^{-1}(S))$ 
for a semistable model $f:X\to Y$ of $A\to U$. For infinitesimally rigid families this description was an easy consequence of
Simpson's correspondence, whereas in the non-rigid case we had to use the classification of 
certain discrete subgroups of $\BP\Sl_2(\R)$. In \cite{VZ07} we started to study families over a higher dimensional base $U$, restricting ourselves to the rigid case. There it became evident that one has to consider numerical invariants of all the irreducible $\C$-subvariations of Hodge structures $\V$ of $R^1f_*\C_A$, and that for ball quotients one needed some condition on the second Chern classes, or equivalently on the length of the Higgs field of certain wedge products of $\V$. In \cite{VZ07} we have chosen the condition that the discriminant of one of the Hodge bundles is zero. This was needed to obtain the purity of the Higgs bundles (see Definition~\ref{pure}) for the special variations of Hodge structures considered there, but it excluded several standard representations. 

In this article we give a numerical characterization of a Shimura variety of Hodge type, or of a Kuga fibre space in full generality, including rigid and non-rigid ones. In order to state and to motivate the results, we need some notations.

Consider a complex polarized variation of Hodge structures $\V$ on $U$ of weight $k$, as defined in
\cite[page 4]{De87} (see also \cite[page 898]{Sim}), and with unipotent local monodromy around the components of $S$. The $\sF$-filtration on $\sV_0=\V\otimes_\C\sO_U$ extends to a filtration of the
Deligne extension $\sV$ of $\sV_0$ to $Y$, again denoted by $\sF$ (see \cite{Sch73}).
By Griffiths' Transversality Theorem (see \cite{Gr70}, for example) the Gauss-Manin connection $\nabla:\sV\to \sV\otimes\Omega_Y^1(\log S)$
induces an $\sO_Y$-linear map
$$   
{\mathfrak gr}_\sF(\sV) = \bigoplus_{p+q=k} E^{p,q} \>\bigoplus\theta_{p,q} >> \bigoplus_{p+q=k} E^{p,q}\otimes \Omega^1_Y(\log S) ={\mathfrak gr}_\sF(\sV)\otimes \Omega^1_Y(\log S),
$$ 
with $\theta_{p,q}: E^{p,q} \to E^{p-1,q+1}\otimes \Omega^1_Y(\log S)$. So by \cite{Si92}   
$\big(E={\mathfrak gr}_\sF(\sV), \ \theta=\bigoplus\theta_{p,q} \big)$ 
is the (logarithmic) Higgs bundle induced by $\V$. We will write $\theta^{(m)}$ for the iterated
Higgs field
\begin{multline}\label{eqitKS}
E^{k,0} \>\theta_{k,0} >> E^{k-1,1}\otimes \Omega^1_Y(\log S) \>\theta_{k-1,1} >> E^{k-2,2}\otimes S^2(\Omega^1_Y(\log S))  \>\theta_{k-2,2} >> \\ \cdots \>\theta_{k-m+1,m-1}  >> E^{k-m,m}\otimes S^m(\Omega^1_Y(\log S)) .
\end{multline}

For families of polarized abelian varieties we are considering subvariations $\V$ of the complex polarized variation of Hodge structures $R^1f_*\C_A$. Of course, $\V$ is polarized by restricting the polarization of $R^1f_*\C_A$,
and $\V$ has weight $1$. Then its Higgs field is of the form 
$$
(E=E^{1,0}\oplus E^{0,1},\theta)\mbox{ \ \  with \ \ }\theta:E^{1,0}\to E^{0,1} \otimes \Omega^1_Y(\log S).
$$

The most important numerical invariant is the slope $\mu(\V)$ of $\V$ or of the Higgs bundle $(E,\theta)$. 
Recall that the slope $\mu(\sF)$ of a torsion free coherent sheaf 
$\sF$ on $Y$, is defined by the rational number
\begin{equation}\label{slope}
\mu(\sF):=\frac{\ch_1(\sF)}{\rk(\sF)}.\ch_1(\omega_Y(S))^{\dim(Y)-1}.
\end{equation}
Correspondingly we define $\mu(\V):=\mu(E^{1,0})-\mu(E^{0,1})$. As we will see, 
$\mu(\V)$ is related to $\mu$-stability, a concept which will be defined in~\ref{fil.2}.
\par
Variations of Hodge structures of weight $k>1$ will only occur as tensor representations of $\W_\Q=R^1f_*\Q_A$ or of irreducible direct factors $\V$ of $R^1f_*\C_A$, in particular in the definition of the second numerical invariant:
\par 
Given a Higgs bundle 
$$
\big(E=E^{1,0}\oplus E^{0,1}, \ \theta:E^{1,0}\to E^{0,1}\otimes \Omega^1_Y(\log S)\big)
$$ 
and some $\ell >0$ one has the induced Higgs bundle
\begin{gather}\notag
\bigwedge^\ell (E,\theta)=\Big(\bigoplus_{i=0}^\ell E^{\ell-i,i}, \ \bigoplus_{i=0}^{\ell-1} \theta_{\ell-i,i} \Big)
\mbox{ \ \ with}\\ \label{eqdetHiggs}
E^{\ell-m,m}=\bigwedge^{\ell-m}(E^{1,0}) \otimes \bigwedge^m (E^{0,1})\mbox{ \ \ and with}\\ \notag
\theta_{\ell-m,m}: \bigwedge^{\ell-m}(E^{1,0}) \otimes \bigwedge^m (E^{0,1}) \>>>
\bigwedge^{\ell-m-1}(E^{1,0}) \otimes \bigwedge^{m+1} (E^{0,1}) \otimes \Omega_Y^1(\log S)
\end{gather}
induced by $\theta$. 
\par
If $\ell=\rk(E^{1,0})$, then $E^{\ell,0}=\det(E^{1,0})$. In this case $\langle\det(E^{1,0})\rangle$ denotes the Higgs subbundle of $\bigwedge^\ell(E,\theta)$ generated by $\det(E^{1,0})$. Writing as in (\ref{eqitKS}) 
$$
\theta^{(m)}=\theta_{\ell-m+1,m-1}\circ \cdots \circ \theta_{\ell,0},
$$
we define as a measure for the complexity of the Higgs field
\begin{multline*}
\varsigma((E,\theta)): = {\rm Max}\{\ m\in \N ; \ \theta^{(m)}(\det(E^{1,0}))\neq 0\}=\\
{\rm Max}\{\ m\in \N ; \ \langle\det(E^{1,0})\rangle^{\ell-m,m} \neq 0\}.
\end{multline*}
If $(E,\theta)$ is the Higgs bundle of a variation of Hodge structures $\V$ we will usually write $\varsigma(\V)=\varsigma((E,\theta))$.
\par
We require some positivity properties of the sheaf of differential forms on the compactification $Y$ of $U$:
\begin{assumptions}\label{compass-pos}
We suppose that the compactification $Y$ of $U$ is a non-singular projective 
algebraic variety, such that  $S=Y\setminus U$ is a normal crossing divisor, and such that
\begin{enumerate}
\item[$\bullet$] $\Omega_Y^1(\log S)$ is nef and $\omega_Y(S)=\Omega^n_Y(\log S)$ is ample with respect to $U$.
\end{enumerate}
\end{assumptions}
By definition a locally free sheaf $\sF$ is {\em numerically effective (nef)} if for all morphisms
$\tau:C\to Y$, with $C$ an irreducible curve, 
and for all invertible quotients $\sN$ of $\tau^*\sF$ one has $\deg(\sN) \geq 0$. An invertible sheaf
$\sL$ is {\em ample with respect to $U$} if for some $\nu\geq 1$ the sections in $H^0(Y,\sL^\nu)$ generate the sheaf $\sL^\nu$ over $U$ and if the induced morphism $U\to \BP(H^0(Y,\sL^\nu))$ is an embedding.
\par
If $U$ is the base of a Kuga fibre space or more generally if
the universal covering $\pi:\tilde{U}\to U$ is a bounded symmetric domain,
we will need a second type of condition to hold true for the compactification $Y$ of $U$:
\begin{condition}\label{compass-pos2} Assume that the universal covering
$\tilde{U}$ of $U$ decomposes as the product $M_1\times \cdots \times M_s$ of irreducible 
bounded symmetric domains.
\begin{enumerate}
\item[$\bullet$] 
Then the sheaf $\Omega^1_Y(\log S)$ is $\mu$-polystable. If 
$\Omega^1_Y(\log S)=\Omega_1\oplus\cdots\oplus\Omega_{s'}$ 
is the decomposition as a direct sum of $\mu$-stable sheaves, then 
$s=s'$ and for a suitable choice of the indices 
$\pi^* \Omega_i|_U = {\rm pr}_i^*\Omega_{M_i}^1$. 
\end{enumerate}
\end{condition}
As we will recall in Section~\ref{verify_arakelov}, Mumford studied in \cite[Section 4]{Mu77}  non-singular toroidal compactifications $Y$ of some finite \'etale covering of the base $U$ of a Kuga fibre space, satisfying the Assumption~\ref{compass-pos}. 
As we will see later, by Yau's Uniformization Theorem \cite{Ya93} the Assumption~\ref{compass-pos} implies the Condition~\ref{compass-pos2}, but in the special case of such a {\it Mumford compactifications} we will verify Condition~\ref{compass-pos2} directly. 
\begin{proposition}\label{shimuraprop}
Let $f:A\to U$ be a Kuga fibre space, such that the induced polarized variation of Hodge structures $\W=R^1f_*\C_A$ has unipotent local monodromies at infinity.
Then, replacing $U$ by a finite \'etale covering if necessary, there exists a compactification $Y$ satisfying the Assumption~\ref{compass-pos} and the Condition~\ref{compass-pos2}, such that for all irreducible non-unitary $\C$ subvariations of Hodge structures $\V$ of $\W$ with Higgs bundle $(E,\theta)$ one has:
\begin{enumerate}
\item[i.] There exists some $i=i(\V)$ such that the Higgs field $\theta$ factorizes through
$$
\theta:E^{1,0} \>>> E^{0,1}\otimes \Omega_i \> \subset >> E^{0,1}\otimes \Omega^1_Y(\log S).
$$
\item[ii.] The `Arakelov equality' $\mu(\V)=\mu(\Omega^1_Y(\log S))$ holds.
\item[iii.] The sheaves $E^{1,0}$ and $E^{0,1}$ are $\mu$-stable.
\item[iv.]  The sheaf $E^{1,0}\otimes {E^{0,1}}^\vee$ is $\mu$-polystable.
\item[v.] Assume for $i=i(\V)$ that $M_i$ is a complex ball of dimension $n_i\geq 1$. Then 
$$
\varsigma(\V)=\frac{\rk(E^{1,0})\cdot\rk(E^{0,1})\cdot(n_i+1)}{\rk(E) \cdot n_i}.
$$
\end{enumerate}
\end{proposition}
We will verify the first four of those properties, presumably well known to experts, 
at the end of Section~\ref{verify_arakelov}. The fifth one will be proved 
in Section~\ref{stablengthsplit}. 
\par
The Arakelov equality in ii) is our main condition. By 
Lemma~\ref{AEindep} its validity is independent of the compactification, as 
long as the Assumptions~\ref{compass-pos} hold. 
As we will see in Section~\ref{stablengthsplit}, assuming the Arakelov equality and assuming that $Y$ is the compactification constructed by Mumford, the properties iii) and iv) are equivalent and in case that 
$M_{i(\V)}$ is a complex ball they are equivalent to v), as well. 

Our main interest is the question, which of the conditions stated in Proposition~\ref{shimuraprop} 
will force an arbitrary family $f:A\to U$ of abelian varieties to be a Kuga fibre space. 
We will need the existence of a projective compactification $Y$ of $U$ satisfying the Assumption~\ref{compass-pos}. Remark however, that this condition automatically holds true for compact non-singular subvarieties $U=Y$ of the fine moduli scheme $\sA_g^{[N]}$ of polarized abelian varieties of dimension $g$ with a level $N$ structure for $N\geq 3$. 

Assuming~\ref{compass-pos}, the Arakelov equality in Proposition~\ref{shimuraprop}  
says that for Kuga fibre spaces certain numerical invariants take the maximal possible value. 
In fact, for any polarized family $f:A\to U$ of abelian varieties and for any irreducible $\C$-subvariation of Hodge structures $\V$ on $U$ in $R^1f_*\C_A$ with Higgs bundle $(E,\theta)$ the unipotency of the 
local monodromies at infinity implies by \cite[Theorem 1]{VZ07} the Arakelov type inequality
\begin{equation}\label{Arakelovinequ}
\mu(\V)=\mu(E^{1,0})-\mu(E^{0,1}) \leq \mu(\Omega^1_Y(\log S)).
\end{equation}
The Arakelov equality
\begin{equation}\label{Arakelovequ}
\mu(\V)= \mu(\Omega^1_Y(\log S))
\end{equation}
can only hold if $E^{1,0}$ and $E^{0,1}$ are both $\mu$-semistable.

For the next step, we need Yau's Uniformization Theorem
(\cite{Ya93}, recalled in \cite[Theorem 1.4]{VZ07}), saying in particular that the Assumption~\ref{compass-pos}
forces the sheaf $\Omega_Y^1(\log S)$ to be $\mu$-polystable. So one has again a direct sum decomposition
\begin{equation}\label{eqint.1}
\Omega_Y^1(\log S)=\Omega_1\oplus \cdots \oplus \Omega_s.
\end{equation}
in $\mu$-stable sheaves of rank $n_i=\rk(\Omega_i)$. 
We say that {\em $\Omega_i$ is of type A}, if it is invertible, 
and {\em of type B}, 
if $n_i>1$ and if for all $m>0$ the sheaf $S^m(\Omega_i)$ is $\mu$-stable. 
Finally it is {\em of type C} in the remaining cases, i.e.\ if 
for some $m>1$ the sheaf $S^m(\Omega_i)$ is $\mu$-unstable, hence a direct sum of two or 
more $\mu$-stable subsheaves.

Let again $\pi:\tilde{U} \to U$ denote the universal covering with covering group 
$\Gamma$. As in the Condition~\ref{compass-pos2}, the decomposition (\ref{eqint.1}) of $\Omega^1_Y(\log S)$ corresponds to a product structure 
\begin{equation}\label{eqfactors}
\tilde{U}=M_1\times \cdots \times M_s,
\end{equation}
where $n_i=\dim(M_i)$. If $\tilde{U}$ is a bounded symmetric domain, the $M_i$ in~\ref{eqfactors}
are irreducible bounded symmetric domains. If the image of the fundamental group is an arithmetic
group there exists a Mumford compactification and the decomposition
\eqref{eqint.1}  coincides with the one in Condition~\ref{compass-pos2}.

Yau's Uniformization Theorem gives in addition a criterion for each $M_i$ to be a bounded symmetric domain.
In fact, if $\Omega_i$ is of type A, then $M_i$ is a one-dimensional complex ball. It is a bounded symmetric domain of rank $>1$, if $\Omega_i$ is of type C.

If $\Omega_i$ is of type B, then $M_i$ is a $n_i$-dimensional complex ball if and only if
\begin{equation}\label{eqyau}
\big[2\cdot (n_i+1)\cdot \ch_2(\Omega_i)-n_i\cdot \ch_1(\Omega_i)^2\big].\ch(\omega_Y(S))^{\dim(Y)-2}=0.
\end{equation}
\par
Before being able to give the numerical characterization of Kuga fibre spaces, hence a converse of Proposition~\ref{shimuraprop},~ii), and v) we will have to state some result on the splitting of variations of Hodge structures, similar to Proposition~\ref{shimuraprop},~i). 
\begin{definition}\label{pure} \ 
\begin{enumerate}
\item[1.] A subsheaf $\sF\subset E^{1,0}$ is {\em pure of type $i$} 
if the composition
$$
\sF \> \subset >> E^{1,0} \> \theta >> E^{0,1}\otimes \Omega^1_Y(\log S) \> {\rm pr} >>
E^{0,1}\otimes \Omega_j$$
is zero for $j\neq i$ and non-zero for $j=i$.
\item[2.] A variation of Hodge structures $\V$ (or the corresponding Higgs bundle $(E^{1,0}\oplus E^{0,1},\theta)$) is {\em pure of type $i$}, if $E^{1,0}$ is pure of type $i$.
\item[3.] If $\V$ (or $(E,\theta)$) is pure of type $i$ and if $\Omega_i$ is of type
A, B, or C, we sometimes just say that $\V$ (or $(E,\theta)$) {\em is pure 
of type A, B, or C}.
\end{enumerate}
\end{definition}
Consider the Higgs bundles $(E,\theta_j)$ with the pure Higgs field $\theta_j$, given by the composite
$$
\theta_j: E^{1,0} \> \theta >> E^{0,1}\otimes \Omega_Y^1(\log S) \> {\rm pr}_j >> E^{0,1}\otimes \Omega_j
\> \subset >> E^{0,1}\otimes \Omega_Y^1(\log S).
$$
In general $(E,\theta_j)$ will not correspond to a variation of Hodge structures. 

However if $(E,\theta)$ is the Higgs bundle of a non-unitary variation of Hodge structures,
it is pure of type $i$ if and only if $\theta_j$ is zero for $j\neq i$.
Moreover one has $\theta_i=\theta$ in this case.

Proposition~\ref{shimuraprop},~i) states that for Kuga fibre spaces the variations of Hodge structures decompose as a direct sum of pure and of unitary subvariations.

If in the decomposition (\ref{eqint.1}) all the $\mu$-stable direct factors 
$\Omega_i$ are of type C, hence if $\tilde{U}$ is the product of bounded 
symmetric domains $M_i=G_i/K_i$ of rank $>1$, the Margulis Superrigidity Theorem and a simple induction argument (see the proof of Proposition~\ref{erg.1}) imply that up to tensor products with unitary representations each representation $\rho$ of the fundamental group $\Gamma$ is coming from a representation of the group $G=G_1\times \cdots\times G_s$. Then by Schur's lemma the irreducibility
of $\rho$ implies that it is the tensor product of representations $\rho_j$ of the $G_j$.
Correspondingly an irreducible variation of Hodge structures $\V$ is the tensor product of a unitary bundle and of polarized $\C$ variations of Hodge structures $\V_j$ given by $\rho_j$. Since the weight of $\V$ is one, all the $\V_j$, except for one, have to be variations of Hodge structures of weight zero, hence they are also unitary and the induced Higgs field is zero. So $\V$ is pure. As we will see in Proposition~\ref{erg.1} one can extend this result to all $U$ with $\tilde{U}$ a bounded symmetric domain.

The next theorem extends this property in another way, replacing the 
condition that $\tilde{U}$ is a bounded symmetric domain by the Arakelov equality.
\par
\begin{theorem}\label{purity_Thm}
Under the Assumptions~\ref{compass-pos} consider an 
irreducible non-unitary polarized $\C$-variation of Hodge structures $\V$ of weight $1$ with 
unipotent monodromy at infinity. If $\V$ satisfies the Arakelov 
equality (\ref{Arakelovequ}), then $\V$ is pure for some $i=i(\V)$.
\end{theorem}
The proof of Theorem~\ref{purity_Thm} will cover most of the Sections~\ref{fil},
\ref{spl} and~\ref{sta}. We will have to consider small twists of the slopes
$\mu(\sF)$. 

Applying Simpson's correspondence \cite{Si92} to the Higgs subbundle $\langle\det(E^{1,0})\rangle$ 
of $\bigwedge^{\rk(E^{1,0})}(E,\theta)$ we will obtain in Lemma~\ref{upper} and Lemma~\ref{numcondA}:
\begin{corollary}\label{iteratedKSineq}
Assume in Theorem~\ref{purity_Thm} that for $i=i(\V)$ the sheaf $\Omega_i$ is of type A or B. Then 
\begin{equation}\label{lenghtineq}
\varsigma(\V) \geq \frac{\rk(E^{1,0})\cdot\rk(E^{0,1})\cdot(n_i+1)}{\rk(E) \cdot n_i}.
\end{equation}
\end{corollary}
If $\Omega_i$ is invertible, hence of type A, we will see in Lemma~\ref{numcondA}
that both the left hand side and the right hand side of~\ref{lenghtineq} are equal to
$\rk(E^{1,0})$.

The main result of this article characterizes a Kuga fibre space as a family of abelian varieties
$f:A\to U$ for which the slopes $\mu(\V)$ are maximal and the complexity $\varsigma(\V)$
is minimal for all $\C$-subvariations of Hodge structures $\V\subset R^1f_*\C_A$:
\par
\begin{theorem}\label{characterization}
Let $f:A \to U$ be a family of polarized abelian varieties such that
$R^1f_*\C_A$ has unipotent local monodromies at infinity, and such that the induced 
morphism $U\to \sA_g$ is generically finite. Assume that $U$ has a projective compactification $Y$ satisfying the Assumptions~\ref{compass-pos}. The the following two conditions are equivalent:
\begin{enumerate}
\item[a.] There exists an \'etale covering $\tau:U'\to U$ such that $f':A'=A\times_UU' \to U'$ is a Kuga fibre space. 
\item[b.] For each irreducible subvariation of Hodge structures $\V$ of $R^1f_*\C_A$
with Higgs bundle $(E,\theta)$ one has:
\begin{enumerate}
\item[1.] Either $\V$ is unitary or the Arakelov equality $\mu(\V)=\mu(\Omega_Y^1(\log S))$ holds.
\item[2.] If for a $\mu$-stable direct factor $\Omega_i$ of $\Omega^1_Y(\log S)$ 
of type B the composition
$$
\theta_i:E^{1,0} \> \theta >> E^{0,1}\otimes \Omega_Y^1(\log S) \> {\rm pr} >> E^{0,1}\otimes \Omega_i \> \subset >> E^{0,1}\otimes \Omega_Y^1(\log S)
$$
is non-zero, then $\displaystyle
\varsigma((E,\theta_i)) = \frac{\rk(E^{1,0})\cdot\rk(E^{0,1})\cdot(n_i+1)}{\rk(E) \cdot n_i}$.
\end{enumerate}
\end{enumerate}
If under the assumption a) or b) $f:A\to U$ is infinitesimally rigid, then $U'$ is a Shimura variety of Hodge type.
\end{theorem}
Remark that the condition 1) in part b) of Theorem~\ref{characterization} implies that
in part 2) $\theta= \theta_i$ and that $i=i(\V)$.

We call $f:A\to U$ rigid, if the induced morphism $U\to \sA_g$ to the moduli stack has no non-trivial deformations, hence if there is no smooth projective morphism $\hat{f}:\hat{A} \to U\times T$
with $\dim(T)>0$ and extending $f$, such that the induced morphism $U\times T \to \sA_g$ is generically finite.
In a similar way, $f:A\to U$ is called infinitesimally rigid, if the morphism from $U$ to the moduli stack
has no infinitesimal deformations. Using Faltings' description of the infinitesimal deformations (see \cite{Fa83}) 
this holds if and only if there are no antisymmetric endomorphisms of the variation of Hodge structures pure of type $(-1,1)$. In particular, if ${\rm End}(R^1f_*\Q_{X})^{-1,1}=0$ the family is infinitesimally rigid.

Here we should point out, that throughout this article a Shimura variety of Hodge type is defined over $\C$ and it is an irreducible component of a Shimura variety of Hodge type in the sense of \cite{Mi04} and, as we will explain in Section~\ref{ofHodgetype}, it is defined `up to \'etale coverings'. We use the same convention for Kuga fibre spaces, and we allow ourselves to replace $U$ by an \'etale covering, whenever it is convenient (see Section~\ref{ec}).

Remark that the condition 1) in Theorem~\ref{characterization},~b) allows to apply Theorem~\ref{purity_Thm}. Since the condition 2) 
automatically holds true if $\V$ is unitary, or if it is pure of type A or C, we can as well restate
the condition 2) as
\begin{enumerate}
\item[2'.] {\it If $\V$ is non unitary and pure of type $i=i(\V)$ with $\Omega_i$ of type B, then 
$$
\varsigma(\V) = \frac{\rk(E^{1,0})\cdot\rk(E^{0,1})\cdot(n_i+1)}{\rk(E) \cdot n_i}.
$$}
\end{enumerate}
Obviously $\varsigma(\V)$ is determined by the Higgs bundle on any open dense subset of $U$
and compatible with replacing $U$ by a finite \'etale covering $U'$. By Lemma~\ref{AEindep} the slopes $\mu(\V)$ and $\mu(\Omega^1_Y(\log S))$ are multiplied in this case by the degree of $U'$ over $U$,
as long as one chooses compactifications $Y$ of $U$ and $Y'$ of $U'$ satisfying both the Assumption~\ref{compass-pos}. Using the Proposition~\ref{shimuraprop} one obtains:
\begin{corollary}\label{characterization2}
Assume in Theorem~\ref{characterization} that $U$ has a projective compactification, satisfying the Assumptions~\ref{compass-pos}, and such that on this compactification the conditions 1) and 2) hold for the Higgs bundles of all irreducible subvariation of Hodge structures $\V$ of $R^1f_*\C_A$. Then there exists an \'etale covering $\tau:U'\to U$ and a compactification $Y'$ of $U'$ with $S'=Y'\setminus U'$ a normal crossing divisor, satisfying again the Assumptions~\ref{compass-pos}, such that for all subvariations $\V'$ of Hodge structures in $\tau^*R^1f_*\C_A$ with Higgs bundle $(E',\theta')$ one has:
\begin{enumerate}
\item[3.] $E'^{1,0}$ and $E'^{0,1}$ are $\mu'$-stable.
\item[4.] $E'^{1,0}\otimes {E'^{0,1}}^\vee$ is $\mu'$-polystable.
\end{enumerate}
Here the slopes $\mu'$ are defined with $\omega_{Y}(S)$ replaced by $\omega_{Y'}(S')$.
\end{corollary}
The proof of Theorem~\ref{characterization} will be given in Section~\ref{stablengthsplit}.
As indicated, the subvariations of Hodge structures which are pure of type B will play a special role. In Section~\ref{stablengthsplit} we will obtain a slightly more precise information. 
\par
\begin{addendum}\label{characterizationadd} Consider in Theorem~\ref{characterization} 
an irreducible complex polarized subvariation of Hodge structures $\V$ of $R^1f_*\C_{A}$ with Higgs bundle $(E,\theta)$.
Assume that $\V$ is non-unitary and satisfies the Arakelov equality. Consider the following conditions
for $i=i(\V)$:
\begin{enumerate}
\item[$\alpha$.] $E^{1,0}$ and $E^{0,1}$ are $\mu\vspace{.1cm}$-stable.
\item[$\beta$.] The kernel of the natural map $\sH om(E^{0,1},E^{1,0})\to \Omega_i$ 
is a direct factor of $\sH om(E^{0,1},E^{1,0})$.\vspace{.1cm}
\item[$\gamma$.] \ \hspace*{\fill} $\displaystyle
\varsigma(\V) = \frac{\rk(E^{1,0})\cdot\rk(E^{0,1})\cdot(n_i+1)}{\rk(E) \cdot n_i}.
$ \hspace*{\fill} \ \vspace{.1cm}
\item[$\delta$.] $M_i$ is the complex ball $\SU(1,n_i)/K$, and $\V$ is the tensor product of a unitary representation with a wedge product of the standard representation of $\SU(1,n_i)$ (as explained
in Section~\ref{stablengthsplit} before Proposition~\ref{kugaAB})
\item[$\eta$.] Let $M'$ denote the period domain for $\V$. Then the period map factors as the projection $\tilde{U} \to M_i$ and a totally geodesic embedding $M_i \to M'$.
\end{enumerate}
Then, depending on the type of $\Omega_i$, the following holds:
\begin{enumerate}
\item[I.] If $\Omega_i$ is of type A, then $\alpha$), $\beta$), $\gamma$), $\delta$) and $\eta$) hold true.
\item[II.] If $\Omega_i$ is of type C, then $\eta$) holds true. 
\item[III.] If $\Omega_i$ is of type B, then the conditions $\beta$) and $\gamma$) are equivalent. They imply 
the conditions $\delta$) and $\eta$). 
\end{enumerate}
\end{addendum}
\par
\begin{lemma}\label{poly} Assume that $\omega_Y(S)$ is ample. 
\begin{itemize}
\item[($\star$)]
If $\sF$ and $\sG$ are two $\mu$-stable locally free sheaves, then $\sF\otimes \sG$ is $\mu$-polystable.
\end{itemize}
\end{lemma}
Since the Arakelov equality says that the slopes of $\sH om(E^{0,1},E^{1,0})$ and of $\Omega_i$ coincide,
the Lemma~\ref{poly} shows that $\alpha$) implies $\beta$). Since $\delta)$ implies $\alpha$) we can state:
\begin{corollary} \label{characterizationadd2} In the Addendum~\ref{characterizationadd} one has:
\begin{enumerate}
\item[IV.] If $\omega_Y(S)$ is ample, for example if $U$ is projective or if $\dim(U)=1$, 
and if $\Omega_i$ is of type B, then the conditions $\alpha$), $\beta$), $\gamma$) and $\delta$) are equivalent and imply  $\eta$).
\end{enumerate} 
\end{corollary}
S.T. Yau conjectures, that Property ($\star$) in Lemma~\ref{poly} remains true if $\omega_Y(S)$ is only nef and big.
Hopefully there will soon be a proof in a forthcoming article by Sun and Yau. This would allow to drop the condition
on the ampleness of $\omega_Y(S)$ in Corollary~\ref{characterizationadd2}. Without referring to Yau's conjecture one still has:
\begin{corollary}\label{characterizationadd3} 
Assume in Theorem~\ref{characterization} and in Addendum~\ref{characterizationadd}
that the Arakelov equality and the condition $\eta$) hold for all non-unitary subvariations $\V$ of $R^1f_*\C_{A}$.
Then there exists an \'etale covering $\tau:U'\to U$ with $U'$ a quotient of a bounded symmetric domain by an arithmetic group. Moreover on a Mumford  compactification of $U'$ the conditions $\alpha$), $\beta$), $\delta$), and $\gamma$) are equivalent for all irreducible non-unitary subvariations $\V'$ of $\tau^*R^1f_*\C_{A}$ which are of type B.
\end{corollary}
In \cite{VZ07} we had to exclude direct factors of $\Omega^1_Y(\log S)$ of 
type C, and we used a different numerical condition for $\V$ of type B.
Recall that the discriminant of a torsion free coherent sheaf $\sF$ on $Y$ is 
given by
$$
\delta(\sF)=\big[2\cdot \rk(\sF)\cdot\ch_2(\sF)-(\rk(\sF)-1)\cdot \ch_1
(\sF)^2\big].\ch_1(\omega_Y(S))^{\dim(Y)-2},
$$
and that the $\mu$-semistability of $E^{1-q,q}$ implies that 
$\delta(E^{1-q,q}) \geq 0$. So the Arakelov equality
implies that
$$
\delta(\V): ={\rm Min}\{\delta(E^{1,0}), \delta(E^{0,1})\} \geq 0.
$$
In \cite{VZ07}, using the condition ($\star$) in Lemma~\ref{poly}, we gave two criteria forcing 
$f:A\to U$ to be a Kuga fiber space.
The first one, saying that all the direct factors of $\Omega_Y^1(\log S)$ are 
of type A, is now a special case of Theorem~\ref{characterization}. In the 
second criterion we allowed the direct factors of $\Omega_Y^1(\log S)$
to be of type A and B, but excluded factors of type C. There, for all irreducible subvariations $\V$ of Hodge structures we required $\delta(\V)=0$.
This additional condition, needed in \cite{VZ07} to prove the purity of irreducible 
subvariations of Hodge structures, forced at the same time the representations in Addendum~\ref{characterizationadd}, $\delta$) to be the tensor product of the standard 
representations of $\SU(1,n_i)$ with a unitary local system and excluded their wedge products.

The bridge between the criterion \cite{VZ07} and Theorem~\ref{characterization} is already contained in
\cite[Proposition 3.4]{VZ07}:
\begin{remark}
Let $f:A \to U$ be a family of polarized abelian varieties such that
$R^1f_*\C_A$ has unipotent local monodromies at infinity, and such that the induced 
morphism $U\to \sA_g$ is generically finite. Assume that $U$ has a projective 
compactification $Y$ satisfying the Assumptions~\ref{compass-pos}.
Then the condition ($\star$) in Lemma~\ref{poly} implies:\\[.1cm]
Let $\V$ be an irreducible subvariation of Hodge structures of $R^1f_*\C_A$ 
with Higgs bundle $(E,\theta)$, pure of type $i=i(\V)$ and with $\Omega_i$ of type B. If $\V$ satisfies the 
Arakelov equality and if $\delta(\V)=0$, then either 
\begin{equation} \label{special_case}
\rk(E^{1,0})=\rk(E^{0,1})\cdot n_{i}\mbox{ \ \  or \ \ }\rk(E^{1,0})\cdot 
n_{i}=\rk(E^{0,1}).
\end{equation}
In particular the condition $\gamma$) in  Addendum~\ref{characterizationadd} holds for $\V$. 
\end{remark}
In fact, by definition $\varsigma(\V) \leq {\rm Min} \{ \rk(E^{1,0}), \rk(E^{0,1})\}$.
Corollary~\ref{iteratedKSineq} and the numerical condition~\ref{special_case} imply that
this is an equality.
\par
We do not know whether the condition 2) in Theorem~\ref{characterization},~b) is 
really needed, or whether in Addendum~\ref{characterizationadd}, for $\V$ of type B, the condition $\gamma$) follows from the Arakelov equality. As we will show in Section~\ref{ex} 
this is the case for $\rk(\V) \leq 7$, provided that $\omega_Y(S)$ is ample or more generally if the Condition~\ref{poly2}, generalizing the condition ($\star$) in Lemma~\ref{poly}, holds true. However, the necessity of the equality~\eqref{eqyau} in the 
characterization of ball quotients might indicate that
a condition on the first Chern class, as given by the Arakelov equality, 
can not be sufficient to characterize complex balls. 

Up to now we did not mention any condition guaranteeing the existence of fibres with complex multiplication or the equality between the monodromy group and the derived Mumford-Tate group $\MT(f)^\der$ (see Section~\ref{mt}), usually needed in the construction of Shimura varieties of Hodge type.
In fact, as in \cite{Mo98}, we will rather concentrate on the condition that $U\to \sA_g$ is totally geodesic. This will allow in the proof of Theorem~\ref{characterization} to identify $f:A\to U$ with a Kuga fibre space $\sX(G,\tau,\varphi_0)$. Next, for rigid families we will refer to \cite{Abd94} and \cite{Mo98} for the proof that they are Shimura varieties of Hodge type (see Section~\ref{KugaShimura} for more details), hence that there are fibres with complex multiplication.

This implies that for a rigid family $f:A\to U$ the group $\MT(f)^\der$ is 
the smallest the $\Q$-algebraic subgroup containing the monodromy group and 
that $U$ is up to \'etale coverings equal to $\sX(\MT(f)^\der,{\rm id},
\varphi_0)$. 

In \cite{VZ07} we used for the last step an explicit identification of possible Hodge cycles.
Although not really needed, we will sketch a similar calculation in Section~\ref{aemt}. There it will be sufficient to assume that the non-unitary irreducible direct factors of $R^1f_*\C_A$ satisfy the Arakelov equality, and we will explicitly construct a subgroup $\MT^\mov(f)^\der$, isomorphic to the monodromy group $\Mon^0(f)$, which up to constant factors coincides with the Mumford-Tate group $\MT(f)^\der$. Using the notations of Section~\ref{Kugadef}, this implies that $\sX(\Mon^0(f),{\rm id },\varphi_0) \cong \sX(\MT^\mov(f)^\der,{\rm id },\varphi_0)$. 
\vspace{.2cm}

It is a pleasure to thank Ngaiming Mok, for several letters explaining his results on geometric rigidity, in particular for the proof of Claim~\ref{Mokarg}. 
Parts of this note grew out of discussions between the second and third named author during a visit at the East China Normal University in Shanghai. We would like to thank the members of its Department of Mathematics for their hospitality. 

We are grateful to the referee of an earlier version of this article, who pointed out several ambiguities and mistakes, in particular in Sections~\ref{stablengthsplit} 
and~\ref{aemt}.

\section{Kuga fibre spaces and Shimura varieties of Hodge type}\label{KugaShimura}
\subsection{Kuga fibre spaces and totally geodesic subvarieties}
\label{Kugadef}
The data to construct a {\em Kuga fibre space} (see \cite{Mu69} and the references therein)
are 
\begin{itemize}
\item[i.] a rational vector space $V$ of dimension $2g$ with a lattice $L$,
\item[ii.] a non-degenerate skew-symmetric bilinear form $Q: V\times V \to \Q$, integral on 
$L \times L$,
\item[iii.] a $\Q$-algebraic group 
$G$ and an injective map $\tau: G \to \Sp(V,Q)$,
\item[iv.] an arithmetic subgroup $\Gamma \subset G$ such that $\tau(\Gamma)$
preserves $L$,
\item[v.] a complex structure 
$$
\varphi_0:S^1=\{z \in \C^*\, ; \,|z| = 1\} \to \Sp(V,Q)
$$
such that $\tau(G)$ is normalized by $\varphi_0(S^1)$ and such that
$Q(v,\varphi_0(\sqrt{-1})v)>0$ for all $v\in V\setminus \{0\}$. 
\end{itemize}
\par
We will allow ourselves to replace the arithmetic subgroup in iv) by a subgroup of finite index, whenever it is 
convenient. In particular, we will assume that $\Gamma$ is neat, as defined in \cite[page 599]{Mu77}.

For $\Gamma$ sufficiently small, $(L,Q,G,\tau,\varphi_0,\Gamma)$ defines a 
Kuga fibre space, i.e.\ a family of abelian varieties, by the following procedure. Let 
$K_\R^0$ be the connected component of the centralizer of $\varphi_0(S^1)$
in $G_\R$. Then there is a map 
$$M:=G_\R^0/K_\R^0 \>>> \Sp(V,Q)_\R/(\text{centralizer of}\, \varphi_0)
\cong \BH_g$$
and the pullback of the universal family over $\BH_g$ descends to the
the desired family over 
$$\sX :=\sX(G,\tau,\varphi_0):= \Gamma \backslash G_\R^0/K_\R^0.$$
In the sequel we will usually suppress $V$ and $Q$ from the notation
and write just $\Sp(Q)$ or $\Sp$, if no ambiguity arises.
\par
Two different sets of data $(L,Q,G,\tau,\varphi_0,\Gamma)$ and $(L',Q',G',\tau',\varphi_0',\Gamma')$ may 
define isomorphic Kuga fibre spaces over $\sX(G,\tau,\varphi_0)\cong \sX(G',\tau',\varphi_0')$. 
Note that different groups $G$ and $G'$ might lead to the same Kuga fibre
space and that $K_\R^0$ is not necessarily compact but the extension
of a central torus in $G_\R$ by a compact group. Note moreover
that replacing $\varphi_0$ by $\tau(g)\varphi_0\tau(g)^{-1}$ for
any $g \in G$ gives an isomorphic Kuga fibre space - this just
changes the reference point.
\par
Kuga fibre spaces are the objects that naturally arise when
studying polarized variations of Hodge structures satisfying the Arakelov equality. We restrict
the translation procedure into the language of Shimura varieties 
to the case of `Hodge type', see Section~\ref{ofHodgetype}.
\par
We provide symmetric domains throughout with the Bergman metric 
(e.g.\ \cite[\S II.6]{Sa80}). By condition v) in Mumford's definition of a Kuga fibre space, $M \to \BH_g$ is a strongly equivariant map in the sense of \cite{Sa80}. By 
\cite[Theorem II.2.4]{Sa80}, it
is a {\em totally geodesic embedding}, i.e.\ each geodesic curve in 
$\BH_g$ which is tangent to $M$ at some point of $M$ is a curve in $M$. The converse
is dealt with in Section~\ref{ofHodgetype}.
\par
\subsection{\'Etale coverings}\label{ec}
Replacing the group $\Gamma$ by a subgroup of finite index corresponds to replacing $U$ by an \'etale covering,
and by definition one obtains again a Kuga fibre space. So we will consider Kuga fibre spaces and
Shimura varieties (see Section~\ref{ofHodgetype}) as equivalence classes up to \'etale coverings.
The way we stated Theorem~\ref{characterization} or the Corollary~\ref{characterizationadd3} 
we are allowed to replace $U$ by an \'etale covering, whenever it is convenient. 

Since $U\to \sA_g$ is induced by a genuine family of polarized abelian varieties $f:A\to U$ and since the subgroup of $N$-division points is \'etale over $U$, an \'etale covering $U'$ of $U$ maps to the moduli scheme $\sA_g^{(N)}$ of abelian varieties with a level $N$ structure, say for $N=3$. We will drop the ${}'$ as well as the ${}^{(N)}$, and we will assume in the sequel: 
\begin{assumptions}\label{fine}
$\sA_g$ is a fine moduli scheme, $\varphi:U\to \sA_g$ is generically finite, and $f:A\to U$ is the pullback of the universal family.
\end{assumptions}
As we will see in the beginning of the Section~\ref{stablengthsplit}, for $\varphi$ finite and $\varphi(U)$ non-singular the Arakelov equality will force $\varphi$ to be \'etale. At other places, for example if we talk about geodesics, we will have to assume that $\varphi(U)$ is non-singular, and that $\varphi$ is \'etale. Then however, since $\sA_g$ is supposed to be a fine moduli scheme, we can as well assume that $\varphi$ is an embedding.
\par
\subsection{The Hodge group, the 
Mumford-Tate group and the monodromy group}\label{mt}

We start be recalling the definitions of the Hodge and Mumford-Tate group.
Let $A_0$ be an abelian variety and $W_\Q=H^1(A_0,\Q)$, equipped with
the polarization $Q$. The {\em Hodge group $\Hg(A_0)=\Hg(W_\Q)$} is defined in \cite{Mu66} (see also \cite{Mu69}) as the smallest $\Q$-algebraic 
subgroup of $\Sp(W_\Q,Q)$, whose extension to $\R$ contains the complex 
structure
$$
\varphi_0:S^1 \>>> \Sp(W_\Q,Q),
$$
where $z$ acts on $(p,q)$ cycles by multiplication with $z^p\cdot \bar z^q$. 
\par
In a similar way, one defines the {\em Mumford-Tate group $\MT(W_\Q)=\MT(A_0)$}.
The complex structure $\varphi_0$ extends to a morphism of real algebraic groups
$$
h^{W_\Q}:{\rm Res}_{\C/\R}\G_m \>>>\Gl(W_\Q \otimes \R),
$$
and $\MT(W_\Q)$ is the smallest $\Q$-algebraic subgroup of $\Gl(W_\Q)$, whose 
extension to $\R$ contains the image of $h^{W_\Q}$. 

By \cite{De82} the group $\MT(W_\Q)$ is reductive, and it 
coincides with the largest $\Q$-algebraic subgroup of the linear group $\Gl(W_\Q)$,
which leaves all $\Q$-Hodge tensors invariant, hence all elements
$$
\eta \in \big[W_\Q^{\otimes m} \otimes W_\Q^{\vee \otimes m'}\big]^{0,0}.
$$
Here $W_\Q^\vee$ is regarded as a Hodge structure concentrated in the bidegrees $(0,-1)$ and $(-1,0)$, and hence $W_\Q^{\otimes m} \otimes W_\Q^{\vee \otimes m'}$ is of weight $m-m'$. So the existence of some $\eta$ forces $m$ and $m'$ to be equal.
\par
Let $f:A\to U$ be a family of polarized abelian varieties and $\W_\Q=R^1f_*\Q_A$
the induced polarized $\Q$-variation of Hodge structures on $U$. By \cite{De82}, \cite{An92} or \cite{Sc96} there exist a union $\Sigma$ of countably many proper closed subvarieties of $U$ such that for $y\in U\setminus \Sigma$ the group $\MT(\W_\Q|_y)$ is independent of $y$. 
We will fix such a `very general' point $y$, write $W_\Q$ instead of $W_\Q|_y$. We define
$\MT(\W_\Q)$ or $\MT(f)$ to be $\MT(W_{\Q})$. 

The monodromy group $\Mon(\W_\Q)$ is defined as the smallest 
$\Q$-algebraic subgroup of $\Gl(W_\Q)$ which contains the image of the 
monodromy representation of $\pi_1(U,y)$, and $\Mon^0(\W_\Q)$ denotes its connected component containing the identity. We will often write $\Mon^0$ or $\Mon^0(f)$ instead of $\Mon^0(\W_\Q)$. 
\par
By \cite{De82} $\Mon^0(\W_\Q)$ is a normal subgroup of the derived subgroup $\MT(\W_\Q)^\der$. Note that the derived subgroup of the Hodge group $\Hg(A_0)$ coincides with the derived Mumford-Tate group $\MT(R^1f_*\Q_A)^\der$.
\par
\subsection{Shimura varieties of Hodge type and totally geodesic 
subvarieties} \label{ofHodgetype}

A Kuga fibre space $\sX(G,\tau,\varphi_0)$ is {\em of Hodge type}, 
if it is isomorphic to a Kuga fibre space $\sX(G',\tau',\varphi_0')$
such that $G'$ is the Hodge group of the abelian variety defined by $\varphi_0'$.
Let us next compare this notion with the one of Shimura varieties of Hodge type. 
\par
In \cite{De79}, the notion of a {\em connected 
Shimura datum} $(G,M)$ consists of a reductive $\Q$-algebraic group $G$
and a $G(\R)^+$-conjugacy class $M$ of homomorphisms 
$h: {\rm Res}_{\C/\R}\G_m  \to G_\R$ with the following properties:
\begin{itemize}
\item[(SV1)] for $h \in M$, only the characters $z/\overline{z}$, $1$,
$\overline{z}/z$ occur in the representation of ${\rm Res}_{\C/\R}\G_m$
on $\Lie(G)$.
\item[(SV2)] $\ad (h(i))$ is a Cartan involution of $G^\ad$.
\item[(SV3)] $G^\ad$ has no $\Q$-factor on which the projection 
of $h$ is trivial.
\end{itemize}
\par
A {\em connected Shimura variety} is defined to be the pro-system 
$(\Gamma \backslash M)_\Gamma$, with $\Gamma$ running over all arithmetic subgroups
$\Gamma$ of $G(\Q)$ whose image in $G^\ad$ is Zariski-dense. 
Since we do not bother about canonical models and since
we allow to replace the base $U$ by an \'etale cover
any time, we say that $U$ {\em is a Shimura variety of Hodge
type}, if $U$ is equal to $\Gamma \backslash M$ for some $\Gamma$.
Usually $\Gamma$ is required moreover to be a congruence subgroup, 
but we drop this condition to simplify matters of passing to
\'etale covers at some places.
\par
We let $\CSp(Q)$ (or $\CSp$ for short) be the group of symplectic similitudes 
with respect
to a symplectic form $Q$. The Shimura datum $(\CSp(Q), M(Q))$ 
attached to the symplectic space consists of all maps
$h: {\rm Res}_{\C/\R}\G_m \to \CSp(Q)_\R$ defined on $\R$-points by
the block diagonal matrix 
\begin{equation} \label{defsympl_h}
h(x+iy) = {\rm diag}\left(\left(\begin{matrix}
 x & -y \\ y & x \\
\end{matrix}\right),\ldots,\left(\begin{matrix}
 x & -y \\ y & x \\
\end{matrix}\right)\right)
\end{equation}
with respect to a symplectic basis $\{a_i,b_i\}$, $i=1,\ldots,g$ 
of the underlying vector space $V$.
\par
A Shimura datum $(G,M)$ is {\em of Hodge type}, if there is a map
$\tau: G \to \CSp(Q)$ such that composition with $\tau$ maps
$M$ to $M(Q)$.
\par
There is a bijection between isomorphism classes of Kuga fibre 
spaces of Hodge type and the universal families of Shimura varieties of Hodge type: 

Given $(L,Q,G,\tau,\varphi_0,\Gamma)$,
let $Z \cong \G_m$ be the center of $\CSp$, define $G' := G \cdot Z
\subset \CSp$ and define $h: {\rm Res}_{\C/\R}\G_m  \to G'_\R$ by 
on $\C$-points by
$h(z) = \varphi_0(z/\overline{z})|z|$. Finally, let $M'$ be the 
$G'_\R$ conjugacy class of $h$. One checks that
$(G',M')$ is a Shimura datum of Hodge type. Conversely given $(G',M')$
of Hodge type, let $G:= G' \cap \Sp$ and let $\varphi_0$ be the
restriction of a generic $h \in M'$ to $S^1 \subset 
{\rm Res}_{\C/\R}\G_m (\C)$. Together with $\tau$ being
the inclusion map, this defines a Kuga fibre space of Hodge type.
\par
\smallskip
We keep the Assumptions~\ref{fine}. We will not assume at the moment that $U$ is a Shimura variety or that any numerical condition holds on the variation of Hodge structure. 
We follow Moonen (\cite{Mo98}) and recall the construction of the smallest Shimura subvariety $\sX^\MT$ of Hodge type in $\sA_g$ that contains the image of $U$.  
\par
\begin{theorem}[\cite{Mo98}] \label{Moonen_exist_Shim}
There exists a Shimura datum $(G, M)$ such that a Shimura
variety $\sX^\MT \cong \Gamma \backslash M$ attached to this
Shimura datum is the  unique smallest Shimura subvariety of
Hodge type in $\sA_g$ that contains the image of $U$.
\par
$G$ may be chosen to be the Mumford-Tate group 
at a very general point $y$ of $U$. 
\end{theorem}
\par
\par
Although the Shimura variety $\sX^\MT$ is unique, the Shimura datum is 
unique only up to the centralizer of $G$ in $\CSp$, 
see \cite[Remark~2.9]{Mo98}.
\par
\begin{proof}
Let $G$ be the Mumford-Tate group at a very general point $y$ of $U$. 
In the topological space of  all maps $h \in M(Q)$ that factor 
through $G_\R$, choose $M$ to be the connected component containing the complex
structure at $y$.
By definition of the Mumford-Tate group, $M$ is not empty and
by the argument of \cite[Lemma~1.2.4]{De79}, $M$ is an $G(\R)^+$-conjugacy
class. Hence $(G,M)$ is a Shimura datum of Hodge type.
Since $y$ was very general, $\varphi: U \to \sA_g$ factors through $\sX^\MT$.
The minimality of $\sX^\MT$ follows from the minimality condition in the
definition of the Mumford-Tate group.
\end{proof}
\par
We now suppose that $U$ is a totally geodesic non-singular subvariety 
of the Shimura variety $\sX^\MT\subset \sA_g$. As in Section~\ref{ec} we can also allow 
a morphism $\varphi:U\to \sA_G$ as long as $\varphi(U)\subset \sA_g$ is a non-singular 
totally geodesic subvariety and $U\to \varphi(U)$ \'etale.
\par
\begin{theorem}[\cite{Mo98} Corollary 4.4]\label{Moonen_product}
If $U \subset \sX^{\MT}$ is totally geodesic, then $U$ is the base of a
Kuga fibre space. It is a Shimura variety of Hodge type 
up to some translation in the following sense:
\par
After replacing $U$ by a finite
\'etale cover, there are Kuga fibre spaces over $\sX_1$ and $\sX_2$ and an isomorphism
$\sX_1 \times \sX_2 \to \sX^\MT$, such that $U$ is the image of $\sX_1 \times
\{b\}$ for some point $b \in \sX_2(\C)$.
\par
For some $a \in \sX_2(\C)$, the subvariety $\sX_1 \times \{a\}$
in $\sX^\MT$ is a Shimura variety of Hodge type.
\end{theorem}
\par
\begin{proof}
In loc.~cit.\ the author deals with Shimura subvarieties of arbitrary 
period domains and shows that there totally geodesic subvarieties
$\sX_i$ such that $U$ is the image of $\sX_1 \times \{b\}$.
\par
We repeat part of his arguments to justify that $\sX_1$
is the base of a Kuga fibre space.
\par 
More precisely, let $(G,M)$ be the Shimura datum underlying $\sX^\MT$.
We have a decomposition of
the adjoint Shimura datum
$$ (G^\ad, M) \cong ((\Mon^0)^\ad, M_1) \times (G_2^\ad,M_2) $$
into connected Shimura data given as follows.
Since $G$ is reductive, there is a complement $G_2$ 
of $\Mon^0$, i.e.\ such that $\Mon^0 \times G_2 \to G$ is surjective
with finite kernel. Write $G_1:= \Mon^0$ and let $M_i$ be the set of maps
$${\rm Res}_{\C/\R}\G_m \>>> G \>>> G^\ad \>>> (G_i)^\ad.$$ 
For suitable arithmetic subgroups $\Gamma_i$ a component of the quotients
$\sX_i := \Gamma_i \backslash M_i$ have the claimed property 
by \cite{Mo98} Corollary 4.4.
\par
It suffices to take $\tau: \Mon^0 \to \Sp$ the natural inclusion
and $\varphi_0$ the restriction of any $h \in M$ to $S^1 \subset \C^*$.
Then $\varphi_0$ normalizes $\Mon^0$ and for a suitable choice of $\Gamma$,
$U$ is the base of the Kuga fibre space given by $(L,Q,\Mon^0,\tau,\varphi_0,\Gamma)$.
\end{proof}
\par
\begin{corollary}[See also \cite{Abd94}] \label{Moonen_rigid}
If in Theorem~\ref{Moonen_product} the subvariety $U$ is totally geodesic and rigid, 
then $U$ is a Shimura variety of Hodge type.  
\end{corollary}
Here rigidity just means that the inclusion $U\to \sA_g$ does not extend to a non-trivial morphism
$U\times T \to \sA_g$. Since we assumed that $\sA_g$ is a fine moduli scheme, this is equivalent
to the fact that the induced family $f:A\to U$ is rigid.

\section{Stability for homogeneous bundles and the
Arakelov equality for Shimura varieties} \label{verify_arakelov}

To prove a first part of the properties of Kuga fibre space stated
in Proposition~\ref{shimuraprop} we recall 
from \cite{Mu77} and \cite{Mk87} some facts on homogeneous vector bundles
on Hermitian symmetric domains and deduce stability results.
\par
Let $M$ be a  Hermitian symmetric domain and let $G = \Aut(M)$
be the holomorphic isometries of $M$.
$\Aut(M)$ is the identity component of the isometry group
of $M$ and $M \cong G/K$ for a maximal compact subgroup $K \subset G$.
Let $V_0$ be a vector space with a representation
$\rho: K \to \Gl(V_0)$ and any $\rho$-invariant metric $h_0$.
Then 
$$V = G \times_K V_0:= G \times V_0/\sim, \quad \text{where}
\quad (g,v) \sim (gk,\rho(k^{-1})v) \quad \text{for} \quad  k\in K $$ 
with the metric $h$ inherited from $h_0$ is 
a vector bundle on $G/K$, homogeneous under the action of $G$, or as we will say, 
a {\em homogeneous} bundle.
\par
Let $U$ be non-singular algebraic variety. 
In this section we suppose that the universal covering of $U$ is
a symmetric domain $M=G/K$ and that the image of the fundamental group
of $U$ in $G$ is a neat arithmetic subgroup. We call a bundle $E_U$ on $U$ {\em homogeneous}, if its 
pullback to $M$ is homogeneous. We call $E_U$ {\em irreducible}, if the pullback is given by an irreducible
representation $\rho$.
\par
For the rest of this section, we work over a smooth toroidal compactification $Y$ of $U$ 
with $S=Y\setminus U$ a normal crossing divisor, as studied in \cite{Mu77}. If $Y^*$ denotes the Baily-Borel compactification of $U$, there exists a morphism $\delta:Y \to Y^*$ whose restriction to $U$ is the identity.
\par
Obviously, the cotangent bundle of a symmetric domain $M=G/K$ is the
homogeneous bundle associated with the adjoint representation on $(\Lie(G)/\Lie(K))^\vee$.
\par
We will not need the exact definition of a singular Hermitian metric, `good on $Y$' in
the sequel. Let us just recall that this implies that the curvature of the Chern 
connection $\nabla_h$ of $h$ represents the first chern class of $E$. 
\par
\begin{theorem}[\cite{Mu77} Theorem 3.1 and Proposition 3.4] 
\label{mumford1} \ 
\begin{enumerate}
\item[a.] Suppose that $E_U$ is a homogeneous bundle with Hermitian metric $h$ induced
by $h_0$ as above. Then there exists a unique locally free sheaf $E$ 
on $Y$ with $E|_U=E_U$, such that $h$ is a singular Hermitian metric 
good on $Y$.
\item[b.] For $E_U=\Omega_U^1$ one obtains the extension $E=\Omega_Y^1(\log S)$.
\item[c.] For $E_U=\omega_U$ one obtains the extension $E=\omega_Y(S)$ and this sheaf is the pullback of an invertible ample sheaf on $Y^*$.
\end{enumerate}
\end{theorem}
\par
\begin{corollary}\label{mumford2}
Assume that $U$ maps to the moduli stack $\sA_g$ of polarized abelian varieties, 
and that this morphism is induced from a homomorphism $G\to \Sp$ by taking the
double quotient with respect to the maximal compact subgroup and a lattice as
in Section~\ref{KugaShimura}. 
\par
Then the Mumford compactification $Y$ satisfies the Assumptions~\ref{compass-pos} 
and Condition~\ref{compass-pos2}. 
\end{corollary}
\begin{proof}
If the bounded symmetric domain $M$ decomposes as 
$M_1\times \cdots \times M_s$, hence if $\Aut(M)=:G =
G_1\times \cdots \times G_s$, the sheaves $\Omega^1_{M_i}$ are 
homogeneous bundles associated with $(\Lie(G_i)/\Lie(K_i))^\vee$.
They descend to sheaves $\Omega_{i\, U}$ on $U$ which extend to 
$\Omega_i$ on $Y$. The uniqueness of the extensions implies that 
$\Omega_Y^1(\log S)=\Omega_1\oplus\cdots \oplus \Omega_s$.

Let $f:A\to U$ denote the universal family over $U$, and let $F^{1,0}_U=f_*\Omega^1_{A/U}$
denote the Hodge bundle.
Since $U\to \sA_g$ is induced by a homomorphism $G\to \Sp$, and since the bundle
$\Omega^1_{\sA_g}$ is homogeneous on $\sA_g$, its pullback to $U$ is homogeneous under $G$.
The latter is isomorphic to $S^2(F^{1,0}_U)$. 
 
The sheaf $\Omega^1_U$ is a homogeneous direct factor, hence the uniqueness of the extension in
Theorem~\ref{mumford1} implies that $\Omega_Y^1(\log S)$ is a direct factor of the
extension of $S^2(F^{1,0}_U)$ to $Y$. We may assume that the local monodromies of
$R^1f_*\C_A$ around the components of $S=Y\setminus U$ are unipotent. Then
the Mumford extension is $S^2(F^{1,0})$, where $F=F^{1,0}\oplus F^{0,1}$ is the logarithmic Higgs bundle of $R^1f_*\C_A$. Moreover, as shown by Kawamata (e.g.\ \cite[Theorem 6.12]{Vi95}), the sheaf
$F^{1,0}$ is nef. So $S^2(F^{1,0})$ and the direct factor $\Omega_Y^1(\log S)$ are both nef.

The ampleness of $\omega_Y(S)$ follows directly from the second part of \cite[Proposition 3.4]{Mu77}.
In fact, as remarked in the proof of \cite[Proposition 4.2]{Mu77}, this sheaf is just the pullback
of the ample sheaf on the Baily-Borel compactification of $U$.

It remains to verify that $\Omega^1_Y(\log S)$ is $\mu$-polystable and that 
for all $i$ 
$\Omega_i$ is $\mu$-stable. 

Using standard calculation of Chern characters on products, as in Section 
\ref{spl}, it is easy to show that the slopes $\mu(\Omega_i)$ coincide with 
$\mu(\Omega^1_Y(\log S))$. The $\mu$-stability of $\Omega_i$ follows from Lemma~\ref{hom_is_stable} by
a case by case verification that for $M_i$ irreducible the representation attached to the
homogeneous bundle $\Omega_{M_i}$ is irreducible.

Alternatively, since we have verified the Assumptions~\ref{compass-pos}, 
we can use Yau's Uniformization Theorem, stated in \cite[Theorem 1.4]{VZ07}. 
It implies that $\Omega_Y^1(\log S)$ is $\mu$-polystable.
Then the sheaves $\Omega_i$, constructed above, are $\mu$-polystable as 
well. Moreover, if $\Omega_i$ decomposes as a direct sum of two $\mu$-polystable subsheaves the corresponding $M_i$ is the product of two subspaces. So if we choose the
decomposition $M=M_1\times \cdots \times M_s$ with $M_i$ irreducible, the
sheaves $\Omega_i$ are $\mu$-stable.
\end{proof}
\begin{example}\label{Deligne_Mumford1}
Let $E^{p,q}_U$ be a Hodge bundle of a uniformizing $\C$-variation of Hodge structures $\V$ over $U$. Then $E^{p,q}_U$ is a homogeneous vector bundle and the corresponding invariant metric $h$ is the Hodge metric, induced by the variation of Hodge structures.
Let $Y$ be a Mumford compactification of $U$. By Theorem~\ref{mumford1} there exists a good extension of $E^{p,q}_U$ to $Y$. 
 
On the other hand, as described in the introduction, one has the canonical Deligne extension of $\V\otimes_\C\sO_U$ to $Y$. The compatibility of this extension
with the $\sF$-filtration (see \cite{Sch73}) gives another extension $E^{p,q}$ of $E_U^{p,q}$ to $Y$.
\end{example}
\begin{lemma}\label{Deligne_Mumford2}
In the Example~\ref{Deligne_Mumford1} the canonical Deligne extension $E^{p,q}$ of $E_U^{p,q}$ to $Y$ coincides with the Mumford extension of $E_U^{p,q}$ in Theorem~\ref{mumford1},~a).
\end{lemma}
\begin{proof}
Let $e_1,\ldots,e_n$ be a local basis for the canonical extension $E^{p,q}$. Building up on \cite{Sch73}, \cite[Theorem 5.21]{CKS86} describes the growth of the Hodge metric near $S$. In particular $||e_i||$ is bounded from above by the logarithm of the coordinate functions $z_1,\ldots ,z_k$. The Deligne extension is uniquely determined by the condition 
of logarithmic growth for the Hodge metric near $S$.
\par 
Since the metric $h$ coincides with the Hodge metric and since Mumford's notion `good' 
implies that $h(e_i)$ is bounded from above by the logarithm of the coordinate functions 
$z_1,\ldots ,z_k$, unicity implies that the Deligne extension and the Mumford extension
coincide.
\end{proof}
\begin{lemma} \label{hom_is_stable}
Suppose that the vector bundle $E$ on $Y$ is Mumford's extension
of an irreducible homogeneous vector bundle $E|_U$. Then $E$ is stable  with respect to the polarization
$\omega_Y(S)$. 
\end{lemma}
\par
\begin{proof}
By definition of Mumford's extension (\cite[Theorem 3.1]{Mu77}),
$E$ carries a metric $h$ coming from the $G$-invariant metric, again denoted by $h$, on the pull back $\tilde E$ of $E$ to $M$. As mentioned already, for a singular metric, good in the sense of Mumford, the curvature of the Chern connection $\nabla_h$ of $h$ represents the first chern class of $E$. 

We claim that the restriction of $\nabla_h$ to $U$ is a  Hermitian Yang-Mills connection
with respect to the  K\"ahler-Einstein  metric $g$ on $\Omega^1_U$.  In fact,
the pull back vector bundle $\tilde E$ on $M$ is an irreducible homogeneous vector bundle.
\par
So our claim says that this $G$-invariant metric $h$ on $\tilde E$ is Hermitian-Yang-Mills with respect to the $G$-invariant (K\"ahler-Einstein) metric  $g$ on $\Omega^1_M$ with the argument adapted from the proof of \cite[Theorem 3.3 (1)]{Ko86}.  The $g-$trace of the curvature $\wedge_g(\Theta_h)$ of $h$  is a $G-$invariant
endomorphism on the vector bundle $\tilde E$, and 
$$
\wedge_g(\Theta_h)_0 := \wedge_g(\Theta_h)|_{\tilde E_0} 
$$  
is an $K-$invariant endomorphism
on the vector space $\tilde E_0.$  Since the maximal compact subgroup $K$  acts on $\tilde E_0$ irreducibly, $\wedge_g(\Theta_h)_0 $  must be a scalar multiple of the identity
on $\tilde E_0.$  The facts that $G$ operates on $M$ transitively and that the induced action of $G$  on $\tilde E$  commutes with $\wedge_g(\Theta_h)$ imply that
$\wedge_g(\Theta_h)$  is  a constant  scalar multiple  of the identity endomorphism.
So, $h$ is a Hermitian-Yang-Mills metric  with respect to the $G$-invariant (K\"ahler-Einstein) metric  $g$ on $\Omega^1_M$. Here we regard $\Omega^1_M$ as an irreducible  homogeneous vector bundle. On the quotient $U$ we obtain the Hermitian-Yang-Mills metric $h$ on $E|_U$ with respect to the K\"ahler-Einstein metric $g$ on $\Omega^1_U.$ 
\par
Suppose that $F \subset E$ is a subbundle and let $s_U$ 
be the $C^\infty$ orthogonal splitting over $U$. 
By Theorem 5.20 in \cite{Kol85} the curvature of the Chern connection to $h|_F$
represents the $c_1(F)$. The Chern-Weil formula implies
$$ R(\nabla_{(h|_F)}) = R(\nabla_h)|_F + s_u \wedge s_u^*. $$
The Hermitian Yang-Mills property of $h$ yields $\mu(F) \leq \mu(E)$ and equality holds
if and only if $s_U$ is holomorphic. 

If the equality holds, the pullback of $s_U$ to $M$ gives an orthogonal
splitting of Hermitian vector bundles
$$\pi^* E|_U \cong  \pi^* F|_U \oplus \pi^* F^\bot|_U.$$
By Proposition~2 on p.~198 of \cite{Mk87} this contradicts the
irreducibility of $E|_U$. Thus $E$ is $\mu$-stable. 
\end{proof}
\par
\begin{lemma} \label{tensorofhomog}
Suppose that $E_i$ are vector bundles on $Y$, that are Mumford's extensions
of irreducible homogeneous vector bundles $E_i|_U$. Then
$E_1 \otimes E_2$ is $\mu$-polystable.
\end{lemma}
\par
\begin{proof}
Let $\rho_i$ be the representation corresponding to $E_i$.
Since the $E_i$ are $\mu$-stable, $E_1 \otimes E_2$ is $\mu$-semistable.
Repeating the calculation of the curvature of the Chern connection
from the previous Lemma, the existence of a subbundle of $E_1 \otimes E_2$ of the same slope
as $E_1\otimes E_2$ implies that the respresentation 
$\rho_1 \otimes \rho_2$ corresponding to $E_1 \otimes E_2$
is not irreducible. Since $K$ is reductive, $\rho_1 \otimes \rho_2$
decomposes as a direct sum of irreducible representations. Each
of them defines a $\mu$-stable bundle, again by the previous Lemma, 
and equality of slopes follows from semistability. 
\end{proof}
\par
Before proving the first part of Proposition~\ref{shimuraprop} for the Mumford compactification
$Y$, let us show that the Arakelov equality is independent of the compactification
$Y$ and compatible with replacing $U$ by an \'etale covering $U'$.
\begin{lemma}\label{AEindep}
Let $\delta:U' \to U$ be a finite \'etale morphism and let $Y,S$ and $Y',S'$ be two 
compactifications of $U$ and $U'$, both satisfying 
the Assumptions~\ref{compass-pos}. Let $\mu$ denote the slope on $Y$  
with respect to $\omega_{Y}(S)$ and $\mu'$ the one on $Y'$ with respect to $\omega_{Y'}(S')$. Given a complex polarized variation of Hodge structures $\V$ on $U$ with unipotent monodromy at infinity, let $(E,\theta)$ and $(E',\theta')$ be the logarithmic Higgs bundles 
of $\V$ and $\V'=\delta^*\V$. Then
\begin{enumerate}
\item[i.] $\deg(\delta)\cdot\mu(E^{1-q,q})=\mu'(E'^{1-q,q})$, for $q=0,1$.
\item[ii.] $\deg(\delta)\cdot\mu(\Omega_{Y}^1(\log S))=\mu'(\Omega_{Y'}^1(\log S'))$.
\item[iii.] In particular the Arakelov equality on $Y$ implies the one on $Y'$.
\end{enumerate}
\end{lemma}
\begin{proof}
Choose a compactification $\bar{Y}$ of $U'$, with $\bar{S}=\bar{Y}\setminus U$ a normal crossing divisor,  such that the inclusion $U'\to Y'$ extends to a birational morphism $\bar{\sigma}: \bar{Y} \to Y'$ and such that the finite morphism $\delta: U'\to U$ extends to a generically finite morphism 
$\bar{\delta}: \bar{Y} \to Y$.

The Assumptions~\ref{compass-pos} implies that the sheaves $\sL=\bar{\delta}^*\omega_{Y}(S)$ and
$\sL'=\bar{\sigma}^*\omega_{Y'}(S')$ are both nef and big. 
Moreover for some effective exceptional divisors $E$ and $E'$ one has 
$$
\omega_{\bar{Y}}(\bar{S})=\sL\otimes\sO_{\bar{Y}}(E)=\sL'\otimes\sO_{\bar{Y}}(E').
$$
Since $\sL$ is big, one can find an effective divisor $F$ on $\bar{Y}$ and some $\nu$ sufficiently large, such that the sheaf $\sL^{\nu}\otimes \sO_{\bar{Y}}(-F)$ is ample.
Replacing $F$ and $\nu$ by some multiple, one can as well assume that $\sL^{\nu}\otimes \sO_{\bar{Y}}(-F)$ and $\sL^\nu\otimes \sO_{\bar{Y}}(-F)\otimes \omega_{\bar{Y}}(\bar{S})^{-1}$ are very ample. Replacing $F$ and $\nu$ again by some multiple, one may even assume that the sheaves $\sL^\beta\otimes \sO_{\bar{Y}}(-F)$ are generated by global sections for all $\beta \geq \nu$ (e.g.\ \cite[Corollary 2.36]{Vi95}). Choosing $\nu$ large enough and a suitable effective divisor $F'$, the same holds true $\sL'^\beta\otimes \sO_{\bar{Y}}(-F')$.

Since for all $\beta\geq 0 $ 
$$
H^0(\bar{Y},\sL^\beta)=H^0(\bar{Y},\omega_{\bar{Y}}(\bar{S})^\beta)=H^0(\bar{Y},\sL'^\beta),
$$
this implies that $\sL=\sL'$. Let us write $\bar{\mu}$ for the slope with respect to the
invertible sheaf $\sL=\sL'$ on $\bar{Y}$.

The Deligne extension of $\V\otimes_C\sO_U$ is compatible with pullbacks. This implies that $\bar{\delta}^*E^{1-q,q}=\bar{\sigma}^*E'^{1-q,q}$, and by the projection formula
\begin{gather*}
\deg(\delta)\cdot\mu(E^{1-q,q})= \bar{\mu}(\bar{\delta}^*E^{1-q,q})=\bar{\mu}(\bar{\sigma}^*E'^{1-q,q})=\mu'(E'^{1-q,q})\\
\mbox{and \ \ }\dim(U) \cdot\deg(\delta)\cdot \mu(\Omega^1_{Y}(\log S))=\deg(\delta)\cdot\mu(\omega_{Y}(S))=
\bar{\mu}(\sL)= \hspace*{2cm}\\ \hspace*{3cm}
\bar{\mu}(\sL')=\mu'(\omega_{Y'}(S'))= \dim(U)\cdot \mu'(\Omega^1_{Y'}(\log S')).
\end{gather*}
Of course, iii) follows from i) and ii).
\end{proof}
\par
We now prove Proposition~\ref{shimuraprop} except for the statement
v). The latter will be shown at the end of Section~\ref{stablengthsplit}, by applying
Addendum~\ref{characterizationadd},~III.
\par
\begin{proof}[Proof of Proposition~\ref{shimuraprop}, part i)--iv) for Mumford's compactification] \ \\
Those properties can be verified over some \'etale covering of $U$. 
So one may assume that that $U\to \sA_g$ factors through a fine moduli scheme, 
hence by by Theorem~\ref{Moonen_product} through $\sX^{\rm MT}=\sX_1\times \sX_2$ with image 
of the form $\sX_1 \times \{b\}$. Let $\T$ denote the irreducible direct factor of the uniformizing $\C$-variation of Hodge structures on the Shimura variety $\sX_1\times \sX_2$, with $\V \subset \T|_{\sX_1 \times \{b\}}$. 

By Schur's Lemma and \cite[Prop. 1.13]{De87} a polarized variation of Hodge structures
on $\sX_1\times\sX_2$ is a direct sum of exterior products of complex polarized variations 
of Hodge structures (see \cite[Prop. 3.3]{VZ05}). 
The irreducibility of $\T$ implies that
$\T={\rm pr}_1^*\V_1 \otimes {\rm pr}_2^*\V_2$ for suitable irreducible $\C$-variations of Hodge structures 
$\V_i$ on $\sX_i$. Remark that $\V$, $\T$ and $\V_1$ are concentrated in bidegrees $(1,0)$ and $(0,1)$.
Hence $\V_2$ has weight zero and is concentrated in bidegree $(0,0)$. 
Since ${\rm pr}_2^*\V_2|_{\sX_1 \times \{b\}}$ is a trivial Hodge structure, independent 
of the point $b$, the local system $\T|_{\sX_1 \times \{b\}}$ is just a direct sum of several copies of $\V_1$. This remains true if one replaces $b$ by a different point $a \in \sX_2$.
The irreducibility of $\V$ implies that $\V\cong\V_1$, so passing from $b$ to $a$ one does not change the 
irreducible components of the complex variation of Hodge structures.

So we may suppose without loss of generality 
that $U$ is a Shimura variety of Hodge type given by the datum $(G,M)$.
\par
Our first aim is to exhibit $E^{1,0}$ and $E^{0,1}$ as homogeneous
vector bundles. Let  $\tau:G \to \CSp$ be the map given by the
property `of Hodge type'. Choose a base point on the symmetric domain
$M$ and its image on $M':=M(Q)$. There are maximal compact subgroups $K$ of 
$G^\der$ and $K'\cong U(g)$ of $\Sp$  such that 
$U \to \sA_g$ is uniformized by the map $M = G^\der/K \to \Sp/K'=:M'$. 
Let $\pi_U: D \to U$ and $\pi_{\sA_g}: D' \to \sA_g$ be the natural
quotients modulo arithmetic subgroups. The choice of the base 
point in $M'$ is equivalent to the choice 
of a $Q$-symplectic basis $\{a_i,b_i\}$ of $V$ such that
we have $h(i)(a_i) = b_i$ and $h(b_i)=-a_i$ by~\ref{defsympl_h}. 
\par
Since the $(1,0)$- and $(0,1)$-parts of $\pi_{\sA_g}^*(R^1 f_* \C_A)$ are the
$i$ resp.\ $-i$-eigenspaces of $h(i)$, they are homogeneous bundles.
Moreover, they are given by the representations $\rho_\can$ and 
$\overline{\rho_\can}$, where $\rho_\can: U(g) \to \GL(g)$
is the standard representation. The $(1,0)$- and $(0,1)$-parts of 
$\pi_U^*(R^1 f_* \C_A)$ are consequently homogeneous
bundles too, given by the representation $\rho_\can\circ \tau|_K$
and $\overline{\rho_\can}\circ \tau|_K $.
\par
Next, we link two notions of irreducibility.
Since $\pi_U$ is the quotient map by an arithmetic group $\Gamma \subset 
G(\Q)$, whose image in $G^\ad$ is Zariski-dense, $\C$-irreducible summands
of $R^1 f_* \C_A$ are in bijection with $\C$ irreducible summands of the
representation 
$$\widetilde{\tau}: \widetilde{G^\ad} \>>> G \>>> \CSp. $$
Here  $\widetilde{G^\ad} \to G^\ad$ is the universal covering
and the map to $\widetilde{G^\ad} \to G$ is induced by the 
canonical splitting of $\Lie(G)$ into its abelian and its semisimple part.
We determine these $\C$ irreducible summands,
 following \cite[\S 2.3.7 (a)]{De79}, 
see also \cite{Sa65} or \cite{Sa80}.
\par  
By \cite[\S 2.3.4]{De79} the simple components of $G_\R$ are
are absolutely simple. Write 
$$G_\R^\ad = \bigtimes_{i \in I} G_i$$
and partition the index set $I = I_c \cup I_\nc$ according 
to whether $G_i$ is compact or not.
By \cite[\S 1.3.8 (a) and \S 2.3.7]{De79} the irreducible direct factors of
$V_\C$ are of the form $\otimes_{t \in T} W_t$ for some $T \subset I$,
where $W_t$ is an irreducible representation of $\widetilde{G_{i,\R}}$.
Moreover, the condition (SV1) forces $T \cap I_\nc$ to contain at 
most one element, see \cite[Lemma~1.3.7]{De79} This shows i).
\par
If  $T \cap I_\nc = \emptyset$, then $\V$ is unitary. We thus restrict to
the other case from now on. Then the condition `Shimura variety'
imposes the restrictions to the representation of the non-compact group 
as in the hypothesis of Lemma~\ref{DeligneSatake}, stated below. From this
lemma we deduce that in each case the representation of
$K \subset G^{\ad}$ is irreducible.
\par
Now we know by Lemma~\ref{hom_is_stable} that for each
irreducible summand $\V$ of $R^1 f_* \C_A$, both $E^{1,0}$
and $E^{0,1}$ are $\mu$-stable.
By Lemma~\ref{tensorofhomog}, the bundle $\Hom(E^{1,0}, E^{0,1})$ is 
$\mu$-polystable with the $\mu$-stable summands given as homogeneous bundles
by the irreducible summands of the representation $\rho \otimes \rho^\vee$,
where $\rho=\rho_\can \circ \tau$. This proves iii) and iv).
Since $M \to M'$ is induced by a group homomorphism and hence totally 
geodesic, the tangent map
$$T_M \>>> T_M'|_M = \Hom(E^{1,0}, E^{0,1})$$
is onto a direct summand. Since it is a map between homogeneous bundles,
the direct summand corresponds to an irreducible summand of the 
representation $\rho \otimes \rho^\vee$. Consequently, the map 
$$\overline{(T_U)} \>>> \overline{\Hom(E^{1,0}, E^{0,1})|_U}$$
between the Mumford extensions is an injection onto a $\mu$-stable summand.
Since the Mumford extension of $T_U$ is $T_Y(-\log S)$ and 
the Mumford extension of $E^{p,q}$ is the Deligne extension, we
obtain
$$\mu(T_Y(-\log S)) = \mu(E^{1,0}) - \mu(E^{0,1}),$$
i.e.\ the Arakelov equality, stated as ii).
\end{proof}
We keep the notations of the preceding proof, that will
be completed with the following lemma. We follow \cite{De79} 
and define a cocharacter $\chi: \G_m \to (G_i)_\C$ induced by 
$h:{\rm Res}_{\C/\R}\G_m  \to G_\R$ in the following way. 
Fix an isomorphism 
$$
({\rm Res}_{\C/\R}\G_m)_\C \cong \G_m \times \G_m
$$
such that the inclusion 
$$
({\rm Res}_{\C/\R}\G_m)(\R) \to ({\rm Res}_{\C/\R}\G_m)(\C)
$$
is given by $z \mapsto (z,\overline{z})$. Let 
$i: \G_m \to \G_m \times \G_m$ be the inclusion given by the identity
in the second argument. Then $\chi := h_\C \circ i$.
\par
Given $\chi$, we let $\tilde{\chi}$ be the inductive 
system of fractional lifts of $\chi$ to $\widetilde{G_i}$ 
(\cite[\S 1.3.4]{De79}).
\par
\begin{lemma} \label{DeligneSatake}
Let $\tau_{i,t}: \widetilde{G_i} \to \GL(W_t)$ be an irreducible
representation whose highest weight $\alpha$ is a fundamental
weight and such that 
\begin{equation} \label{invopp}
\langle \widetilde{\chi}, \alpha+\iota(\alpha) \rangle =1,
\end{equation}
where $\iota$ is the opposition involution. Then $W_t$ is the sum of
two non-empty weight spaces, denoted by $W_t^{1,0}$ and $W_t^{0,1}$.
Both weight spaces are irreducible representations of the
maximal compact subgroup $K_i$ of $G_i$.
\end{lemma}
\par
\begin{proof} The equivalence of the condition~\eqref{invopp}
and the decomposition into two weight spaces is in (\cite[\S 1.3.8]{De79}).
The possible solutions to~\ref{invopp} are listed on \cite[p.~461]{Sa65}.
We distinguish the cases according to the Dynkin diagram of $G_i$.
We use that the cocharacter $\widetilde{\chi}$ satisfying
\eqref{invopp} determines a special node in the Dynkin diagram
(\cite[\S 1.2.5]{De79}).
\par
{\em Type $a_n$:} In this case $G_i=\SU(p,q)$ with $p+q=n-1$, depending
on the signature of the bilinear form induced by the Cartan involution
$\ad(h(i))$. We may assume $p \geq q$. The maximal compact 
subgroup is 
$$
K_i = S(U(p) \times U(q)).
$$
If $q>1$ only the standard representation satisfies~\ref{invopp}.
The weight spaces  $W_t^{1,0}$ and $W_t^{0,1}$ carry the 
standard representation of $\SU(p)$ and $\SU(q)$ respectively
and are hence irreducible.
\par
If $q=1$ all $j$-th wedge product representations for $j=1,\ldots,n-1$
satisfy~\ref{invopp}. The weight spaces  $W_t^{1,0}$ (resp.\ $W_t^{0,1}$) 
carry the $j$-th (resp.\ $j-1$-st) exterior power representation of $\SU(p)$,
which is also irreducible.
\par
{\em Type $b_n$:} In this case is $G_i=\SO(2,2n-1)$ (type $IV_{2n-1}$ in 
\cite{Sa80})) and the only representation that satisfies~\ref{invopp}
is the spin representation of the double cover ${\rm Spin}(2,2n-1) \to G_i$.
The maximal compact subgroup is 
$$
K_i \cong \SO(2n-1,\R) \times \SO(2,\R).
$$
We claim that one weight space carries the tensor product of
the spin representation of $\SO(2n-1)$ and one of the natural 
representations $\SO(2,\R) \to U(1)$ while the other weight
space carries the tensor product of the spin representation
and the complex conjugate representation of $\SO(2,\R)$.
In both cases the representations are well known to be irreducible.
\par
In order to prove the claim we write down the spin representation
explicitly and exhibit its weight spaces. We follow the notations 
of \cite[\S 3.5]{Sa65}.
Let $G_i$ be the group of transformations of $V_\R$ preserving
a bilinear form $S$ of signature $(2n-1,2)$. Let 
$\{e_1,\ldots,e_{2n-1}\}$ (resp.\ $\{e_{2n},e_{2n+1}\}$) be an orthonormal
bases of $V^+$ (resp. $V^-$), the subspaces where the form is
positive (resp.\ negative) definite.
We let $f_j = (e_{2j-1}+ i e_{2j})/2$ for $j=1,\ldots,n-1$
and $f_n = (e_{2n} + i e_{2n+1})$.
Denote by $W$ the complex vector space generated by the
$f_j$. The exterior algebra $E=\Lambda(W)$ embeds into the Clifford
algebra of $C(V,S)$. For an ordered subset $\sJ  = \{i_1,\ldots,i_a \} 
\subset N:=\{1,\ldots,n\}$ we consider the elements $f_\sJ = f_{i_1} \cdots
f_{i_a}$ and their complex conjugates in the Clifford algebra.
We identify $E$ with the left ideal $E \cdot \overline{f_N}$ and
obtain a representation of  ${\rm Spin}(2,2n-1)$ on $E$.
\par
We may choose in 
$$ \Lie(G_i) = \left\{ 
\left( \begin{matrix}
X_1 & X_{12} \\ X_{12}^T & X_2
\end{matrix} \right) ; \,X_1,X_{12},X_2 \,\text{real}, X_1,X_2\,
\text{skew symmetric}
\right\}.$$
a maximal abelian subalgebra,
$$ \mathfrak{h} = \left\{ 
{\rm diag}\left(
\left( \begin{matrix}
0 & -\xi_1 \\ \xi_1 & 0
\end{matrix} \right),\ldots,
\left( \begin{matrix}
0 & -\xi_{n-1} \\ \xi_{n-1} & 0
\end{matrix} \right),
0,
\left( \begin{matrix}
0 & -\xi_{n} \\ \xi_{n} & 0
\end{matrix} \right), \xi_i \in \R
\right)\right \}.$$ 
Then by the calculation in \cite[p.~455]{Sa65}) 
the $f_\sJ$ are eigenvectors with 
corresponding weight $\frac{i}{2} (\sum_{i\not\in \sJ} \xi_i - 
\sum_{i \in \sJ} \xi_i)$. The map $\chi$ corresponding to the
special node is generated by the element $H_0 \in \Lie(G_i)$
with $X_1=0$, $X_{12}=0$ and 
$X_2=\left( \begin{matrix}
0 & -1 \\ 1 & 0
\end{matrix} \right)$. We deduce that 
the weight spaces $W_i^{1,0}$ (resp.\ $W_i^{0,1}$) are generated by 
the $f_\sJ$ with $n \not \in \sJ$
(resp.\ by the $f_\sJ$ with $n \in \sJ$).
\par
From this we first read off that $\SO(2,\R)$ acts on the
weight spaces as claimed. Fix the root system
$$\{i(\xi_1-\xi_2),\ldots ,i(\xi_{2n-2} - \xi_{2n-1}), 
i\xi_{2n-1}\}$$ of ${\rm so}(2n-1)$. Consider 
$W_i^{1,0}$ as a representation of $\widetilde{\SO(2n-1)}$
of dimension $2^{n-1}$. A vector of highest weight 
is $f_{N \setminus \{n\}}$
with weight $i/2 \sum_{i=1}^{n-1} \xi_i$. Consequently, the 
representation contains a spin representation of 
${\rm Spin}(2n-1) \to \SO(2n-1)$. For dimension reasons the representation
is irreducible. The same argument applies to $W_i^{0,1}$.
\par
{\em Type $c_n$:} In this case $G_i=\Sp(n)$, and as in the beginning
of the proof of Proposition~\ref{shimuraprop} above, 
the weight spaces carry the standard representation of $U(n)$
and its complex conjugate. Thus, they are irreducible. 
\par
{\em Type $d_n$:} This case splits into two subcases according to the
$\chi$ or equivalently according to the position of the corresponding special
node in the Dynkin diagram.
\par
{\em Special node at the `fork' end}. In this case 
$$G_i = \SU^-(n,\BH) \cong \SU(n,n) \cap \SO(2n,\C) \subset \Sl(2n,\C)$$ 
where $\BH$ denotes the Hamiltonians. In this matrix representation
the weight spaces are given by the $n$ first (resp.\  last) column vectors.
The maximal compact subgroup $K_i \cong U(n)$ sits in $G_i$ via
$$ A+iB \mapsto \left(\begin{array}{cc} A & B \\ 
-B & A\\ \end{array}\right) $$
Consequently, both weight spaces are $n$-dimensional and carry the 
irreducible standard representation of $U(n)$. 
\par
{\em Special node at the opposite end}. This is completely
similar to the case $b_n$ replacing `spin' by `half spin' representations
throughout.
\par
{\em Exceptional Lie algebras} do not admit any solution to~\ref{invopp}.
\end{proof}

\section{Slopes and filtrations of coherent sheaves}\label{fil}

We will need small twists of the slope $\mu(\sF)$ defined with respect to the
nef and big invertible sheaf $\omega_Y(S)$ in~\ref{slope}. So we will decompose the slope in a linear combination of different slopes and we will deform the coefficients a little bit. In particular, as in \cite{La}, we will compare the Harder-Narasimhan filtrations for small twists of slopes.

On the non-singular projective variety $Y$ of dimension $n$ consider $n-1$-tuples 
of $\R$-divisors
$$
\underline{D}^{(\iota)}=(D^{(\iota)}_1,\ldots,D^{(\iota)}_{n-1}),
$$
for $\iota=1,\ldots,m$. The collection of those divisors will be denoted by $\underline{D}^{(\bullet)}$.
Given two such tuples  $\underline{D}^{(\bullet)}$ and $\underline{D}'^{(\bullet)}$ we define the sum componentwise, hence
$$
\underline{D}^{(\bullet)}+\underline{D}'^{(\bullet)}= \big[(D^{(\iota)}_1+D'^{(\iota)}_1,
\ldots,D^{(\iota)}_{n-1}+D'^{(\iota)}_{n-1}); \ \iota = 1,\ldots,m\big].
$$
\begin{definition} We call $\underline{D}^{(\bullet)}$ a {\em semi-polarization}
if the $\R$-divisors $D^{(\iota)}_j$ are nef for $\iota=1,\ldots,m$ and for
$j=1,\ldots,n-1$ and if the intersection cycle
$$(\underline{D}^{(\iota)})^{n-1}:=D^{(\iota)}_1.\cdots.D^{(\iota)}_{n-1}$$ 
is not numerically trivial for $\iota=1,\ldots,m$.
\end{definition}
For a coherent torsion free sheaf $\sF$ on $Y$ and for each $\iota\in \{1,\ldots,m\}$
one defines the slope 
$$
\mu_{\underline{D}^{(\iota)}} (\sF)= \frac{\ch_1(\sF).(\underline{D}^{(\iota)})^{n-1}}{\rk(\sF)},
$$
and adding up
\begin{equation}\label{addslope}
\mu_{\underline{D}^{(\bullet)}} (\sF)=
\mu_{[\underline{D}^{(1)},\ldots, \underline{D}^{(m)}]} (\sF)=\sum_{\iota=1}^m \mu_{\underline{D}^{(\iota)}} (\sF)=
\sum_{\iota=1}^m \frac{\ch_1(\sF).(\underline{D}^{(\iota)})^{n-1}}{\rk(\sF)}.
\end{equation}
In the sequel we will assume that $\underline{D}^{(\bullet)}$ is a semi-polarization, and 
we fix a torsion free coherent sheaf $\sF$ on $Y$. If there is no ambiguity, we write
$\mu'$ in this Section instead of $\mu_{\underline{D}^{(1)},\ldots, \underline{D}^{(m)}}$,
and we reserve the notion $\mu$ for the special case where the slope is taken with respect to $\omega_Y(S)$. 

Given an exact sequence of torsion free coherent sheaves
$$
0\>>> \sF' \>>> \sF \>>> \sF'' \>>> 0,
$$
an easy calculation shows that
\begin{equation}\label{additiv}
\mu'(\sF)= \frac{\rk(\sF')}{\rk(\sF)}\mu'(\sF') + \frac{\rk(\sF'')}{\rk(\sF)}\mu'(\sF'').
\end{equation}
In order to define `stability' for locally free or torsion free coherent sheaves one has to take care of boundary divisors of slope zero, i.e.\ of prime divisors $D$ with $\mu'(\sO_Y(D))=0$. Since the divisors 
${D}^{(\iota)}_j$ are nef, this is equivalent to the condition $D.(\underline{D}^{(\iota)})^{n-1}=0$, for $\iota=1,\ldots,m$.
\begin{definition}\label{fil.2} 
Keeping the notations introduced above, let $\sF$ and $\sG$ be two coherent torsion free sheaves on $Y$. 
\begin{enumerate}
\item[a.] A subsheaf $\sG$ of $\sF$ is {\em $\mu'$-equivalent} to $\sF$, if $\sF/\sG$ is a torsion sheaf and if $\ch_1(\sF)-\ch_1(\sG)$ is the class of an effective divisor $D$ with $\mu'(\sO_Y(D))=0$, or equivalently with $D.(\underline{D}^{(\iota)})^{n-1}=0$, for $\iota=1,\ldots,m$. We call  {\em $\mu'$-equivalence} the equivalence relation
on coherent sheaves generated by $\mu$-equivalent inclusions. 
\item[b.] A morphism $\sG \to \sF$ is {\em surjective up to $\mu'$-equivalence}, if its image
is $\mu'$-equivalent to $\sF$.
\item[c.] $\sG\subset \sF$ is {\em saturated}, if $\sF/\sG$ is 
torsion free.
\item[d.] $\sF$ is {\em $\mu'$-stable}, if $\mu'(\sG) < \mu'(\sF)$
for all subsheaves $\sG$ of $\sF$ with $\rk(\sG)<\rk(\sF)$.
\item[e.] $\sF$ is {\em $\mu'$-semistable}, if $\mu'(\sG) \leq \mu'(\sF)$ for all subsheaves $\sG$ of $\sF$.
\item[f.] $\sF$ is {\em $\mu'$-polystable} if it is the direct sum of $\mu'$-stable sheaves of the same slope. 
\item[g.] A saturated subsheaf $\sG$ of $\sF$ is called a {\em maximal destabilizing subsheaf}, 
if for all subsheaves $\sE$ of $\sF$ one has $\mu'(\sE)\leq \mu'(\sG)$ and if the equality implies that $\sE\subset \sG$.
\end{enumerate}
\end{definition}
We will give a nicer description of the relation `$\mu$-equivalence' 
in a special case at the beginning of Section~~\ref{spl}.
\begin{lemma}\label{fil.4} \
\begin{enumerate}
\item[1.] If $\sF$ is $\mu'$-stable and if $\sG\subset \sF$ is a subsheaf
with $\mu'(\sG)=\mu'(\sF)$ then $\sF$ and $\sG$ are $\mu'$-equivalent.
\item[2.] A $\mu'$-polystable sheaf $\sF$ is $\mu'$-semistable.
\item[3.] In particular, if $\sH$ is invertible, then $\bigoplus \sH$ is $\mu'$-semistable.
\end{enumerate}
\end{lemma}
\begin{proof}
If $\sG$ is a subsheaf of $\sF$ with $\rk(\sG)=\rk(\sF)$ then $\ch_1(\sF)-\ch_1(\sG)$ is an effective divisor $D$. Since all the $D_j^{(\iota)}$ are nef, one finds $D.(\underline{D}^{(\iota)})^{n-1}\geq 0$ and hence 
$\mu'(\sG) \leq \mu'(\sF)$. This implies 2) in case that $\sF$ is $\mu'$-stable.

For $\mu'$-polystable sheaves 2) follows by induction on the number of direct factors, and
3) is an example for the statement in 2).

If $\sF$ is $\mu'$-stable and $\mu'(\sG)=\mu'(\sF)$, then by definition
$\rk(\sF)=\rk(\sG)$, hence $D.(\underline{D}^{(\iota)})^{n-1}= 0$ as claimed in 1).
\end{proof}
Later the divisors $D^{(\iota)}_i$ will correspond to the determinant of the $\mu$-polystable direct factors $\Omega_j$ of $\Omega^1_Y(\log S)$ in the decomposition~\ref{eqint.1}, each one occurring as often as the rank of $\Omega_j$, except the one corresponding to the upper index $\iota$. 
For one $\iota$ we will multiply in~\ref{addslope} $\mu_{\underline{D}^{(\iota)}}$
by a factor $1+\epsilon$. 

We consider in this section a more general and more flexible set-up than needed in the sequel, hoping that it might be of use in a different context. We choose a second tuple
$$
\underline{H}^{(\iota)}=(H^{(\iota)}_1,\ldots,H^{(\iota)}_{n-1})
$$
of nef $\R$-divisors, for $\iota=1,\ldots,m$, and the polynomial 
$$
\mu'_t(\sF)=
\mu_{\underline{D}^{(\bullet)}+t\cdot \underline{H}^{(\bullet)}} (\sF)=
\sum_{\iota=1}^m \frac{\ch_1(\sF).(\underline{D}^{(\iota)}+t\cdot \underline{H}^{(\iota)})^{n-1}}{\rk(\sF)}.
$$
Of course one has $\mu'_0(\sF)=\mu'(\sF)$.
The cycle $(\underline{D}^{(\iota)}+t\cdot \underline{H}^{(\iota)})^{n-1}$ can be written as  
$$
D^{(\iota)}_{1}.\cdots . D^{(\iota)}_{n-1} + 
\sum_{I \in \sI} t^{n-|I|-1}\cdot D^{(\iota)}_{i_1}.\cdots . D^{(\iota)}_{i_{|I|}} . H^{(\iota)}_{j_1}.\cdots . H^{(\iota)}_{j_{n-1-|I|}}
$$
where the sum is taken over the set $\sI$ of ordered subsets 
$$
I=\{i_1,\ldots,i_{|I|}\} \mbox{ \ \ of \ \ }\{1,\ldots,n-1\}
$$  
of cardinality $|I| < n-1$, and where $\{j_1,\ldots,j_{j_{n-1-|I|}}\}$ is 
the complement of $I$ in $\{1,\ldots,n-1\}$, again as an ordered set.
For a coherent sheaf $\sG$ one has
\begin{gather}\label{muI1}
\mu'_t(\sF)-\mu'_t(\sG)=\mu'(\sF)-\mu'(\sG)+
\sum_{\sI} t^{n-|I|-1}\cdot 
(\mu'^{I}(\sF)-\mu'^{I}(\sG)),\\
\mbox{with \ \ } \label{muI2}
\mu'^{I}(\sG)=\sum_{\iota=1}^m \frac{\ch_1(\sG).D^{(\iota)}_{i_1}.\cdots . D^{(\iota)}_{i_{|I|}} . H^{(\iota)}_{j_1}.\cdots . H^{(\iota)}_{j_{n-1-|I|}}}{\rk(\sG)}.
\end{gather}
\begin{lemma}\label{fil.5} For a coherent sheaf $\sF$ of rank $r$ consider the sets
$$
{\rm S} = \{ \mu'(\sG); \ \sG \subset \sF\} \subset \R \mbox{ \ \ and \ \ }
\sS=\big\{\mu'_t(\sG)=\sum_{\nu=0}^{n-1}
a_\nu\cdot t^\nu; \ \sG\subset \sF\big\} \subset \R[t].
$$ 
Then 
\begin{enumerate}
\item[i.] the set ${\rm S}$ is discrete and bounded from above.
\item[ii.] There exists some $\epsilon_0>0$ and some `maximal' element
$G(t)\in \sS$, such that for all $F(t)\in \sS$ with $F(t)\neq G(t)$ one has $G(\epsilon)> F(\epsilon)$ for $0 < \epsilon\leq \epsilon_0$.
\end{enumerate}
\end{lemma}
\begin{proof}
Let ${\rm S}'$ be the set of all coefficients occurring in $F(t)\in \sS$.
We will first show, that the set ${\rm S}'$ is discrete and bounded from above.
Since ${\rm S}\subset {\rm S}'$, this implies i).
 
For $\sH$ invertible and sufficiently ample $\sF^\vee \otimes \sH$ is generated by global sections. Hence
$\sF$ is embedded in $\bigoplus \sH$.
Then under the projection to suitable factors, any subsheaf $\sG\subset \sF$ of rank $r'$ is isomorphic to a subsheaf of $\bigoplus^{r'} \sH$ and $\ch_1(\sG)= r\cdot \ch_1(\sH) - D$ for some effective divisor $D$.

Since the divisors ${D}^{(\iota)}_j$ and $H^{(\iota)}_j$ are all nef, the intersection of the 1-dimensional cycles
$$
D^{(\iota)}_{i_1}.\cdots . D^{(\iota)}_{i_{|I|}} . H^{(\iota)}_{j_1}.\cdots . H^{(\iota)}_{j_{n-1-|I|}}
$$
in~\ref{muI2} with any divisor is a non-negative multiple of a fixed real number, So one may write
$$
\sum_{\iota=1}^m (\underline{D}^{(\iota)}+t\cdot \underline{H}^{(\iota)})^{n-1}=
\sum_{\nu=0}^{n-1} \left(\sum_\mu \alpha_{\nu,\mu} C_{\mu,\nu} \right) t^\nu
$$ 
for $\alpha_{\mu,\nu} \in \R$ and for linear combinations $C_{\mu,\nu}$ of curves with
$ D.C_{\mu,\nu} \geq 0$ for all effective divisors $D$. Then $-{\rm S}'$ is discrete, as a 
subset of the union of translates of finite many copies of 
$$
\bigcup_{\nu} \sum_\mu \alpha_{\mu,\nu} \cdot\N.
$$
Moreover ${\rm S}'$ it is bounded above by the maximal
coefficient $c$ of $\mu'_t(\sH)$.

On the set $\sS$ consider the lexicographical order. So
$\sum_{\nu=0}^{n-1}a_\nu\cdot t^\nu <  \sum_{\nu=0}^{n-1}b_\nu\cdot t^\nu$
if $a_\nu=b_\nu$ for $\nu <j$ and if $a_j < b_j$. Obviously $\sS$ contains a maximal
element $G(t)=\sum_{\nu=0}^{n-1}b_\nu\cdot t^\nu$ for this order. 

Choose $\epsilon_0 \in (0,1)$ to be a real number with
$$
\frac{1}{\sqrt{\epsilon_0}} \geq \sup_{c\in {\rm S}}
\left\{\sum_{\nu=j+1}^{n-1} (c-b_\nu)t^{\nu-j-1}; \ t\in [0,1], \ j=1,
\ldots,r-1\right\},
$$
and such that for $\nu=0,\ldots,r$ one has
$[b_\nu-\sqrt{\epsilon_0},b_\nu+\sqrt{\epsilon_0}]\cap {\rm S}'=\{b_\nu\}$.

Since $G(t)>F(t)$, for some $j$ and for $0 < \epsilon \leq \epsilon_0$ one finds
\begin{multline*}
G(\epsilon)-F(\epsilon) = \sum_{\nu=j}^{n-1} (b_\nu-a_\nu)\cdot 
\epsilon ^\nu \geq\\
\epsilon^j\cdot\big((b_j-a_j) + \epsilon \cdot \sum_{\nu=j+1}^{n-1} 
(b_\nu-c)\cdot \epsilon ^{\nu-j-1}\big) >
\epsilon^j\cdot (\sqrt{\epsilon_0} - \epsilon\cdot \frac{1}{\sqrt{\epsilon_0}})
\geq 0.
\end{multline*}
\end{proof}
We will consider next values of the polynomials $F(t)\in \sS$ for small $\epsilon \in \R_{\geq 0}$. 
\begin{definition}\label{fil.3} For $\epsilon \in \R_{\geq 0}$
consider a filtration $0=\sG_0 \subset \sG_1 \subset \cdots \subset \sG_\ell=\sF$
with $\sG_\alpha/\sG_{\alpha-1}$ torsion free and $\mu'_\epsilon$-semistable, for $\alpha=1,\ldots,\ell$, and with
\begin{equation}\label{eqfil.1}
\mu'_{\epsilon, {\rm max}}(\sF)=\mu'_\epsilon(\sG_1) \geq
\mu'_\epsilon(\sG_2/\sG_{1}) \geq \cdots \geq \mu'_\epsilon(\sG_\ell/\sG_{\ell-1})=\mu'_{\epsilon,{\rm min}}(\sF).
\end{equation}
The filtration is called a {\em $\mu'_\epsilon$-Harder-Narasimhan filtration} if the inequalities in (\ref{eqfil.1}) are all strict, 
and it is called a {\em weak $\mu'_\epsilon$-Jordan-H\"older filtration} if $\mu'_{\epsilon,{\rm max}}(\sF)=\mu'_{\epsilon,{\rm min}}(\sF)$. 
\end{definition}
\begin{lemma}\label{fil.7} Let $\sF$ be a coherent torsion free sheaf on $Y$. 
\begin{enumerate}
\item[a.] For all $\epsilon \geq 0$ there exists a Harder-Narasimhan filtration
$$
\sG_0=0\subset \sG_1 \subset \cdots \subset \sG_\ell=\sF
$$
of $\sF$ with respect to $\mu'_{\epsilon}$ and this filtration is unique. 
\item[b.] There exists some $\epsilon_0>0$ such that the filtration in a) is independent of 
$\epsilon$ for $\epsilon_0 \geq \epsilon > 0$.
\item[c.] If $\sF$ is $\mu'$-stable, then for some $\epsilon_0>0$ and for all $\epsilon_0 \geq \epsilon \geq 0$ the sheaf $\sF$ is $\mu'_\epsilon$-semistable.
\end{enumerate}
\end{lemma}
\begin{proof}
For $\epsilon >0$ we apply Lemma~\ref{fil.5},~ii). For the polynomial $G(t)$, given there, choose a subsheaf $\sG \subset \sF$ with $G(t)=\mu'_t(\sG)$, for all $t\in \R$. Moreover for $0<\epsilon \leq \epsilon_0$ the slope $\mu'_\epsilon(\sG)=G(\epsilon)$ is maximal among the possible slopes of subsheaves of $\sF$. This allows to assume that $\sG$ is saturated. 
If there are several subsheaves of $\sF$ with the same slope, we choose a saturated one of maximal rank. 

If for $\sE\subset \sF$ one has $\mu'_\epsilon(\sE)=\mu'_\epsilon(\sG)$, then by~\ref{additiv} the slope of $\sE\oplus \sG$ is $\mu'_\epsilon(\sG)$. The maximality of the slope of $\sG$ implies  $\mu'_\epsilon(\sE\cap \sG)\leq \mu'_\epsilon(\sG)$ and  $\mu'_\epsilon(\sE + \sG)\leq \mu'_\epsilon(\sG)$. By~\ref{additiv} this is only possible if $\mu'_\epsilon(\sE + \sG) = \mu'_\epsilon(\sG)$. Then the maximality of the rank of $\sG$ implies that $\rk(\sE + \sG) = \rk(\sG)$, and $\sE\subset \sG$. 

So $\sG$ is a maximal destabilizing subsheaf of $\sF$, and it is independent
of $\epsilon \in (0,\epsilon_0]$. The existence and uniqueness of a $\mu'_\epsilon$-Harder-Narasimhan filtration follows by induction on the rank. Here of course we have to lower
$\epsilon_0$ in each step. 

For $\epsilon=0$ the existence and uniqueness of the Harder-Narasimhan filtration follows by the same argument, replacing the reference to part ii) of Lemma~\ref{fil.5} by the one to part i). 

Assume now that $\sF$ is $\mu'$-stable and consider the Harder-Narasimhan filtration in a). Then
$$
\mu'(\sG_1)=\lim_{\epsilon \to 0} \mu'_{\epsilon}(\sG_1)
\geq \lim_{\epsilon \to 0} \mu'_{\epsilon}(\sF)=\mu'(\sF).
$$
By assumption, $\sF$ is stable, with respect to $\mu'$, hence
$\sG_1=\sF$, and $\ell=1$. 
\end{proof}
Although this will not be used in the sequel, let us state a strengthening of the last part of
Lemma~\ref{fil.7}.
\par
\begin{addendum}
For $\epsilon_0$ sufficiently small, the sheaf $\sF$ in part c) is $\mu'_\epsilon$-stable
for all $\epsilon_0 \geq \epsilon \geq 0$.
\end{addendum}
\begin{proof}
Part i) of Lemma~\ref{fil.5} and the $\mu'$-stability of $\sF$ imply that
$$
\gamma={\rm Inf}\{\mu'(\sF)-\mu'(\sG); \ \rk(\sG) < \rk(\sF)\}>0.
$$  
Let us return to the slopes $\mu'^I$ introduced in~\ref{muI1} and~\ref{muI2}. 
By part a) of Lemma~\ref{fil.7} there exists a Harder-Narasimhan filtration 
$$
\sG^{I}_0=0\subset \sG^{I}_1 \subset \cdots \subset \sG^{I}_{\ell_I}=\sF
$$
with respect to $\mu'^{I}$. In particular for $\sG\subset \sF$ one has 
$\mu'^{I}(\sG) \leq \mu'^{I}(\sG_1^{I})$.
  
Choose $\epsilon_0>0$ such that for $0 < \epsilon \leq \epsilon_0$, and for 
all $I\in \sI$ with $|I|<n-1$ one has
$$
\frac{1}{|\sI|+1} \cdot \gamma  \geq  \epsilon^{n-|I|-1}\cdot\big(\mu'^{I}(\sG_1^{I})-\mu'^{I}(\sF)\big)
$$
For a subsheaf $\sG \subset \sF$ of strictly smaller rank one finds 
$$
\mu'^{I}(\sF)-\mu'^{I}(\sG) \geq \mu'^{I}(\sF)-\mu'^{I}(\sG_1^{I}),
$$
and thereby
\begin{multline*}
\mu'_\epsilon(\sF)-\mu'_\epsilon(\sG) \geq
\gamma + \sum_{I\in \sI} \epsilon^{n-1-|I|}\cdot (\mu'^{I}(\sF)-\mu'^{I}(\sG))\geq\\
\gamma+\sum_{I\in \sI} \epsilon^{n-1-|I|}\cdot (\mu'^{I}(\sF)-\mu'^{I}(\sG_1^{I}))\geq \gamma-\frac{|\sI|}{|\sI|+1}\cdot\gamma >0.
\end{multline*}
\end{proof}
\begin{corollary}\label{fil.9}
Assume in Lemma~\ref{fil.7} that $\sF$ is $\mu'$-semistable. Then
there exists a weak $\mu'$-Jordan-H\"older filtration
$$
\sG_0=0\subset \sG_1 \subset \cdots \subset \sG_{\ell}=\sF
$$
and some $\epsilon_0>0$ such that for all $\epsilon \in (0,\epsilon_0]$
the filtration $\sG_\bullet$ is a $\mu'_\epsilon$-Harder-Narasimhan-filtration.
\end{corollary}
\begin{proof}
The filtration $\sG_\bullet$, constructed Lemma~\ref{fil.7},~b), is a $\mu'_\epsilon$-Harder-Narasimhan filtration for all $0 < \epsilon \leq \epsilon_0$. 
Taking the limit of the slopes for $\mu'_\epsilon$ one obtains
$$
\mu'_{{\rm max}}(\sF)=\mu'(\sG_1) \geq
\mu'(\sG_2/\sG_{1}) \geq \cdots \geq \mu'(\sG_\ell/\sG_{\ell-1})=\mu'_{{\rm min}}(\sF),
$$
and since $\sF$ is $\mu'$-semistable, those are all equalities.
\end{proof}

\section{Splittings of Higgs bundles}\label{spl}

The negativity of kernels of Higgs bundles provide a well-known criterion
for the orthogonal complement of a subbundle $\sK$ of $E^{1,0}$ to
be a holomorphic subbundle: It suffices to show that the slope
of the cokernel $\sQ$ with respect to the canonical polarization is zero.
In this section we extend this to a criterion that zero slope
with respect to canonical {\em semi}-polarizations implies -- best
that one can expect -- vanishing of $\partial/\partial \overline{z}$-derivatives
of the orthogonal splitting map $\sQ \to E^{1,0}$ in the corresponding directions.
\par
Assume again that $Y$ is non-singular, that $U\subset Y$ the complement of a 
normal crossing divisor $S$, and that the positivity conditions stated
as Assumptions~\ref{compass-pos} hold true.
Then one has the decomposition (see~\ref{eqint.1})
$$\Omega_Y^1(\log S)=\Omega_1\oplus \cdots \oplus \Omega_s $$ 
as a direct sum of $\mu$-stable subsheaves $\Omega_i$ of rank $n_i$.
\begin{lemmadef}\label{spl.1} \
\begin{enumerate}
\item[i.] The $\mu$-stable direct factors $\Omega_i$ and their determinants
$\det(\Omega_i)$ are nef. The cycles $\ch_1(\Omega_i)^{n_i+1}$ are numerically trivial.
\item[ii.] For $\nu_1, \ldots,\nu_s$ with $ \nu_1+\cdots+\nu_s=n$ the product 
$\ch_1(\Omega_1)^{\nu_1}. \cdots . \ch_1(\Omega_s)^{\nu_s}$ 
is a positive multiple of $\ch_1(\omega_Y(S))^{n}$,
if $ \nu_\iota=n_\iota$ for $\iota=1,\ldots,s$. Otherwise it is zero.
\item[iii.] $\ch_1(\Omega_1)^{n_1}. \cdots . \ch_1(\Omega_s)^{n_s}>0$.
\item[iv.] Let $D$ be an effective $\Q$ divisor. Then $D.\ch_1(\omega_Y(S))^{n-1}=0$ if and only if
$$
D.\ch_1(\Omega_1)^{\nu_1}. \cdots . \ch_1(\Omega_s)^{\nu_s}=0
$$ 
for all $\nu_1, \ldots,\nu_s$ with $ \nu_1+\cdots+\nu_s=n-1$. 
\item[v.] Let ${\rm NS}_0$ denote the subspace of the Neron-Severi group ${\rm NS}(Y)_\Q$ of $Y$ which is generated by all prime divisors $D$ satisfying the equivalent conditions in iv). Then all effective divisors $B$ with class in ${\rm NS}_0$ is supported in $S$.
\item[vi.] If for some $\alpha\in \Q$ one has $\ch_1(\Omega_i)-\alpha\cdot\ch_1(\Omega_j) \in {\rm NS}_0$ then $i = j$.
\end{enumerate}
\end{lemmadef}
\begin{proof}
Parts i), ii), iii) and vi) have been shown in \cite[Lemmata 1.6 and 1.9]{VZ07}.
Part iv) follows from the nefness of $\det(\Omega_i)$. For v) consider a prime divisor $D$ whose support meets $U$. Since $\omega_Y(S)$ is nef and ample with respect to $U$, the restriction $\omega_Y(S)|_D$ is nef and big,
and hence  $D.\ch_1(\omega_Y(S))^{n-1}>0$. So the nefness of $\omega_Y(S)$ implies that none of the components of $B$ in v) can meet $U$. 
\end{proof}
Using the notations from Section~\ref{fil} consider $m=1$ and the 
tuple $\underline{D}^{(1)}$ where all divisors are $D^{(1)}_j=K_Y+S$ for 
some canonical divisor $K_Y$. Then the slope
$\mu_{\underline{D}^{(1)}}(\sF)$, considered there, is equal to $\mu(\sF)$. 
Using Lemma~\ref{spl.1} the $\mu$-equivalence, as given by Definition~\ref{fil.2}, can be made more precise.
Recall that we define two torsion free coherent sheaves $\sG$ and $\sF$ to be $\mu$-equivalent, if there is a chain of $\mu$-equivalent inclusions
$$
\sG=\sG_1 \hookrightarrow \sF_1 \hookleftarrow \sG_2 \hookrightarrow \sF_2 \hookleftarrow
\cdots \cdots \hookrightarrow \sF_{\ell-1} \hookleftarrow \sG_\ell \hookrightarrow \sF_\ell=\sF.
$$
\begin{addendum}\label{spl.1a} 
Let $\tau:U'\to Y$ be the complement of all prime divisors $D\leq S$ with $D\in {\rm NS}_0$. 
Let $\sF$ and $\sG$ be torsion free coherent sheaves on $Y$. 
\begin{enumerate}
\item[vii.] Assume that $\sG$ is a subsheaf of $\sF$ which is $\mu$-equivalent to $\sF$.
Then $\ch_1(\sF)-\ch_1(\sG)$ lies in the subspace ${\rm NS}_0$, defined in Lemma~\ref{spl.1}~v), and $\sG|_{U'} \to \sF|_{U'}$ is an isomorphism. In particular this holds if $\sG \hookrightarrow \sF$ is an inclusion of $\mu$-semistable sheaves of the same slope and rank.
\item[viii.] The following conditions are equivalent:
\begin{enumerate}
\item[a.] $\sG$ and $\sF$ are $\mu$-equivalent.
\item[b.] There exists an isomorphism $\tau^*\sG \to \tau^* \sF$.
\item[c.] There exists an effective divisor $B\in {\rm NS}_0$ with $\sG \subset \sF\otimes \sO_Y(B)$.
\end{enumerate} 
\item[ix.] Let $\theta: \sG \to \sF$ be a morphism of $\mu$-semistable sheaves of the same slope, and let ${\rm Im}'(\theta)$ denote the saturated image, i.e.\ the kernel of
$$
\sF\>>> (\sF/{\rm Im}(\theta))/_{\rm torsion}.
$$
Then ${\rm Im}'(\theta)$ is a $\mu$-semistable subsheaf of $\sF$ of slope $\mu(\sF)$, and
the inclusion ${\rm Im}(\theta)\hookrightarrow {\rm Im}'(\theta)$ is an isomorphism over $U'$.
\end{enumerate}
\end{addendum}
\begin{proof}
Part vii) follows directly from the definition of $\mu$-equivalence in~\ref{fil.2} and from the definition of ${\rm NS}_0$ in Lemma and Definition~\ref{spl.1}. As a consequence, in viii) the condition a) implies b). 

On the other hand, given an isomorphism $\tau^*\sG \cong \tau^* \sF$, hence
$\tau_*\tau^*\sG \cong \tau_*\tau^* \sF$, one finds effective divisors $B$ and $B'$, both supported in $Y\setminus U'$, with 
$$
\sG \hookrightarrow \sG\otimes \sO_Y(B')= \sF\otimes \sO_Y(B) \hookleftarrow \sF.
$$ 
In particular b) implies c). Finally, since $B \in {\rm NS}_0$ one finds that c) implies a). 

For part ix) one just has to remark that the nefness of $\omega_Y(S)$
implies that
$$
\mu(\sG) \leq \mu({\rm Im}(\theta)) \leq \mu({\rm Im}'(\theta)) \leq \mu(\sF).
$$
\end{proof}
\begin{exdef}\label{spl.1c} Let $\sF$ be a $\mu$-semistable torsion free coherent sheaf. As for slopes defined by polarizations (e.g.\ \cite[page 23]{HL}) one finds for semi-polarizations
a maximal $\mu$-polystable subsheaf ${\rm Soc}(\sF)=\sG_1\oplus \cdots \oplus \sG_\ell$ of slope $\mu(\sF)$. Remark that in general the saturated hull of ${\rm Soc}(\sF)$ is no longer $\mu$-polystable, but for some effective divisor $B\in {\rm NS}_0$ it will be contained in the $\mu$-polystable sheaf $(\sG_1\oplus \cdots \oplus \sG_\ell)\otimes \sO_Y(B)$, and
both are $\mu$-equivalent.  

In Section~\ref{stablengthsplit} we will need the {\em cosocle} ${\rm Cosoc}(\sF)$ of $\sF$, defined as the dual of the socle of $\sF^\vee$. In down to earth terms this is
the largest $\mu$-polystable sheaf of slope $\mu(\sF)$ for which there exists a morphism
$\theta: \sF \to {\rm Cosoc}(\sF)$, surjective over some open set.  
\end{exdef}
In the sequel we consider again an irreducible polarized complex variation of Hodge structures $\V$ of weight $1$ with
unipotent monodromy at infinity and with Higgs bundle
$$
\big(E=E^{1,0}\oplus E^{0,1}, \ \theta: E^{1,0}\to E^{0,1}\otimes \Omega^1_Y(\log S)\big).
$$
We assume that $\V$ is non-unitary, hence that $\theta\neq 0$.

Recall that a Higgs subsheaf $(\sG,\theta|_{\sG})$ of a Higgs bundle
$(E,\theta)$ is a subsheaf with $\theta(\sG)\subset \sG\otimes \Omega^1_Y(\log S)$.
Correspondingly a torsion free Higgs quotient sheaf is of the form $\sQ = E/\sG$, 
where $\sG$ is saturated and a Higgs subsheaf.
By \cite[Proposition 2.4]{VZ07} one obtains as a corollary of Simpson's correspondence:
\begin{lemma}\label{spl.2} Let $\underline{D}^{(\iota)}$ be a finite system of $n-1$-tuples of nef $\R$-divisors. Let $(E,\theta)$ be the Higgs bundle of a complex polarized variation of Hodge structures with unipotent monodromy at infinity. Then:
\begin{enumerate}
\item[i.] $\mu_{\underline{D}^{(\bullet)}}(\sG)\leq 0$ for all Higgs subsheaves $\sG$. 
\item[ii.] $\mu_{\underline{D}^{(\bullet)}}(\sQ)\geq 0$ for all torsion free Higgs
quotient sheaves $\sQ$.
\item[iii.] If for one $\iota$ and for all $j$ the divisors $\underline{D}^{(\iota)}_j$ are ample with respect to $U$, then the following conditions are equivalent for a saturated Higgs subsheaf $\sG$ of $E$ and for $\sQ=E/\sG$:
\begin{enumerate}
\item[1.] $\mu_{\underline{D}^{(\bullet)}}(\sG)= 0$. 
\item[2.] $\mu_{\underline{D}^{(\bullet)}}(\sQ)= 0$.
\item[3.] $\sG$ is a direct factor of the Higgs bundle $E$.
\end{enumerate}  
\end{enumerate}
\end{lemma}
Let us write $D_i$ for a divisor with $\sO_Y(D_i)=\det(\Omega_i)$ and consider 
for $\iota=1,\ldots, s$ the tuple $\widetilde{\underline{D}}^{(\iota)}$
\begin{equation}\label{divisors}
(\stackrel{n_\iota-1}{\overbrace{D_\iota, \ldots, D_\iota}},\stackrel{n_1}{\overbrace{D_1,\ldots,D_1}}, \ldots\stackrel{n_{\iota-1}}{\overbrace{D_{\iota-1},\ldots,D_{\iota-1}}},
\stackrel{n_{\iota+1}}{\overbrace{D_{\iota+1},\ldots,D_{\iota+1}}}, \ldots ,
\stackrel{n_s}{\overbrace{D_s,\ldots,D_s}}).
\end{equation}
For some binomial coefficients one can write
$$
\mu(\sF)=\sum_{\iota=1}^s \alpha_\iota \cdot \mu_{\widetilde{\underline{D}}^{(\iota)}}(\sF).
$$
To get rid of the $\alpha_\iota$ we replace $\widetilde{\underline{D}}^{(\iota)}$ by the tuple $\underline{D}^{(\iota)}$ obtained by multiplying each of the divisors in
$\widetilde{\underline{D}}^{(\iota)}$ by $\sqrt[n-1]{\alpha_\iota}$. So for the intersection cycle
one gets
$$
(\underline{D}^{(\iota)})^{n-1}= \alpha_\iota \cdot (\widetilde{\underline{D}}^{(\iota)})^{n-1}
$$ 
and one finds 
\begin{equation}\label{slopeeq}
\mu_{\underline{D}^{(\iota)}}(\sF)=
\alpha_\iota \cdot \mu_{\widetilde{\underline{D}}^{(\iota)}}(\sF) \mbox{ \ \ and \ \ }
\mu(\sF)=\sum_{\iota=1}^s \mu_{\underline{D}^{(\iota)}}(\sF)= \mu_{\underline{D}^{\bullet}}(\sF).
\end{equation}
Remark that $\mu_{\underline{D}^{(\iota)}}(\Omega_i)\neq 0$ if and only if $\iota=i$.
\begin{properties}\label{spl.1b} \ 
\begin{enumerate}
\item[1.] If $\V$ is irreducible and non-unitary there exists 
some $\iota$ with $\mu_{\underline{D}^{(\iota)}}(E^{1,0})>0$.
\item[2.] If $\sL$ is an invertible sheaf, nef and big, then for all $j$ one has
$\mu_{\underline{D}^{(j)}}(\sL)>0$.
\end{enumerate}
\end{properties}
\begin{proof}
For part 1) remark that Lemma~\ref{spl.2},~ii) and iii) imply that $\mu(E^{1,0})>0$.
For 2) recall that for $\nu \gg 1$ the sheaf $\sL^\nu\otimes \Omega_j^{-1}$ has a section with divisor $\Gamma$. Since the $D_j$ are all nef, $\mu_{\underline{D}^{(j)}}(\sO_Y(\Gamma))\geq 0$ and hence
$$
\nu\cdot \mu_{\underline{D}^{(j)}}(\sL)\geq \alpha_j \ch_1(\Omega_1)^{n_1}. \cdots . \ch_1(\Omega_s)^{n_s}>0.
$$
\end{proof} 
Next we consider a small twist of $\mu$ by choosing for $\epsilon \geq 0$ 
$$
\mu^{\{\iota\}}_{\epsilon}(\sF)=\epsilon\cdot\mu_{\underline{D}^{(\iota)}}(\sF)+  \mu(\sF).
$$
For $s>1$ none of the divisors $\underline{D}^{(\iota)}_j$ is ample.
So we are not allowed to apply part iii) of Lemma~\ref{spl.2} to the slope
$\mu_{\underline{D}^{(\iota)}}$. 

For $\mu^{\{\iota\}}_\epsilon$ things are better. For a for a Higgs subbundle $\sG$ of $E$ the first part of Lemma~\ref{spl.2} only implies that $\mu_{\underline{D}^{(\iota)}}(\sG)\leq 0$. Since $\mu(\sG)\leq 0$ the equality $\mu^{\{\iota\}}_{\epsilon}(\sG)=0$ can only hold for $\epsilon >0$ if $\mu(\sG)=\mu_{\underline{D}^{(\iota)}}(\sG)=0$. This implies that the saturated hull of $\sG$ in $E$ is a direct factor, contradicting the irreducibility of $\V$. So $\rk(\sG) < \rk(E)$ implies that $\mu^{\{\iota\}}_\epsilon(\sG) < 0$. 

As we will show in Section~\ref{erg} the same holds for the slopes 
$\mu_{\underline{D}^{(\iota)}}$ if the universal covering $\tilde{U}$ is a bounded symmetric domain. Without this information, one just has the following criterion.  
\begin{proposition}\label{spl.3} 
Let 
$$
0\>>> \sK \>>> E^{1,0} \>>> \sQ \>>> 0
$$
be an exact sequence, and let $s:\sQ\to E^{1,0}$ be the orthogonal complement of $\sK$. 
Assume that for some $\iota$ the slope $\mu_{\underline{D}^{(\iota)}}(\sQ)=0$. Then
\begin{enumerate}
\item[a.] The composition 
$$
\sQ \> s >> E^{1,0} \> \theta >> E^{0,1} \otimes \Omega^1_Y(\log S) \> {\rm pr}_\iota >>
E^{0,1} \otimes \Omega_\iota
$$
is zero.
\item[b.] $s:\sQ \to E^{1,0}$ is holomorphic in the direction $\Omega_\iota$.
\end{enumerate}
\end{proposition}
Remark that a priori $s$ is a $C^\infty$ map. So part b) of 
the Proposition needs some explanation. Recall that we have the decomposition
$\tilde{U}=M_1\times \cdots \times M_s,$ 
corresponding to the decomposition of $\Omega^1_Y(\log S)$ in $\mu$-stable direct factors.
Write $n_0=0$, again $n_i=\rk(\Omega_i)=\dim(M_i)$ and $m_i=\sum_{j=0}^i n_j$.

Given a point $y\in U$ let us choose a local coordinate system
$z_{1},\ldots,z_{n}$ in a neighborhood of $y$ such that $\pi^*(z_{m_{i-1}+1}),\ldots,\pi^*(z_{m_i})$ are
coordinates on $M_i$.

\begin{definition}\label{spl.4}
The inclusion $s:\sQ \to E^{1,0}$ is holomorphic in the direction
$\Omega_\iota$ if its image is invariant under the action of 
$\partial/\partial{\bar{z}_k}$ 
on $E^{1,0}$ for $k=m_{\iota-1}+1,\ldots,m_\iota$.
\end{definition}

\begin{proof}[Proof of Proposition~\ref{spl.3}] We assume $\iota =1$.
Locally, in some open set $W \subset U$ choose  complex coordinates $z_1,\ldots,z_n$
as above and unitary frames of $E^{1,0}$ and $E^{0,1}$. That is, choose
$C^\infty$-sections $e_1,\ldots,e_\ell$ of $E^{1,0}$ and
$f_1,\ldots,f_{\ell'}$ of $E^{0,1}$ orthogonal
with respect to the scalar product $h(\cdot,\cdot)$ coming from
the Hodge metric, and such that $e_1,\ldots,e_k$ generate $\sK$ while
$e_{k+1},\ldots,e_\ell$ generate $s(Q)$. Write the Higgs field
$\theta$ in these coordinates as
$$ \theta(e_\alpha) = \sum_{i=1}^n \sum_{\beta=1}^{\ell'}
\theta^i_{\alpha,\beta} f_\beta dz_i. $$
By \cite[Theorem~5.2]{Gr70} the curvature $R$ of the metric connection 
$\nabla_h$ on $E^{1,0}$ is given by
\begin{equation}\label{eqspl.1} 
R_{E^{1,0}} = \theta \wedge \theta^* = \sum_{i,j=1}^n 
(R_{E^{1,0}})^{i,j} dz_i \wedge \bar{z}_j,
 \quad \text{where} \quad (R_{E^{1,0}})^{i,j}_{\alpha,\beta} = 
\sum_{\gamma=1}^{\ell'} \theta^i_{\alpha,\gamma} 
\overline{\theta^i_{\beta,\gamma}}.
\end{equation}
For the subbundle $\sK \subset E^{1,0}$ the composition
$$b: \sK \>>> E^{1,0} \>\nabla_h >> E^{1,0} \otimes \Omega^1_U \>>> \sQ \otimes \Omega^1_U$$
of the metric connection and the quotient is called second fundamental
form. Taking complex conjugates we obtain a map
$$c: \bar{\sK} \cong \sK^\vee \>>> \bar{\sQ} \otimes \Omega^{0,1}_U 
\cong \sQ^\vee \otimes \Omega^{0,1}_U . $$ 
Both maps are only $C^\infty$. 
We write the map $c$ in coordinates
$$
c(e_\alpha) =  \sum_{i=1}^n  \sum_{\beta=1}^k c^i_{\alpha,\beta} e_\beta dz_i.
$$
By \cite[Theorem~5.2]{Gr70} the curvature of
the metric connection on $\sQ$ is given by
\begin{gather}\notag 
R_{\sQ} = (\theta s ) \wedge (\theta s)^* + c \wedge c^*= \sum_{i,j=1}^n
(R_{\sQ})^{i,j} 
dz_i \wedge \bar{z}_j,\quad \mbox{where}\\
\label{eqspl.2}
(R_{\sQ})^{i,j}_{\alpha,\beta} = \sum_{\gamma=1}^{\ell'} \theta^i_{\alpha,\gamma} 
\overline{\theta^i_{\beta,\gamma}}
+ \sum_{\gamma=1}^{k} c^i_{\alpha,\gamma} 
\overline{c^i_{\beta,\gamma}},\quad \text{for} \quad \alpha,\beta \in 
\{k+1,\ldots,\ell\}.
\end{gather}
We conclude that for all $i$, the matrices $(R_{\sQ})^{i,i}$ are positive 
semi-definite. Moreover their traces are zero if and only if $\theta^i_{\alpha,\beta} =0$ 
and $c^i_{\alpha,\beta}=0$ for all $\alpha,\beta \in \{k+1,\ldots,\ell\}$.

We write $R(\Omega_i)$ for the curvature of $\det(\Omega_i)$.
By Lemma~\ref{spl.1}~ii) and after rescaling $z_i$ by suitable constants we
may assume that over $W$ 
$$
R(\Omega_i) =
dz_{m_{j-1}+1}\wedge d\bar{z}_{m_{j-1}+1} + \cdots + dz_{m_j}\wedge d\bar{z}_{m_j},
$$
keeping the convention $m_0=n_0 =0$. Then 
$$
R(\Omega_1)^{n_1-1}\wedge R(\Omega_2)^{n_2} \wedge \cdots \wedge R(\Omega_s)^{n_s}=
\sum_{i=1}^{n_1} C_i \cdot \bigwedge_{j\neq i} dz_{j}\wedge d\bar{z}_{j}, 
$$
for some binomial coefficients $C_i >0$. The hypothesis $\mu_{\underline{D}^{(1)}}(\sQ)=0$
is equivalent to
$$
0=\big(\frac{\sqrt{-1}}{2\pi}\big)\cdot \int_U {\rm tr}(R_\sQ)\wedge
R(\Omega_1)^{n_1-1}\wedge R(\Omega_2)^{n_2} \wedge \cdots \wedge R(\Omega_s)^{n_s}.
$$
Since ${\rm tr}(R_\sQ)$ and all the $R(\Omega_i)$ are positive semidefinite, 
the integral has to be zero on all open sets, in particular on $W$. We deduce
\begin{equation}\label{eqspl.3} \begin{split} 
0 &= \int_W (\sum_{i,j=1}^n {\rm tr}(R_Q)^{i,j} dz_i \wedge d\bar{z}_j) \wedge
(\sum_{i=1}^{n_1} C_i \cdot \bigwedge_{j\neq i} dz_{j}\wedge d\bar{z}_{j}) \\
&= \int_W C_i\, \tr(R_\sQ)^{i,i}\bigwedge_{j = 1}^n dz_{j}\wedge d\bar{z}_{j}. 
 \end{split}\end{equation}
Hence $\tr(R_\sQ)^{i,i} =0$ for all $i$ and we obtain the vanishing on $U$
of the composition ${\rm pr}_\iota \circ\theta \circ s$
as claimed in a) and of all $c^i_{\alpha,\beta}$. Since 
the $(0,1)$-part of the metric connection $\nabla_h$ is $\bar{\partial}$, 
the vanishing of $c^i_{\alpha,\beta}$  is what is claimed in b).
Since the sheaves $\Omega_\iota$, $E^{1,0}$ and $E^{0,1}$ are locally free, both vanishing statements extend to the whole of $Y$.
\end{proof}

\section{Purity of Higgs bundles with Arakelov equality}\label{sta}

In this section we will prove Theorem~\ref{purity_Thm}. So keeping 
the assumptions from Section~\ref{spl} we will assume in addition that $\V$ is non-unitary and that it
satisfies the Arakelov equality 
$$
\mu(\V)=\mu(E^{1,0})-\mu(E^{0,1})=\mu(\Omega^1_Y(\log S)).
$$ 
By \cite[Theorem 1]{VZ07} we know that $E^{1,0}$ and $E^{0,1}$ are both $\mu$-semistable. We keep the notations from the last section. In particular as in~\ref{divisors} and~\ref{slopeeq} we define
tuples $\underline{D}^{(\iota)}$ of divisors for $\iota=1,\ldots ,s$ with $\mu=\mu_{\underline{D}^{(\bullet)}}$. Moreover
$$
\mu^{\{\iota\}}_\epsilon = \mu + \epsilon\cdot \mu_{\underline{D}^{(\iota)}}
$$ 
denotes a small perturbation of the slope $\mu$. First we show that this is the slope associated with a small perturbation of the original collection of divisors by a suitable collection of
nef divisors, as studied in Section~\ref{fil}. 
\par
\begin{lemma}\label{twists}
For some tuples of nef $\R$-divisors $\underline{H}^{(i)}$ one has $\mu^{\{\iota\}}_{\epsilon}=\mu_{\underline{D}^{(\bullet)}+ \epsilon\cdot \underline{H}^{(\bullet)}}$. 
\end{lemma}
\begin{proof}
There are several choices for the $\underline{H}^{(i)}$. In the description of the tuple of divisors $\underline{\tilde{D}}^{(\iota)}$
in~\ref{divisors} denote the first entry by $D_\ell$. Then the first entry in $\underline{D}^{(\iota)}$
is $\sqrt[n-1]{\alpha_\iota} \cdot D_\ell$. Here $\ell=\iota$, if $n_\iota >1$, or some other index
in case that $n_\iota=1$. 

Then choose the tuples of $\R$-divisors $\underline{H}^{(\bullet)}$
with $H_j^{(i)}=0$ for $i=1,\ldots,s$ and for $j=1,\ldots,n-1$,
except for $H_1^{(\iota)}$ which is chosen to be $\sqrt[n-1]{\alpha_\iota} \cdot D_\ell$.
This implies that $\underline{D}^{(i)}+ \epsilon\cdot \underline{H}^{(i)}=\underline{D}^{(i)}$ for $i\neq \iota$,
whereas 
$$
(\underline{D}^{(\iota)}+\epsilon \underline{H}^{(\iota)})^{n-1}=
(1+\epsilon)\cdot (\underline{D}^{(\iota)})^{n-1}.
$$
So for a sheaf $\sF$ one finds
\begin{multline*}
\mu_{\underline{D}^{(\bullet)}+ \epsilon\cdot \underline{H}^{(\bullet)}}(\sF)=
\sum_{i=1}^s \mu_{\underline{D}^{(i)}+ \epsilon\cdot \underline{H}^{(i)}}(\sF)=
(1+\epsilon)\cdot \mu_{\underline{D}^{(\iota)}}(\sF) + 
\sum_{i\neq \iota} \mu_{\underline{D}^{(i)}}(\sF)\\
=\epsilon\cdot \mu_{\underline{D}^{(\iota)}}(\sF) + 
\sum_{i=1}^s \mu_{\underline{D}^{(i)}}(\sF)= 
\epsilon\cdot \mu_{\underline{D}^{(\iota)}}(\sF) + 
\mu_{\underline{D}^{(\bullet)}}(\sF)=
\mu^{\{\iota\}}_{\epsilon}(\sF) .
\end{multline*}

\end{proof}
By Corollary~\ref{fil.9} one finds a filtration $\sG^{(\iota)}_\bullet$ of $E^{1,0}$
and some $\epsilon_0>0$ such that $\sG^{(\iota)}_\bullet$ is a $\mu^{\{\iota\}}_\epsilon$-Harder-Narasimhan filtration of $E^{1,0}$, for all $\epsilon \in (0,\epsilon_0]$, and a weak $\mu$-Jordan-H\"older filtration.
Of course we may choose $\epsilon_0$ to be independent of $\iota$.

So the quotient sheaves $\sG^{(\iota)}_i/\sG^{(\iota)}_{i-1}$ are $\mu^{\{\iota\}}_\epsilon$-semistable
for all $\epsilon \in [0,\epsilon_0]$, however not necessarily $\mu_{\underline{D}^{(\iota)}}$-semistable.

\begin{lemma}\label{sta.3} \
Let $\sF$ be a $\mu$-stable subsheaf of $E^{1,0}$ with $\mu(\sF)=\mu(E^{1,0})$. Then $\sF$ is pure of type $\iota$ for some
$\iota \in\{1,\ldots,s\}$. Moreover, each subsheaf $\sF'$ of $E^{1,0}$ which is isomorphic to $\sF$ is pure of the same type $\iota$.
\end{lemma}
Recall from Definition~\ref{pure} that $\sF$ is pure of type $\iota$ if the restriction 
$\theta|_{\sF}$ of the Higgs field factors like
$$
\sF \>\theta_\iota >> E^{0,1} \otimes \Omega_\iota \> \subset >> E^{0,1} \otimes 
\Omega_Y^1(\log S).
$$
Equivalently, writing $T_i$ for the dual of $\Omega_i$ and $\theta^\vee_i$ for the the composite
$$
E^{1,0} \otimes T_i \> \theta\otimes {\rm id}_{T_i} >>
E^{0,1}\otimes \Omega_i\otimes T_i \> {\rm contraction} >>
E^{0,1},
$$
one requires $\theta^\vee_i(\sF\otimes T_i)$ to be zero for $i\neq \iota$. Since $\V$ is non-unitary
this is only possible if $\theta^\vee_\iota(\sF\otimes T_\iota)\neq 0$.
\begin{proof}[Proof of Lemma~\ref{sta.3}] Assume that $\sF'\cong \sF$ and that for some
$i\neq i'$ one has 
$$
\theta_i^\vee(\sF\otimes T_i)\neq 0 \mbox{ \ \ and \ \ } 
\theta_{i'}^\vee(\sF'\otimes T_{i'})\neq 0.
$$ 
We will write $\sB_i$ and $\sB_{i'}$ for the saturated hull of those images. 
The Arakelov equality implies that $\theta_i^\vee$ and $\theta_{i'}^\vee$ are morphisms between $\mu$-semistable sheaves of the same slope, hence $\mu(\sB_\iota)=\mu(\sF)+\mu(T_\iota)$ for
$\iota=i, \ i'$.

The sheaves $\sF$ and $T_\iota$ are $\mu$-stable. By Lemma~\ref{fil.7} for $\epsilon>0$, sufficiently small, and for all $j$ the sheaves $\sF$ and $T_\iota$ are $\mu_\epsilon^{\{j\}}$-semistable. Hence $\sF \otimes T_\iota$ is $\mu_\epsilon^{\{j\}}$-semistable, and consequently,  
$$
\mu^{\{j\}}_\epsilon(\sB_\iota)\geq \mu^{\{j\}}_\epsilon (\sF)+\mu^{\{j\}}_\epsilon(T_\iota)
\mbox{ \ \ and \ \ }
\mu_{\underline{D}^{(j)}}(\sB_\iota)\geq \mu_{\underline{D}^{(j)}} (\sF)+\mu_{\underline{D}^{(j)}}(T_\iota).
$$
For $\iota=i$ and $j \neq i$ one obtains
$$
0\geq \mu_{\underline{D}^{(j)}}(\sB_i) \geq \mu_{\underline{D}^{(j)}}(\sF)+\mu_{\underline{D}^{(j)}}(T_i)=\mu_{\underline{D}^{(j)}}(\sF),
$$
and for $\iota=i'\neq k$ 
$$
0\geq \mu_{\underline{D}^{(k)}}(\sB_{i'})\geq\mu_{\underline{D}^{(k)}}(\sF)+\mu_{\underline{D}^{(k)}}(T_{i'})=\mu_{\underline{D}^{(k)}}(\sF')=\mu_{\underline{D}^{(k)}}(\sF).
$$
Then $i\neq i'$ implies that $\mu_{\underline{D}^{(j)}}(\sF)\leq 0$ for all $j$, hence $\mu(\sF)\leq 0$. Since
$\V$ is non-unitary and since $\mu(\sF)=\mu(E^{1,0})$ this contradicts part iii) of Lemma~\ref{spl.2}. 
\end{proof}
Let us define
\begin{equation}\label{kernel}
\sK^{(\iota)}={\rm Ker}\big(E^{1,0}\>>> E^{0,1}\otimes\Omega^1_Y(\log S) \>>> E^{0,1}\otimes \bigoplus_{j\neq \iota} \Omega_j\big).
\end{equation}
\begin{corollary}\label{sta.4}
There exists some $\iota$ with $\sK^{(\iota)}\neq 0$.
\end{corollary}
\begin{proof}
Choose a direct factor $\sF$ of the socle of $E^{1,0}$, hence a $\mu$-stable subsheaf $\sF\subset E^{1,0}$ with $\mu(\sF)=\mu(E^{1,0})$. Then by Lemma~\ref{sta.3} the bundle $\sF$ is contained in 
$\sK^{(\iota)}$ for some $\iota$.
\end{proof}
\par
\begin{lemma}\label{sta.05}
Assume that $\V$ is pure of type $\iota$, hence that
$E^{1,0}=\sK^{(\iota)}$. Then for all $j\neq \iota$ one has $\mu_{\underline{D}^{(j)}}(E^{1,0})=0$.
\end{lemma}
\begin{proof}
If $E^{1,0}=\sK^{(\iota)}$, the saturated image $\sB_\iota$ of  
$$
\theta^\vee_\iota: E^{1,0}\otimes T_\iota \>>> E^{0,1}
$$
has to be non-zero. $\theta^\vee_\iota$ is a map of $\mu$-semistable sheaves of the same slope,
hence $\mu(\sB_\iota)=\mu(E^{1,0}) - \mu(\Omega_\iota)$.
For $\epsilon$ sufficiently small $E^{1,0}\otimes T_\iota$  is $\mu_\epsilon^{\{j\}}$-semistable, and  
$$
\mu_\epsilon^{\{j\}}(\sB_\iota) \geq \mu_\epsilon^{\{j\}}(E^{1,0}) - \mu_\epsilon^{\{j\}}(\Omega_\iota).
$$
Then for $j\neq \iota$ one finds
$$
\mu_{\underline{D}^{(j)}}(\sB_\iota) \geq \mu_{\underline{D}^{(j)}}(E^{1,0}) - \mu_{\underline{D}^{(j)}}(\Omega_\iota)=\mu_{\underline{D}^{(j)}}(E^{1,0}),
$$
which by Lemma~\ref{spl.2} can neither be positive, nor negative, hence it must be zero.
\end{proof}
A similar argument will be used to obtain a stronger statement, which finally will lead to a contradiction, except if $E^{1,0}=\sK^{(\iota)}$ for some $\iota$.
\par
\begin{lemma}\label{sta.1} Let $\ell$ be the length of the filtration $\sG^{(\iota)}_\bullet$.
\begin{enumerate}
\item[a.] Then $\sG^{(\iota)}_{\ell-1} \subset\sK^{(\iota)}$. 
\item[b.] If $\sK^{(\iota)}\neq E^{1,0}$ then  $\mu_{\underline{D}^{(\iota)}}(\sG^{(\iota)}_\ell/\sG^{(\iota)}_{\ell-1})=0$.
\end{enumerate}
\end{lemma}
\begin{proof} 
Let $\nu \in \{1,\ldots,\ell+1\}$ be the largest number with $\sG^{(\iota)}_{\nu-1}\subset \sK^{(\iota)}$.  
If $\nu=\ell+1$ then $E^{1,0}=\sG^{(\iota)}_{\ell}=\sK^{(\iota)}$ and there is nothing to show.
So let us assume that $\nu\leq \ell$, or equivalently that $\sK^{(\iota)}\neq E^{1,0}$.

For all $j\neq \iota$ the restriction of $\theta^\vee_j$ to $\sG^{(\iota)}_{\nu}$ induces a morphism 
\begin{equation}\label{eqsta.2}
\sG^{(\iota)}_{\nu}/\sG^{(\iota)}_{\nu-1}\otimes T_j \>>> E^{0,1},
\end{equation}
and by assumption for at least one $j\neq \iota$ this morphism is non-zero. So we will fix such an index $j$
and assume in the sequel that the saturated image $\sB_j$ of $\theta^\vee_j|_{\sG^{(\iota)}_{\nu}}$ is non-zero.

Since $\sG^{(\iota)}_\bullet$ is a weak $\mu$-Jordan H\"older filtration, the morphism in (\ref{eqsta.2}) is a morphism between $\mu$-semistable sheaves of the same slope, non-zero by assumption. Then
$$
\mu(\sG^{(\iota)}_{\nu}/\sG^{(\iota)}_{\nu-1}\otimes T_j)=\mu(E^{1,0})+\mu(T_j)=\mu(E^{0,1})=\mu(\sB_j).
$$
Since $j\neq \iota$ 
\begin{gather*}
\mu_{\underline{D}^{(\iota)}}(\sG^{(\iota)}_{\nu}/\sG^{(\iota)}_{\nu-1}\otimes T_j)=\mu_{\underline{D}^{(\iota)}}(\sG^{(\iota)}_{\nu}/\sG^{(\iota)}_{\nu-1}),
\mbox{ \ \ and hence}\\ 
\mu^{\{\iota\}}_\epsilon(\sG^{(\iota)}_{\nu}/\sG^{(\iota)}_{\nu-1}\otimes T_j))=\epsilon\cdot \mu_{\underline{D}^{(\iota)}}(\sG^{(\iota)}_{\nu}/\sG^{(\iota)}_{\nu-1})+
\mu(\sB_j).
\end{gather*}
For $0 < \epsilon \leq \epsilon_0$ the sheaf $\sG^{(\iota)}_{\nu}/\sG^{(\iota)}_{\nu-1}
\otimes T_j$ is $\mu^{\{\iota\}}_\epsilon$-semistable, which implies that
\begin{equation}\label{eqsl1}
\epsilon\cdot \mu_{\underline{D}^{(\iota)}}(\sB_j) + \mu(\sB_j)=
\mu^{\{\iota\}}_\epsilon(\sB_j) \geq \epsilon\cdot \mu_{\underline{D}^{(\iota)}}(\sG^{(\iota)}_{\nu}/\sG^{(\iota)}_{\nu-1})+
\mu(\sB_j).
\end{equation}
By the choice of $\sG^{(\iota)}_{\bullet}$ as a $\mu^{\{\iota\}}_\epsilon$-Harder-Narasimhan filtration
one has an inequality
\begin{equation}\label{eqsl2}
\mu^{\{\iota\}}_\epsilon(\sG^{(\iota)}_{\ell}/\sG^{(\iota)}_{\ell-1}) \leq \mu^{\{\iota\}}_\epsilon(\sG^{(\iota)}_{\nu}/\sG^{(\iota)}_{\nu-1}),
\end{equation}
with equality if and only if $\nu=\ell$. Since $\sG^{(\iota)}_{\bullet}$ is a weak $\mu$-Jordan-H\"older filtration the slope $\mu(\sG^{(\iota)}_{\nu}/\sG^{(\iota)}_{\nu-1})$ is independent of $\nu$.
So the inequality~\ref{eqsl2} carries over to one for the slope 
$\mu_{\underline{D}^{(\iota)}}$. As we have seen in Lemma~\ref{spl.2},~i) one has $\mu_{\underline{D}^{(\iota)}}(\sB_j)\leq 0$, so rewriting the inequalities~\ref{eqsl1} and~\ref{eqsl2} one gets
\begin{equation}\label{eqsl3}
\mu_{\underline{D}^{(\iota)}}(\sG^{(\iota)}_{\ell}/\sG^{(\iota)}_{\ell-1}) \leq
\mu_{\underline{D}^{(\iota)}}(\sG^{(\iota)}_{\nu}/\sG^{(\iota)}_{\nu-1})\leq 0. 
\end{equation}
Lemma~\ref{spl.2},~ii) implies however that $\mu_{\underline{D}^{(\iota)}}(\sG^{(\iota)}_{\ell}/\sG^{(\iota)}_{\ell-1}) \geq 0$.
So both inequalities in~\ref{eqsl3} are equalities and b) holds true. Moreover~\ref{eqsl2} is an equality,
hence $\nu=\ell$, as claimed in a). 
\end{proof}
\par
\begin{corollary}\label{sta.2}
If in Lemma~\ref{sta.1} the sheaf $\sQ=E^{1,0}/\sK^{(\iota)}$ is non-zero, it is $\mu$ and $\mu^{\{\iota\}}_\epsilon$-semistable. One has
$$
\mu^{\{\iota\}}_\epsilon(\sQ)= \mu(\sQ) = \mu(E^{1,0})=\mu(E^{0,1})+\mu(\Omega_Y^1(\log S)),
$$
and hence $\mu_{\underline{D}^{(\iota)}}(\sQ)=0$.
\end{corollary}
\begin{proof}
Since $\sK^{(\iota)}$ as the kernel of a morphism between $\mu$-semistable sheaves of the same slope is $\mu$-semistable, $\sQ$ has the same property. 

By Lemma~\ref{sta.1},~b) the slope
$\mu_{\underline{D}^{(\iota)}}(\sG^{(\iota)}_{\ell}/\sG^{(\iota)}_{\ell-1})=0$. 
Since $\sG^{(\iota)}_\bullet$ is a weak $\mu$-Jordan-H\"older filtration and a $\mu^{\{\iota\}}_\epsilon$-Harder-Narasimhan filtration, for all $0\leq \epsilon\leq \epsilon_0$ the quotient $\sG^{(\iota)}_{\ell}/\sG^{(\iota)}_{\ell-1}$ 
is $\mu^{\{\iota\}}_\epsilon$-semistable
and has slope 
$$
\mu^{\{\iota\}}_\epsilon (\sG^{(\iota)}_{\ell}/\sG^{(\iota)}_{\ell-1})=\mu(\sG^{(\iota)}_{\ell}/\sG^{(\iota)}_{\ell-1})=\mu(E^{1,0}).
$$
For $j\neq \iota$ the sheaf $\Omega_j$ is $\mu^{\{\iota\}}_\epsilon$-stable 
of slope 
$\mu^{\{\iota\}}_\epsilon (\Omega_j)=\mu(\Omega_j)=\mu(\Omega_Y^1(\log S))$,
hence 
$$
\sG^{(\iota)}_{\ell}/\sG^{(\iota)}_{\ell-1}\otimes \bigoplus_{j\neq \iota} T_j
$$
is again $\mu^{\{\iota\}}_\epsilon$-semistable of slope $\mu(E^{0,1})$.

Let $\sB$ be the saturated image of $\sG^{(\iota)}_{\ell}/\sG^{(\iota)}_{\ell-1}\otimes \bigoplus_{j\neq \iota} T_j$ in $E^{0,1}$. Then $\mu(\sB)=\mu(E^{0,1})$ and $\mu^{\{\iota\}}_\epsilon(\sB) \geq \mu(E^{0,1})$. 

On the other hand Lemma~\ref{spl.2} implies that $\mu_{\underline{D}^{(\iota)}}(\sB)\leq 0$, hence $\mu^{\{\iota\}}_\epsilon(\sB) = \mu(E^{0,1})$. So $\sB$ as a quotient 
of a $\mu^{\{\iota\}}_\epsilon$-semistable sheaf of the same slope has to be $\mu^{\{\iota\}}_\epsilon$-semistable of slope
$$
\mu(E^{0,1})=\mu^{\{\iota\}}_\epsilon\big(\sG^{(\iota)}_{\ell}/\sG^{(\iota)}_{\ell-1}\otimes \bigoplus_{j\neq \iota} T_j\big).
$$ 
Since $\sQ=E^{1,0}/\sK^{(\iota)}$ is a subsheaf of $\sB\otimes \bigoplus_{j\neq \iota}\Omega_j$ one finds 
$$
\mu^{\{\iota\}}_\epsilon(\sQ) \leq \mu^{\{\iota\}}_\epsilon\big(\sB\otimes \bigoplus_{j\neq \iota}\Omega_j\big)=\mu(E^{0,1}),
$$
and since it is a quotient of $\sG^{(\iota)}_{\ell}/\sG^{(\iota)}_{\ell-1}$ one has $\mu^{\{\iota\}}_\epsilon(\sQ) \geq \mu(E^{0,1})$. One obtains the equality of slopes
in Corollary~\ref{sta.2}. Finally $\sQ$ as a subsheaf of a $\mu^{\{\iota\}}_\epsilon$-semistable sheaf of the same slope is itself $\mu^{\{\iota\}}_\epsilon$-semistable.
\end{proof}

\begin{proof}[Proof of Theorem~\ref{purity_Thm}]
Renumbering the factors we will assume that $\sK^{(1)}\neq 0$, and we will write
$$
\Omega=\bigoplus_{j=2}^s\Omega_j\mbox{ \ \ and \ \ } T=\Omega^\vee.
$$
So $\sK^{(1)}$ is the kernel of the composition
$$
E^{1,0}\>>> E^{0,1}\otimes \Omega^1_Y(\log S) \> {\rm pr} >> E^{0,1}\otimes \Omega.
$$
Let $\sK_1$ be the kernel of
$$
E^{1,0}\>>> E^{0,1}\otimes \Omega^1_Y(\log S) \> {\rm pr}_1 >> E^{0,1}\otimes \Omega_1.
$$
\begin{claim}\label{sta.5}
$E^{1,0}$ is the direct sum $\sK^{(1)}\oplus \sK_1$.
\end{claim}
\begin{proof}
By Corollary~\ref{sta.2} the sheaf $\sQ=E^{1,0}/\sK^{(1)}$ satisfies $\mu_{\underline{D}^{(1)}}(\sQ)=0$. 
So Proposition~\ref{spl.3},~a), tells us that the orthogonal complement
$s(\sQ)$ is contained in $\sK_1$.

The intersection of $\sK^{(1)}$ and $\sK_1$ lies in the kernel of $\theta$. Hence it is 
zero and the induced map $\sK_1 \to \sQ$ is injective. On the other hand
$$
\sQ \> s >> E^{1,0} \>>> \sQ
$$
factors through $\sK_1 \to \sQ$, and the latter must be surjective. This implies that
$$
E^{1,0}=\sK^{(1)} \oplus \sK_1.
$$
\end{proof}
Let $\sB^{(1)}$ and $\sB_1$ be the saturated images of 
$$
E^{1,0}\otimes T \>>> E^{0,1} \mbox{ \ \ and \ \ }
E^{1,0}\otimes T_1 \>>> E^{0,1},  
$$ 
respectively.
\begin{claim}\label{sta.6}
$\sB^{(1)}\cap\sB_1=0$.
\end{claim}
\begin{proof}
Dualizing the exact sequences
\begin{gather*}
0\>>> \sB^{(1)} \>>> E^{0,1} \>>> \sC^{(1)}=E^{0,1}/\sB^{(1)} \>>> 0\\
\mbox{and \ \ }0\>>> \sB_1 \>>> E^{0,1} \>>> \sC_1=E^{0,1}/\sB_1 \>>> 0
\end{gather*}
one obtains that
$$
{\sC^{(1)}}^\vee = {\rm Ker}({E^{0,1}}^\vee \> \tau >> {\sB^{(1)}}^\vee)\mbox{ \ \ and \ \ }
\sC_1^\vee = {\rm Ker}({E^{0,1}}^\vee \> \tau_1 >> \sB_1^\vee). 
$$
The dual Higgs bundle $E^\vee$ has ${E^{0,1}}^\vee$ as subsheaf of bidegree $(1,0)$ and ${E^{0,1}}^\vee$ is of bidegree $(0,1)$. The composite
$$
{E^{0,1}}^\vee \> \tau >> {\sB^{(1)}}^\vee \> \subset >> {E^{1,0}}^\vee \otimes \Omega\mbox{ \ \ and \ \ }
{E^{0,1}}^\vee \> \tau_1 >> \sB_1^\vee \> \subset >> {E^{1,0}}^\vee \otimes \Omega_1
$$
are the components of the dual Higgs field.
Applying Claim~\ref{sta.5} to $E^\vee$ one obtains a decomposition ${E^{0,1}}^\vee=
{\sC^{(1)}}^\vee \oplus \sC_1^\vee$, hence 
$E^{0,1}\cong \sC^{(1)} \oplus \sC_1$ and $\sB^{(1)}\cap \sB_1=0$.
\end{proof}
So one obtains a direct sum decomposition of Higgs bundles
$$
(E,\theta)=\big( \sK^{(1)} \oplus \sB^{(1)}, \theta^{(1)}=\theta|_{\sK^{(1)}}\big)\oplus
\big( \sK_1 \oplus \sB_1, \theta_1=\theta|_{\sK_1}\big)
$$
corresponding to a decomposition $\V=\V^{(1)}\oplus \V_1$. The irreducibility of $\V$ and the assumption $\sK^{(1)}\neq 0$ imply $\V_1=0$, hence $\sK_1=0$.
\end{proof}
\subsection{Using superrigidity}\label{erg}
As mentioned in the introduction, the purity of the Higgs fields in 
Theorem~\ref{purity_Thm} follows from the Margulis 
Superrigidity Theorem, without
using the Arakelov equality, provided all the direct $\mu$-stable factors of $\Omega^1_Y(\log S)$ 
are of type C. We will show below, that for variations of Hodge structures of weight $1$ 
it is sufficient to assume that the universal covering $\tilde{U}$ of $U$ is a 
bounded symmetric domain. In different terms, if $\Omega_i$ is of type~B 
we suppose that it satisfies the Yau-equality
$$
2(n_i+1)\cdot \ch_2(\Omega_i).\ch_1(\Omega_i)^{n_i-2}.\ch_1(\omega_Y(S))^{n-n_i}=
n_i\cdot \ch_1(\Omega_i)^{n_i}.\ch_1(\omega_Y(S))^{n-n_i}
$$ 
(\cite{Ya93}, see also \cite[Theorem 1.4]{VZ07}).
\begin{proposition}\label{erg.1}
Suppose that $\tilde{U}$ is a bounded symmetric domain
and that $\V$ is an irreducible complex 
polarized variation of Hodge structures of weight $1$ with unipotent monodromy at infinity. 
Then the associated Higgs bundle $(E^{1,0}\oplus E^{0,1},\theta)$ is pure of type $\iota$ for
some $\iota \in \{1,\ldots,s\}$.
\end{proposition}
\begin{proof}[Sketch of the proof]
By assumption $U = \Gamma \backslash \tilde{U}$ is the quotient of 
a bounded symmetric domain $\tilde{U}=M_1\times \cdots \times M_s$ by a lattice $\Gamma$. 
We can write $M_i = G_i/K_i$ as quotient of a real, non-compact, simple Lie group by a maximal compact subgroup. 

Assume first that $U=U_1\times U_2$. By \cite[Proposition 3.3]{VZ05} an irreducible 
local system on $\V$ is of the form ${\rm pr}_1^*\V_1\otimes {\rm pr}_2^*\V_2$,
for irreducible local systems $\V_i$ on $U_i$ with Higgs bundles
$(E_i,\theta_i)$. Since $\V$ is a variation of Hodge structures of weight $1$,
one of those, say $\V_2$ has to have weight zero, hence it must be unitary. 

Then the Higgs field on $U$ factors through $E^{0,1}\otimes \Omega^1_{U_1}$.
By induction on the dimension we may assume that $\V_1$ is pure of type $\iota$
for some $\iota$ with $M_\iota$ a factor of $\tilde{U}_1$. So the same holds true for
$\V$.

Hence we may assume that $U$ is irreducible, or even that 
\begin{equation}\label{eqerg.1}
\mbox{\rm {no finite \'etale covering of }}U \mbox{\rm { is a product of proper subvarieties}}.
\end{equation}
By \cite{Zi} \S~2.2, replacing $\Gamma$ by a subgroup of finite index, hence replacing $U$ by a finite unramified cover, there is a partition of $\{1,\ldots,s\}$ into subsets $I_k$ such that $\Gamma = \prod_k \Gamma_k$ and $\Gamma_k$
is an irreducible lattice in $\prod_{i \in I_k} G_i$. Here irreducible
means that for any $N \subset \prod_{i \in I_k} G_i$ a normal subgroup,
$\Gamma_k$ is dense in $\prod_{i \in I_k} G_i/N$. 
The condition (\ref{eqerg.1}) is equivalent to the irreducibility of $\Gamma$,
so $I_1 =\{1,\ldots,s\}$. 
\par
If $s=1$ or if $\V$ is unitary, the statement of the proposition is trivial. 
Otherwise,
$G:=\prod_{i=1}^s G_i$ is of real rank $\geq 2$ and the conditions
of Margulis' superrigidity theorem (e.g.\ \cite[Theorem 5.1.2 ii)]{Zi})
are met.  As consequence, the homomorphism $\Gamma \to \Sp(V,Q)$,
where $V$ is a fibre of $\V$ and where $Q$ is the symplectic form on $V$,
factors through a representation $\rho: G \to \Sp(V,Q)$. Since the 
$G_i$ are simple, we can repeat the argument used in the proof of \cite[Proposition 3.3]{VZ05} in the product case: 
$\rho$ is a tensor product of representations, all of which but one have weight $0$. 
\end{proof}
\begin{corollary}\label{erg.2}
Under the assumptions made in Proposition~\ref{erg.1} let $\sQ\neq 0$ be a quotient of $E^{1,0}$ with $\mu_{\underline{D}^{(i)}}(\sQ)=0$, for some $i\in\{1,\ldots,s\}$. Then $\sQ=E^{1,0}$.
\end{corollary}
\begin{proof}
By Proposition~\ref{erg.1} $\V$ is pure of type $\iota$ for some $\iota$. On the other hand Proposition~\ref{spl.3} implies that the orthogonal complement of $\sQ$ lies in the kernel
of the composite
$$
E^{1,0}\>\theta >> E^{0,1}\otimes \Omega^1_Y(\log S) \> {\rm pr}_i >> E^{0,1}\otimes \Omega_i. 
$$
Since $\theta$ is injective and factors through $E^{0,1}\otimes \Omega_\iota$ this implies that
$i=\iota$ and we assume that both are $1$. 

Now one argues as in the proof of Proposition~\ref{spl.3}.
The metric connection $\nabla_h$ is zero in directions not contained in $M_1$, hence in the equation
(\ref{eqspl.1}) one finds that $(R_{E^{1,0}})^{i,j}=0$ as soon as $i>n_1$ or $j>n_1$.
Similarly $c^i_{\alpha,\beta}=0$ for $i>n_1$, and hence
the equation (\ref{eqspl.2}) implies that $(R_{\sQ})^{i,j}_{\alpha,\beta}=0$ for $i>n_1$ or $j>n_1$.
One has again
\begin{multline*}
\mu_{\underline{D}^{(j)}}(\sQ)=\big(\frac{\sqrt{-1}}{2\pi}\big)\cdot \int_U {\rm tr}(R_\sQ)\wedge
R(\Omega_1)^{n_j-1}\wedge R(\Omega_1)^{n_1} \wedge \cdots \\
\wedge R(\Omega_{j-1})^{n_{j-1}}\wedge R(\Omega_{j+1})^{n_{j+1}} \wedge \cdots \wedge
 R(\Omega_s)^{n_s}.
\end{multline*}
As in equation (\ref{eqspl.3}) this is the same as
$$
\int_W (\sum_{i,j=1}^n {\rm tr}(R_Q)^{i,j} dz_i \wedge d\bar{z}_j) \wedge
(\sum_{i=n_{j-1}+1}^{n_j} C_i \cdot \bigwedge_{j\neq i} dz_{j}\wedge d\bar{z}_{j}).
$$
For $j>1$ this is zero, hence $\mu(\sQ)=0$ and one can apply Lemma~\ref{spl.2}.
\end{proof}

\section{Stability of Higgs bundles, lengths of iterated
Higgs fields and splitting of the tangent map}
 \label{stablengthsplit}

In this section we prove  Theorem~\ref{characterization},
the numerical characterization of Shimura varieties,
the equivalent numerical and geometrical characterizations
of ball quotients stated as Addendum~\ref{characterizationadd},
the Corollary~\ref{characterizationadd3} and we finish the proof of Proposition~\ref{shimuraprop}.
Moreover we recall the proof of Corollary~\ref{iteratedKSineq}, essentially contained in 
\cite[Section 5]{VZ07}. 
\par
As usual we assume that $U$ has a non-singular projective compactification
$Y$ with boundary $S=Y\setminus U$ a normal crossing divisor, satisfying the Assumptions~\ref{compass-pos}. In addition, replacing $U$ by an \'etale covering, we will assume as in Section~\ref{ec} that
the family $f:A\to U$ is induced by a generically finite morphism $\varphi:U \to \sA_g$ to a fine moduli scheme $\sA_g$ of polarized abelian varieties with a suitable level structure.
\par
So we consider again an irreducible non-unitary complex polarized variation of Hodge structures $\V$ on $U$, satisfying the Arakelov equality, and with unipotent local monodromy operators at infinity.

By Theorem~\ref{purity_Thm}, the logarithmic Higgs bundle 
$(E=E^{1,0}\oplus E^{0,1},\theta)$ of $\V$ is pure
of type $\iota$, i.e.\ the Higgs field factors through $E^{0,1}\otimes \Omega_\iota$.
We write $\ell=\rk(E^{1,0})$ and $\ell'=\rk(E^{0,1})$, and $n_\iota$ denotes
$\rk(\Omega_\iota)=\dim(M_\iota)$. The Arakelov equality
says that 
$$
\mu(E^{1,0})-\mu(E^{0,1})=\mu(\Omega^1_Y(\log S))=\mu(\Omega_\iota).
$$
Since $\ch_1(E^{1,0})+\ch_1(E^{0,1})=0$ and hence $\ell\cdot\mu(E^{1,0}) + \ell'\cdot\mu(E^{0,1})=0$,
one can restate the Arakelov equality as
\begin{equation}\label{eqco.1}
\frac{\ell+\ell'}{\ell'} \cdot \mu(E^{1,0}) = \mu(\Omega_\iota).
\end{equation}
Let us formulate two easy consequences of the Arakelov equality.
\begin{lemma}\label{etale}
Assume that each irreducible non-unitary $\C$-subvariation of Hodge structures of $R^1f_*\C_A$ satisfies the Arakelov equality. Then:
\begin{enumerate}
\item If $\varphi$ is generically finite, then for each direct factor $\Omega_\iota$ of $\Omega_Y^1(\log S)$
there exists at least one non-unitary local subsystem $\V$ which is pure of type $\iota$.
\item If $\varphi(U)$ is non-singular, then $\varphi:U\to \varphi(U)$ is \'etale.
\end{enumerate}
\end{lemma}
\begin{proof}
Let $F^{1,0}$ be the $(1,0)$-part in the Hodge filtration of $R^1f_*\C_A$.
Since $U$ is generically finite over $\sA_g$ the sheaf $\det(f_*\Omega^1_{X/Y})=f_*\omega_{X/Y}$ is big.
Since it is nef, using the slopes introduced in Section~\ref{sta}, one finds by 
the Property~\ref{spl.1b}~2) that
$$
\mu_{\underline{D}^{(\iota)}}(f_*\omega_{X/Y}) = g\cdot \mu_{\underline{D}^{(\iota)}}(f_*\Omega^1_{X/Y})= g\cdot \mu_{\underline{D}^{(\iota)}}(F^{1,0})>0
$$ 
for all $\iota$. Consider an irreducible complex polarized subvariation of Hodge  structures $\V$ with Higgs bundle $(E^{1,0}\oplus E^{0,1},\theta)$. If $\V$ is unitary
$ \mu_{\underline{D}^{(j)}}(E^{1,0})=0$ for all $j$. Otherwise by Theorem~\ref{purity_Thm} $\V$ is pure of type $i=i(\V)$. Lemma~\ref{sta.05} implies that $\mu_{\underline{D}^{(j)}}(E^{1,0})=0$ for $j\neq i(\V)$.

Given $\iota$, the inequality $\mu_{\underline{D}^{(\iota)}}(F^{1,0})>0$ implies that there exist
direct factors $E^{1,0}$ with $\mu_{\underline{D}^{(\iota)}}(E^{1,0})>0$. For the corresponding irreducible subvariations $\V$ of $R^1f_*\C_A$ one finds $\iota=i(\V)$.

For the second statement we choose a nonsingular compactification
$Z$ and a normal crossing divisor $\Sigma\subset Z$ with $\varphi(U)=Z\setminus \Sigma$.
Let us choose a blowing up $\delta: Y' \to Y$ with centers in $S$ such
that $S'=\delta^*(S)$ is again a normal crossing divisor, and such that $\varphi$ extends to
$\varphi:Y'\to Z$. By the Arakelov equality the image $\sI$ of $\hat{\tau}:F^{1,0}\otimes {F^{0,1}}^\vee \to \Omega_Y^1(\log S)$ has the same slope as $\Omega_Y^1(\log S)$.
Since the second sheaf is $\mu$-polystable, $\sI$ is a subsheaf of a direct sum of certain direct factors of $\Omega_Y^1(\log S)$ and both are $\mu$-equivalent. The first part of Lemma~\ref{etale} implies that  all direct factors occur, hence $\sI\hookrightarrow \Omega_Y^1(\log S)$ is an isomorphism over some open set $U'$. The part viii) of Addendum~\ref{spl.1a} allows to choose $U'=U$. 

Since $\sA_g$ is a fine moduli scheme, the Higgs bundle is the pullback of the Higgs bundle on $\varphi(U)$. Hence $\delta^*(\hat{\tau})$ factors through
$$
\varphi^* \Omega_{Z}^1(\log \Sigma) \longrightarrow \Omega_{Y'}^1(\log S')
$$
with image in $\delta^*\Omega_{Y}^1(\log S)\subset \Omega_{Y'}^1(\log S')$.
Since the last inclusion is an isomorphism over $U$, the surjectivity of the Higgs field on $U$ implies the morphism $\varphi$ is unramified on $U$.
\end{proof}
Let us return to the Higgs bundle $\bigwedge^\ell (E,\theta)$ introduced in (\ref{eqdetHiggs})
and to the Higgs subbundle $\langle\det(E^{1,0})\rangle$ generated by $\det(E^{1,0})$.
From now on we will write $\langle\det(E^{1,0})\rangle$ for the saturated Higgs subbundle of 
$\bigwedge^\ell (E,\theta)$. So $\langle\det(E^{1,0})\rangle^{\ell-m,m}$ denotes the saturated hull of the image of the induced map
$$
\theta^{(m)^\vee}:\det(E^{1,0})\otimes S^m(T) \>>> E^{\ell-m,m}= \bigwedge^{\ell-m}(E^{1,0}) \otimes \bigwedge^m E^{0,1}. 
$$
Note that by this change of notation we neither change the slopes, nor the length
$$
\varsigma(E)=\varsigma((E,\theta))=
{\rm Max}\{\ m\in \N ; \ \langle\det(E^{1,0})\rangle^{\ell-m,m} \neq 0\}.
$$
So the next Lemma implies Corollary~\ref{iteratedKSineq}.
\begin{lemma}\label{upper} Assume that $\Omega_\iota$ is of type A or B, hence that
$S^m(\Omega_\iota)$ is $\mu$-stable for all $m$. Then the Arakelov equality implies that 
\begin{equation}\label{equpper} 
{\rm Min}\{\ell,\ell'\} \geq \varsigma(E) \geq \frac{\ell\cdot\ell'\cdot(n_\iota+1)}{(\ell + \ell')\cdot n_\iota}.
\end{equation}
The right hand side of~\ref{equpper} is an equality if and only if $\langle\det(E^{1,0})\rangle$ is a direct factor of $\bigwedge^\ell (E,\theta)$.
\end{lemma}
\begin{proof}
For $0\leq m\leq \varsigma=\varsigma(E)$ the sheaf $\langle\det(E^{1,0})\rangle^{\ell-m,m}$ is a $\mu$-stable sheaf of slope 
\begin{multline*}
(\ell-m)\cdot \mu(E^{1,0}) + m\cdot \mu(E^{0,1})=
\ell \cdot \mu(E^{1,0}) - m\cdot\mu(\Omega_Y^1(\log S))=\\
\big(\frac{\ell \cdot\ell'}{\ell+\ell'}-m\big)\cdot \mu(\Omega_Y^1(\log S)),
\end{multline*}
and of rank $\binom{n_\iota+m-1}{m}$. 
The degree of this sheaf with respect to the polarization $\omega_Y(S)$ is non-positive, hence
\begin{multline}\label{equpper2}
0\geq \frac{\mu(\langle\det(E^{1,0})\rangle)}{\mu(\Omega^1_Y(\log S))}=\sum_{m=0}^\varsigma 
\binom{n_\iota+m-1}{m}\cdot \big(\frac{\ell \cdot\ell'}{\ell+\ell'}-m\big)= \\
\left(\frac{\ell \cdot\ell'}{n_\iota\cdot(\ell+\ell')} 
- \frac{\varsigma}{n_\iota+1}\right)\cdot (\varsigma+1)\cdot\binom{\varsigma+n_\iota}{\varsigma+1},
\end{multline}
and one obtains the second inequality stated in~\ref{equpper}. This is an equality if and only if
(\ref{equpper2}) is an equality. By Simpson's correspondence for polystable Higgs bundles the 
latter holds if and only if $\langle\det(E^{1,0})\rangle$ is a Higgs direct factor of $\bigwedge^\ell (E,\theta)$. The first inequality in~\ref{equpper} is obvious, since
$E^{\ell-m,m}$ is zero for $m \geq {\rm Min}\{\ell,\ell'\}$.
\end{proof}
We now distinguish three cases, according to the type of the bounded symmetric domain
attached to $\Omega_\iota$.

\subsection{Type A: $\Omega_\iota$ is invertible}\label{typeA}

This case is easy to understand. Let us recall the arguments used already in  \cite{VZ07}.
The Arakelov equality and Lemma~\ref{spl.2} imply that
\begin{equation}\label{eqtypeA}
E^{1,0}\>>> E^{0,1}\otimes \Omega_\iota,
\end{equation}
is injective and surjective on some open dense subscheme. So $\ell=\ell'$ and the inequality 
(\ref{equpper}) implies that $\varsigma((E,\theta)) = \ell$.

Both sides in (\ref{eqtypeA}) are $\mu$-semistable of the same slope, and they are $\mu$-equivalent.
A $\mu$-stable subsheaf $\sF$ of $E^{1,0}$ of slope $\mu(E^{1,0})$
generates a Higgs subbundle $\sF\oplus \sF\otimes T_\iota$, whose first Chern class is zero.
So the irreducibility implies that $\sF=E^{1,0}$ and we can state:
\begin{proposition}\label{numcondA}
If $\Omega_\iota$ is invertible, then the Arakelov equality (\ref{eqco.1}) implies that
$E^{1,0}$ and $E^{0,1}$ are both $\mu$-stable of the same rank, that 
$\varsigma((E,\theta))=\ell$ and that $\langle\det(E^{1,0})\rangle$ is a direct factor of
$\bigwedge^\ell(E,\theta)$.
\end{proposition}

\subsection{Type B: $S^m(\Omega_\iota)$ stable for all $m$ and not invertible}
In this case we do not know whether the factor $M_\iota$ of the universal covering $\tilde{U}$ corresponding to $\Omega_\iota$ is a bounded domain, and 
the Arakelov equality just implies that certain numerical and stability conditions are equivalent.
\begin{proposition}\label{numcondB}
Let $\V$ be an irreducible non-unitary complex polarized variation of Hodge structures of weight $1$, pure of type A or B, with unipotent local monodromy at infinity, and with Higgs bundle $(E,\theta)$. 
Assume that $\V$ satisfies the Arakelov equality. 
Consider the following conditions:
\begin{enumerate}
\item[a.] $E^{1,0}$ and $E^{0,1}$ are $\mu$-stable.
\item[b.] ${E^{1,0}}^\vee \otimes E^{0,1}$ is $\mu$-polystable.
\item[c.] The saturated image of $T_\iota\to \sH om(E^{1,0},E^{0,1})$ is a direct factor of the sheaf
$\sH om(E^{1,0},E^{0,1})$.
\item[d.] The Higgs bundle $\langle\det(E^{1,0})\rangle$ is a direct factor of the Higgs bundle
$\bigwedge^\ell(E,\theta)$.
\item[e.] $\mu(\langle\det(E^{1,0})\rangle)=0$.
\item[f.] $\varsigma((E,\theta))=\frac{\ell\cdot\ell'\cdot(n_\iota+1)}{(\ell + \ell')\cdot n_\iota}.$  
\end{enumerate}
Then:
\begin{enumerate}
\item[i.] The conditions c), d), e), and f) are equivalent and they imply that 
$M_\iota$ is a complex ball of dimension $n_\iota$. 
\item[ii.] The condition b) implies c).
\item[iii.] Whenever condition ($\star$) in Lemma~\ref{poly} is satisfied, for example if $U$ is projective or of dimension one, a) implies b).
\end{enumerate}
\end{proposition}
If $\V$ is pure of type A, we know that the conditions a), d), and f) automatically hold true.
Nevertheless we included this case in the statement , since we will later refer 
to the equivalence between c) and f).
\par
\begin{proof}[Proof of Proposition~\ref{numcondB}]
The stability of $E^{1,0}$ implies the one for ${E^{1,0}}^\vee$, and hence b) follows from a) and 
from ($\star$). 
\par
For part ii) remark that the Arakelov equality says that 
$$
\mu(T_\iota) = \mu(\sH om(E^{1,0},E^{0,1})).
$$
So c) is a consequence of b).

By Simpson's correspondence d) and e) are equivalent, and by Lemma~\ref{upper} the numerical 
condition in f) is equivalent to d). So for i) it remains to verify the equivalence of c) and d).
\begin{claim}\label{cl2}
The condition d) implies c).
\end{claim}
\begin{proof}
The inclusion $\langle \det(E^{1,0})\rangle^{\ell-1,1} \to E^{\ell-1,1}$ is given by
$T_\iota\to \sH om(E^{1,0},E^{0,1})$, tensorized with $\det(E^{1,0})$.
So Condition d) implies that the saturated image of $\langle \det(E^{1,0})\rangle^{\ell-1,1}$ is a direct factor of $E^{\ell-1,1}= {E^{1,0}}^\vee  \otimes E^{0,1} \otimes \det(E^{1,0})$, hence that c) holds.
\end{proof}
\begin{remark}\label{corr}
The implication `c) implies d)' has been claimed in \cite[page 327]{VZ07} in a more special
situation. There however, as pointed out by the referee of the present article, the argument is not complete. We did not verify that the image of $\Phi_{m+1}\circ\theta^{\ell-m,m}$ really lies in $\langle\det(E^{1,0})\rangle^{\ell-m,m}$. This can easily be done, using Claim~\ref{co.3} below and the property ($**$) on page 294 of \cite{VZ07}. Here, without using the last condition, we will work out
the argument in details without further reference to \cite{VZ07}. 
\end{remark}
Let us write 
$$
E^{\ell-m,m}=\bigwedge^{\ell-m}({E^{1,0}})\otimes \bigwedge^m(E^{0,1})\cong
\bigwedge^{m}({E^{1,0}}^\vee)\otimes \bigwedge^m(E^{0,1})\otimes  \det(E^{1,0}).
$$
Using the right hand isomorphism we will regard $E^{\ell-m,m}$ as a subsheaf of 
$$S^m({E^{1,0}}^\vee\otimes E^{0,1})\otimes  \det(E^{1,0}).$$ 
Then the dual Higgs field $\theta_{\ell-m,m}^\vee:E^{\ell-m,m}\otimes T_\iota \to  E^{\ell-m-1,m+1}$
is given by a quotient of the multiplication map
$$
S^m({E^{1,0}}^\vee\otimes E^{0,1})\otimes ({E^{1,0}}^\vee\otimes E^{0,1}) \otimes  \det(E^{1,0}) \>>>
S^{m+1}({E^{1,0}}^\vee\otimes E^{0,1})\otimes  \det(E^{1,0}).
$$
restricted to $E^{\ell-m,m}\otimes T_\iota$. Since the slope is additive for tensor products
$\mu(E^{\ell-m,m})$ is equal to $(\ell-m) \cdot \mu(E^{1,0}) + m \cdot \mu(E^{0,1})$.
The Arakelov equality implies that
\begin{equation}\label{araeq}
\mu(E^{\ell-m,m}) =
m \cdot \mu(T_\iota) + \ell \cdot \mu(E^{1,0})= m \cdot \mu(T_\iota) + \mu(\det(E^{1,0})).
\end{equation}
\begin{claim}\label{co.3}
Let $V$ be a $\mu$-semistable subsheaf of $E^{\ell-m,m}$ of slope of $\mu(E^{\ell-m,m})$.
Assume that for some $b>0$ there exists a morphism 
$$
E^{\ell-m-1,m+1} \>>>  S^{m+1}(T_\iota)^{\oplus b}\otimes  \det(E^{1,0})
$$ 
such that the composite
$$
\gamma'_m:V\otimes T_\iota \> \subset >> E^{\ell-m,m}\otimes T_\iota \>\theta_{\ell-m,m}^\vee>>
E^{\ell-m-1,m+1} \>>>  S^{m+1}(T_\iota)^{\oplus b}\otimes  \det(E^{1,0})
$$ 
is surjective up to $\mu$-equivalence, as defined in Definition~\ref{fil.2}. Then there exists a morphism 
$$
E^{\ell-m,m} \to S^{m}(T_\iota)^{\oplus b}\otimes  \det(E^{1,0}),
$$ 
whose restriction $\gamma_{m}:V \to S^{m}(T_\iota)^{\oplus b}\otimes  \det(E^{1,0})$ induces $\gamma'_m$,
in the sense that $\gamma'_m$ is the composite of $\gamma_{m}\otimes {\rm id}_{T_\iota}$ with the multiplication map 
$$
S^{m}(T_\iota)^{\oplus b}\otimes T_\iota\otimes  \det(E^{1,0}) \>>> S^{m+1}(T_\iota)^{\oplus b}\otimes  \det(E^{1,0}).
$$ 
In particular $\gamma_{m}$ is again surjective up to $\mu$-equivalence.
\end{claim}
\begin{proof}
The morphism $\gamma'_m$ is generically surjective, hence one has a generically surjective morphism
$$
\gamma'_m\otimes {\rm id}_{\Omega_\iota} :V\otimes T_\iota \otimes \Omega_\iota \>>> S^{m+1}(T_\iota)^{\oplus b}
\otimes \Omega_\iota\otimes  \det(E^{1,0}),
$$
factoring through $E^{\ell-m,m}\otimes T_\iota \otimes \Omega_\iota$.
Restricting to $V\subset V\otimes T_\iota \otimes \Omega_\iota$ and composing with the natural contraction map 
$$
\alpha_m:S^{m+1}(T_\iota)^{\oplus b}\otimes \Omega_\iota \otimes  \det(E^{1,0})\>>> S^{m}(T_\iota)^{\oplus b}\otimes  \det(E^{1,0})
$$
one gets 
$\gamma_{m} : V  \>\subset >> E^{\ell-m,m} \>>>  S^{m}(T_\iota)^{\oplus b}\otimes  \det(E^{1,0}).$
By construction $\gamma'_m$ is the restriction of the composite
\begin{multline*}
V\otimes T_\iota \> \subset >> 
V\otimes T_\iota \otimes \Omega_\iota\otimes T_\iota
\> \gamma'_m \otimes {\rm id}_{\Omega_\iota\otimes T_\iota}>>\\
 S^{m+1}(T_\iota)^{\oplus b} \otimes \Omega_\iota\otimes T_\iota\otimes  \det(E^{1,0})
\> {\rm id}_{S^{m+1}(T_\iota)^{\oplus b}}\otimes \alpha\otimes {\rm id}_{ \det(E^{1,0})} >> S^{m+1}(T_\iota)^{\oplus b} \otimes  \det(E^{1,0}),
\end{multline*}
where $\alpha: T_\iota \otimes \Omega_\iota \to \sO_Y$ denotes again the contraction map. 
The last of the morphisms is up to the tensor product with the identity on $\det(E^{1,0})$ a direct sum of
morphisms factoring like
$$
S^{m+1}(T_\iota) \otimes \Omega_\iota\otimes T_\iota \> \alpha_m\otimes {\rm id}_{T_\iota}>>
S^{m}(T_\iota)\otimes T_\iota \> {\rm mult} >> S^{m+1}(T_\iota).
$$
So one obtains $\gamma'_m$ as the composite of $\gamma_{m}$ with the multiplication map.
In particular $\gamma_{m}$ is the direct sum of non-zero morphisms and the stability of $S^{m}(T_\iota)$ implies that the image of $\gamma_{m}$ is $\mu$-equivalent to $S^{m}(T_\iota)^{\oplus b}\otimes  \det(E^{1,0})$.
\end{proof}
Let us return to the notations introduced in the first part of Section
~\ref{spl}. In particular ${\rm NS}_0$ denotes the
subgroup of the Neron-Severi group $NS(Y)_\Q$ generated by prime-divisors $D$ with
$\mu(\sO_Y(D))=0$, and $U'$ is the complement of those prime-divisors.  

Let us write $\sS'^{\ell-m,m}$ for the cosocle of $E^{\ell-m,m}$. As remarked in
the Example and Definition~\ref{spl.1c} it is a $\mu$-polystable 
sheaf of slope $\mu(E^{\ell-m,m})$ of maximal rank, for which there exists a morphism
$\theta:E^{\ell-m,m} \to\sS'^{\ell-m,m}$, which is surjective over some open set.
Using parts vii) and ix) of the Addendum~\ref{spl.1a} one finds that 
$\theta$ is surjective over $U'$. 

Let $\sS^{\ell-m,m}$ be the direct sum of all direct factors of $\sS'^{\ell-m,m}$, which are $\mu$-equivalent to the $\mu$-stable sheaf
$S^m(T_\iota)\otimes \det(E^{1,0})$. Remark that $\sS^{\ell-m,m}$ is not unique.
By Addendum~\ref{spl.1a}~vii) we may choose an effective divisor
$B_m\in {\rm NS}_0$ such that for some $b_m$ 
$$
S^m(T_\iota)^{\oplus b_m}\otimes \det(E^{1,0})\hookrightarrow \sS^{\ell-m,m}=
S^m(T_\iota)^{\oplus b_m}\otimes \det(E^{1,0})\otimes \sO_Y(B_m).
$$ 
In particular both sheaves are $\mu$-equivalent. Let us denote the induced morphism by  
$\beta_{m} : E^{\ell-m,m} \to \sS^{\ell-m,m}$.

As a next step, we will show by induction on $m$, that for a suitable choice of the divisors $B_m$ the dual Higgs field $\theta^\vee_{\ell-m,m}$ defines a morphism
$\sS^{\ell-m,m}\to {\sS}^{\ell-m-1,m+1}$. The induction step is given by: 

\begin{claim}\label{co.4} We assume that c) holds (and of course the Arakelov equality). 
Then choosing the effective divisor $B_{m+1}\in {\rm NS}_0$ and hence ${\sS}^{\ell-m-1,m+1}$ large enough, there exists a commutative diagram 
\begin{equation}\label{split3}
\begin{CD}
E^{\ell-m,m}\otimes T_\iota \> \beta_m\otimes {\rm id}_{T_\iota} >>
\sS^{\ell-m,m}\otimes T_\iota \\
\V \theta^\vee_{\ell-m,m} V V  \V V {\tau}^\vee V \\
E^{\ell-m-1,m+1} \> \beta_{m+1} >>
{\sS}^{\ell-m-1,m+1}.
\end{CD}
\end{equation}
\end{claim}
The morphism ${\tau}^\vee$ has an explicit description. For simplicity
we just formulate this on the open subscheme $U'$. Part vii) of Addendum~\ref{spl.1a}
allows to extend this description to the boundary, perhaps after replacing
$B_{m+1}$ by a larger divisor in ${\rm NS}_0$.
\begin{claim}\label{co.4a} 
For some morphism $\tau'_m: S^{m}(T_\iota)^{\oplus b_m} \to S^{m}(T_\iota)^{\oplus b_{m+1}}$ the morphism $\tau^\vee|_{U'}$ is the composite of $\tau'_m\otimes {\rm id}_{T_\iota}|_{U'}$ and the direct product of $b_{m+1}$ copies of the multiplication map $S^{m}(T_\iota)\otimes T_\iota|_{U'} \to S^{m+1}(T_\iota)|_{U'}$.
\end{claim}
\begin{proof}[Proof of the Claims~\ref{co.4} and~\ref{co.4a}] By the Arakelov equality, as restated in~\ref{araeq} and by the choice of the sheaves $\sS^{\ell-\bullet,\bullet}$
 the four sheaves in~\ref{split3} all have the same slope and they are all $\mu$-semistable. By Addendum~\ref{spl.1a} for each of the morphisms
the image coincides with the saturated image over the open set $U'$. In particular
the restriction of $\beta_m$ and $\beta_{m+1}$ to $U'$ is surjective.

Writing $V_m$ for the kernel of $\beta_m$, hence $V_m\otimes T_\iota$ for the one of
$\beta_m\otimes {\rm id}_{T_\iota}$, consider the image $\sI$ of $V_m\otimes T_\iota$ 
under $\theta^\vee_{\ell-m,m}$. We claim that $\sI$ is contained in $V_{m+1}$.

If not $\beta_{m+1}\circ \theta^\vee_{\ell-m,m}({V_m\otimes T_\iota})$ is a non-zero subsheaf of ${\sS}^{\ell-m-1,m+1}$. By Addendum~\ref{spl.1a},~ix) 
its saturated hull is a $\mu$-semistable subsheaf of ${\sS}^{\ell-m-1,m+1}$. Since both are of the same slope, and since the second one is $\mu$-polystable, the saturated image has to be a direct factor, hence isomorphic to $S^{m+1}(T_\iota)^b\otimes \det(E^{1,0})\otimes \sO_Y(B_{m+1})$ for some $b>0$.

By Claim~\ref{co.3} one obtains a morphism $E^{\ell-m,m}\to S^{m}(T_\iota)^b\otimes \det(E^{1,0})\otimes \sO_Y(B_{m+1})$ whose restriction to $V_m$ is non-zero. Obviously this contradicts the definition of ${\sS}^{\ell-m-1,m+1}$ as a maximal $\mu$-polystable quotient and of $V_m$ as the kernel of $\beta_m$.

The restriction of $\beta_m\otimes {\rm id}_{T_\iota}$ to $U'$ is surjective.
Since $\theta^\vee_{\ell-m,m}(\sI)\subset V_{m+1}$, the morphism 
$\tau^\vee$ exists on $U'$, and enlarging $B_{m+1}$ it extends to $Y$.

In order to get the explicit description stated in Claim~\ref{co.4a}, we apply Claim~\ref{co.3} to $V=E^{\ell-m,m}$. The image of $\gamma'_m=\beta_{m+1}\circ \theta^\vee_{\ell-m,m}$ is a $\mu$-semistable subsheaf of the $\mu$-polystable sheaf ${\sS}^{\ell -m-1,m+1}$, hence $\mu$-equivalent to a direct factor of the form 
$S^{m+1}(T_\iota)^{\oplus b}\otimes \det(E^{1,0})$. Claim~\ref{co.3} implies that for some
$$
\gamma_{m}: E^{\ell-m,m}  \to S^{m}(T_\iota)^{\oplus b}\otimes  \det(E^{1,0})
$$ 
the morphism $\beta_{m+1}\circ \theta^\vee_{\ell-m,m}$ is the composite of
$\gamma_m\otimes {\rm id}_{T_\iota}$ with the multiplication map. Since $\gamma_m$ 
factors through the cosocle $\sS'^{\ell-m,m}$ and hence through $\sS^{\ell-m,m}$, one finds the morphism $\tau'_m$. 
\end{proof}
Recall that $\langle \det(E^{1,0})\rangle$ is the saturated subsheaf
of $E^{\ell-m,m}$ which is generated by $\det(E^{1,0})$. If non-zero $\langle\det(E^{1,0})\rangle^{\ell-m,m}$
contains $S^m(T_\iota)\otimes \det(E^{1,0})$ and both are $\mu$-equivalent.
As a next step we will show that $\langle\det(E^{1,0})\rangle^{\ell-m,m}|_{U'}$ is a direct factor of $E^{\ell-m,m}|_{U'}$.
\begin{claim}\label{co.5} Assume c). Then the composite
$$
\langle\det(E^{1,0})\rangle^{\ell-m,m}\> \subset >> E^{\ell-m,m} \> \beta_m >> \sS^{\ell-m,m}
$$
is injective and defines a splitting of the inclusion $\langle\det(E^{1,0})\rangle^{\ell-m,m}|_{U'}\subset E^{\ell-m,m}|_{U'}$. 
\end{claim}
\begin{proof} 
If $\langle\det(E^{1,0})\rangle^{\ell-m,m}=0$ there is nothing to show. Otherwise
by the equality~\ref{araeq} $\mu(\langle\det(E^{1,0})\rangle^{\ell-m,m})=\mu(E^{\ell-m,m})$ and
by part ix) of the Addendum~\ref{spl.1a} $\langle\det(E^{1,0})\rangle^{\ell-m,m}$ is a $\mu$-semistable subsheaf of $E^{\ell-m,m}$, containing 
$S^m(T_\iota)\otimes \det(E^{1,0})$ as a $\mu$-equivalent subsheaf.

Recall that $T_\iota$ is a direct factor of ${E^{1,0}}^\vee\otimes E^{0,1}$, and hence
$S^m(T_\iota)$ a direct factor of $S^m({E^{1,0}}^\vee\otimes E^{0,1})$. This sheaf also contains
$\bigwedge^{m}({E^{1,0}}^\vee)\otimes \bigwedge^m(E^{0,1})$ as a direct factor.
Writing 
$$
S^m({E^{1,0}}^\vee\otimes E^{0,1})=\bigwedge^{m}({E^{1,0}}^\vee)\otimes \bigwedge^m(E^{0,1})\oplus R_m,
$$ 
consider the image of $S^m(T_\iota)$ under the projection $S^m({E^{1,0}}^\vee\otimes E^{0,1})\to R_m$.
If this is zero we are done. If not one has an injection
$$
\alpha':S^m(T_\iota) \oplus S^m(T_\iota) \>>> \bigwedge^{m}({E^{1,0}}^\vee)\otimes \bigwedge^m(E^{0,1})\oplus R_m,
$$
where the first factor maps to 
$\bigwedge^{m}({E^{1,0}}^\vee)\otimes \bigwedge^m(E^{0,1})$ and the second one to $R_m$.
For both factors the composite with the projection 
 $$
\bigwedge^{m}({E^{1,0}}^\vee)\otimes \bigwedge^m(E^{0,1})\oplus R_m
\longrightarrow S^m(T_\iota)
$$  
is non-zero, hence by $\mu$-semistability it is surjective up to $\mu$-equivalence. So $\alpha'$ splits, and $S^m(T_\iota)$ as the image of the composite of $\alpha'$ with the projection to $\bigwedge^{m}({E^{1,0}}^\vee)\otimes \bigwedge^m(E^{0,1})$, splits as well. 

Since $\langle\det(E^{1,0})\rangle^{\ell-m,m}|_{U'}$ is defined as the image
of $S^m(T_\iota)\otimes \det(E^{1,0})|_{U'}$ in $E^{\ell-m,m}|_{U'}=\bigwedge^{m}({E^{1,0}}^\vee)\otimes \bigwedge^m(E^{0,1})\otimes \det(E^{1,0})|_{U'}$, it is a direct factor, hence its image in the cosockle is 
non-zero. By the choice of $\sS^{\ell-m,m}$ we are done.
\end{proof}
\begin{claim}\label{cl4}
The condition c) implies d).
\end{claim}
\noindent
{\it Proof.}
Writing $\tau$ for the composite of $\tau^\vee\otimes \Omega_\iota$ with the contraction
to $\sS^{\ell-m-1,m+1}$ one obtains by Claim~\ref{co.4} a Higgs bundle 
$$
(\sS,\tau)=\Big(\bigoplus_{m=0}^\ell \sS^{\ell-m,m},\bigoplus_{m=0}^{\ell-1} 
\big(\sS^{\ell-m,m} \> \tau >> \sS^{\ell-m,m}\otimes \Omega_\iota\ \big) \Big)
$$ 
together with a map of Higgs bundles 
\begin{equation}\label{hgbdl}
\bigwedge^\ell(E,\theta)=\big(\bigoplus_{m=0}^\ell E^{\ell-m,m},\theta \big) \> \beta >> (\sS,\tau).
\end{equation}
For  $\varsigma= \varsigma((E,\theta))$ the sheaf 
$$
\langle\det(E^{1,0})\rangle=\bigoplus_{m=0}^\varsigma \langle\det(E^{1,0})\rangle^{\ell-m,m}= \bigoplus_{m=0}^\ell \langle\det(E^{1,0})\rangle^{\ell-m,m}
$$ 
is a Higgs subbundle of the left hand side of~\ref{hgbdl}, hence its saturated image
$$
(\widetilde{\langle\det(E^{1,0})\rangle},\tau|_{\widetilde{\langle\det(E^{1,0})\rangle}})
$$
in the right hand side is a Higgs subbundle of $(\sS,\tau)$. By Claim~\ref{co.5}
the induced map 
$$
\langle\det(E^{1,0})\rangle^{\ell-m,m}\>>> \widetilde{\langle\det(E^{1,0})\rangle}^{\ell-m,m}
$$
is injective and both sheaves are $\mu$-equivalent. Since $\langle\det(E^{1,0})\rangle^{\ell-m,m}$
is $\mu$-stable, $\widetilde{\langle\det(E^{1,0})\rangle}^{\ell-m,m}$ is $\mu$ equivalent
to a direct factor $\widehat{\langle\det(E^{1,0})\rangle}^{\ell-m,m}$ of $\sS^{\ell-m,m}$. By the explicit description of $\tau^\vee$ in Claim~\ref{co.4a}
the Higgs field $\tau$ respects the splitting, and one obtains a quotient Higgs bundle 
$\widehat{\langle\det(E^{1,0})\rangle}$ of $(\sS,\tau)$, hence of $\bigwedge^\ell  (E,\theta)$.
Since $\langle\det(E^{1,0})\rangle$ is a sub-Higgs bundle Lemma~\ref{spl.2} implies that $\mu(\langle\det(E^{1,0})\rangle) \leq 0$, and since $\widehat{\langle\det(E^{1,0})\rangle}$ is a quotient-Higgs bundle, $\mu(\widehat{\langle\det(E^{1,0})\rangle}) \geq 0$. So the $\mu$-equivalence of all direct factors implies that 
$$
\mu(\langle\det(E^{1,0})\rangle)=\mu(\widehat{\langle\det(E^{1,0})\rangle})=0.
$$
Using Lemma~\ref{spl.2} again one finds that $\langle\det(E^{1,0})\rangle$ splits as a Higgs subbundle of $\bigwedge^\ell(E,\theta)$.
\end{proof}
To finish the proof of Proposition~\ref{numcondB} it remains to verify:
\begin{claim}\label{ball}
The splitting in d) implies that $M_\iota$ is an $n_\iota$-dimensional complex ball.
\end{claim}
\begin{proof} (See also \cite[Section 5]{VZ07})
The Higgs bundle $\langle\det(E^{1,0})\rangle$ splits as a sub-Higgs bundle of $\bigwedge^\ell E$, hence it is itself a Higgs bundle arising from a local system. In particular the Chern classes $\ch_1(\langle\det(E^{1,0})\rangle)$ and $\ch_2(\langle\det(E^{1,0})\rangle)$ are zero. 

Assume for a moment that there exists an invertible sheaf $\sL$ with $\det(E^{1,0})=\sL^\ell$, and 
consider the Higgs bundle
$$
\big( F=F^{1,0}\oplus F^{0,1}= \sL \oplus \sL \otimes T_\iota,
\sL \to \sL \otimes T_\iota\otimes \Omega_\iota\big).
$$
Then $S^\ell(F)$ is a Higgs bundle with $\sL^\ell\otimes S^m(T_\iota)$ in bidegree $(\ell-m,m)$, hence
isomorphic to $\langle\det(E^{1,0})\rangle$. The first Chern class of $\langle\det(E^{1,0})\rangle$ is zero, hence $\ch_1(F)$ as well. On the other hand,
$$
\ch_1(F)= \ch_1(\sL) + n_\iota\cdot \ch_1(\sL) - \ch_1(\Omega_\iota)=
\frac{n_\iota+1}{\ell}\ch_1(E^{1,0}) - \ch_1(\Omega_\iota),
$$
and $\displaystyle\ch_1(\sL)=\frac{1}{n_\iota+1}\ch_1(\Omega_\iota)$. For the second Chern class
it is easier to calculate the discriminant
$$
\Delta(\sF)=2 \cdot \rk(\sF) \cdot \ch_2(\sF) - (\rk(\sF)-1)\cdot \ch_1(\sF)^2.
$$
By \cite[Lemma 3.3]{VZ07}, a), the discriminant is invariant under tensor products with invertible sheaves, 
hence $\Delta(\sL\oplus \sL\otimes T_\iota) = \Delta(\sO_Y \oplus T_\iota)$.

Since $\ch_1(\langle\det(E^{1,0})\rangle)^2=\ch_2(\langle\det(E^{1,0})\rangle)=0$ one finds $\Delta(\langle\det(E^{1,0})\rangle)=0$, and \cite[Lemma 3.3]{VZ07}
implies that $\Delta(F)=0$, hence
\begin{equation}\label{chern}
0=\Delta(\sO_Y \oplus T_\iota)=2 \cdot (n_\iota +1) \cdot \ch_2(T_\iota)
- n_\iota \cdot \ch_1(T_\iota)^2.
\end{equation}
In general on may choose a finite covering $\sigma:Y'\to Y$ such that $\sigma^*(\det(E^{1,0})=\sL'^\ell$
for some invertible sheaf $\sL'$. Repeating the calculations of Chern classes with
$T_\iota$ replaced by $T'_\iota=\sigma^*(T_\iota)$ one obtains that  
$2 \cdot (n_\iota +1) \cdot \ch_2(T'_\iota) - n_\iota \cdot \ch_1(T'_\iota)^2=0$ and again
\eqref{chern} holds true. 

By Yau's Uniformization Theorem, recalled in \cite[Theorem 1.4]{VZ07},
\eqref{chern} implies that $M_\iota$ is a complex ball. 
\end{proof}
The Proposition~\ref{numcondB} gives a numerical condition on the length of the wedge product of the Higgs field which, together with the Arakelov equality, implies that $M_\iota$ is a complex ball. A similar condition holds automatically for local systems which are pure of type A. This is not surprising, since in this case the corresponding factor $M_\iota$ automatically is a $1$-dimensional ball. 
\par
In slight abuse of notation we say that a local system
$\V$ is given by a wedge product of the standard representation of  $\SU(1,n)$, 
if the representation defining $\V$ factors through
one of the standard wedge product representations (e.g.\ \cite{Sa80}, p.~461)
$$\bigwedge^k: \ \SU(1,n) \to SU\left(\left(\begin{matrix} n \\ k-1 \end{matrix}\right),
\left(\begin{matrix} n \\ k \end{matrix}\right)\right)$$
and if, moreover, the period map for $\V$ factors through the (totally geodesic)
embedding of symmetric spaces attached to $\bigwedge^k$. In different terms, 
for $k=1$ the corresponding Higgs field is given by
$$
E^{1,0}=\omega_i^{-\frac{1}{n_i+1}}\otimes \Omega_i, \ \ \ 
E^{0,1}=\omega_i^{-\frac{1}{n_i+1}} \mbox{ \ \ and \ \ }
\theta={\rm id}: \omega_i^{-\frac{1}{n_i+1}}\otimes \Omega_i \>>> 
\omega_i^{-\frac{1}{n_i+1}}\otimes \Omega_i,
$$ 
where $\omega_i^{-\frac{1}{n_i+1}}$ stands for an invertible sheaf, whose 
$(n_i+1)$-st power is $\det(\Omega_i)$.
\par
\begin{proposition}\label{kugaAB}
Let $\V$ be an irreducible complex polarized variation of Hodge structures of weight $1$, pure of type $\iota$, with unipotent local monodromy at infinity, and with Higgs bundle $(E,\theta)$. Assume that $\Omega_\iota$ is of type A or B, and that the saturated image of $T_\iota\to \sH om(E^{1,0},E^{0,1})$ splits.
\par
Then $\V$ is the tensor product of a unitary representation with a wedge product of the standard representation of $\SU(1,n)$. In particular the period map $\tau:\tilde{U} \to M'$ factors as the projection $\tilde{U} \to M_\iota$
and a totally geodesic embedding $M_\iota \to M'$. 
\end{proposition}
\begin{proof}
Proposition~\ref{numcondB},~i) implies that $M_\iota$ is a complex ball.
\par  
Before we proceed, we fix some notation. For a simply connected
complex space we denote by $\Aut(M)$ the group of biholomorphic
self-mappings of $M$. This coincides with the definition 
in Section~\ref{verify_arakelov} if $M$ is a hermitian 
symmetric domain. We write,
as usual, $\tilde{U} = \prod_k M_k$ and fix origins $o_k$ in all $M_k$. Let
$G_k:= \Aut(M_k)$, $K_k := \Stab(o_k)$. Thus for hermitian symmetric
domains $M_k$ we have hence $M_k = G_k/K_k$.
\par
Let $\tau: \tilde{U} \to M'$ be the period map for the bundle $\V$. In $M'$
fix an origin $o'$, let $G':=\Aut(M') \cong \SU(\ell,\ell')$, 
and let $K':=\Stab(o')$.
By the purity of the Higgs bundle, $\tau$ factors
as the projection $ \tilde{U} \to M_\iota$ composed with a map
$\tau_1: M_\iota \to M'$. 
\par
The next claim derives the second statement from the main hypothesis.
Remember that,
since the splitting $T_\iota \to \sH om(E^{1,0},E^{0,1})$ comes 
from a splitting of Higgs bundles, 
it is orthogonal for the Hodge metric, hence for the K\"ahler metric.
\par
\begin{claim}\label{Mokarg}
Let $\tau_1:M_\iota \to M'=G'/K'$ be a holomorphic
map to a hermitian symmetric domains. 
Assume that $\tau_1^* T_{M'}= T_{M_\iota}\oplus R$ is a holomorphic splitting, 
orthogonal with respect to the K\"ahler metric on $M'$. 
Then $M_\iota \to M'$ is a totally geodesic embedding
and $M_\iota$ a symmetric domain.
\end{claim}
\par
\begin{proof} (From a letter by N. Mok.)
First, the splitting condition on $\tau_1^* T_{M'}$ implies that
$\tau_1$ is locally an embedding. Second, we check that the image
$\tau_1(M_\iota)$ is totally geodesic in $M'$. This is again a local
condition. By \cite[Theorem~I.14.5]{He62} it suffices to check that 
the splitting 
$T_{M'}|_{\tau_1(M_\iota)} = T_M\oplus R$ is preserved under parallel transport.
\par
Take any local holomorphic sections $s$ of $T_{M_\iota}$ and $t$ of $R$.  
Then $\langle s,t \rangle = 0$ with respect to 
the Hermitian inner product. The derivative of $t$ with respect to a 
$(1,0)$ vector is orthogonal to $s$ because $s$ is holomorphic 
and because $\langle \cdot,\cdot \rangle$ is Hermitian bilinear. 
Since $s$ and $t$ are arbitrary, it follows that $R$ is invariant 
under differentiation in the $(1,0)$ direction.  But since $R$ is a 
holomorphic subbundle, it is invariant under differentiation in the 
$(0,1)$-direction.  As a consequence $R$ is parallel, and hence
its orthogonal complement $T_{M_\iota}$ is parallel, too.
\par
Finally, since $M'$ is a global symmetric domain, it has geodesic
symmetries at each point of $\tau_1(M_\iota)$. Since $\tau_1(M_\iota)$ is
totally geodesic in $M'$, these are geodesic symmetries of $\tau_1(M_\iota)$.
Consequently, $\tau_1(M_\iota)$ is a global symmetric domain and
$\tau_1$ is (globally) an embedding.
\end{proof}
\par
We continue with the proof of Proposition~\ref{kugaAB} and let 
$$B:= \{\varphi \in \Aut(M'): \varphi(\tau_1(M_\iota)) = 
\tau_1(M_\iota)\} \subset G'.$$
In the next step we deduce from Claim~\ref{Mokarg} that  
$\tau_1(M_\iota) = B/K_B$,
where $K_B$ is a maximal compact subgroup.
The first observation is:
\par
\begin{claim}\label{nouniversalcov}
The embedding $\tau_1: M_\iota\to M'$ is induced by a homomorphism from 
$\tilde{\tau_1}: G_\iota \to G'$ that factors through $B$. 
\end{claim}
\begin{proof}
As explained in \cite[\S 1.1]{Sa65} or \cite[II \S 2]{Sa80}, 
the geodesic holomorphic embedding $M_\iota \to M'$ is induced by a local 
isomorphism $G_\iota \to G'$ and hence a homomorphism of Lie algebras 
$\Lie(G_\iota)\to \Lie(G')$. This induces a 
homomorphism $\hat\tau_1: \tilde{G_\iota}\to G'$ from the universal covering 
$\tilde{G_\iota}$ of $G_\iota$. It remains to show that $\hat\tau_1$
factors through $G_\iota$, then the factorization through $B$ 
is obvious from the definition.
\par
It suffices to exhibit a factorization of $\hat\tau_1$ on
the $\R$-valued points. Since 
$$G_\iota(\R) \subset G_\iota(\C) \cong {\rm Sl}(1+n)(\C),$$
and since  ${\rm Sl}(1+n)(\C)$ is simply connected, this
factorization is obvious.
\end{proof} 
\par
By this claim, the natural map ${\rm res}: B \to \Aut(M_\iota) 
\cong G$ induces
a surjection $B/K_B \to M_\iota$. This map is also injective, 
since elements in $B$
preserve $M_\iota$. Consequently, the kernel $\Upsilon$ of 
${\rm res}$ is a compact
subgroup. By Claim~\ref{nouniversalcov} again, this kernel is a direct factor.
In fact, the kernel is a maximal direct factor, since $G_\iota = 
\Aut(M_\iota)$ does not contain direct
compact factors. We deduce that given the choice of origins, 
the product decomposition $B \cong G_\iota \times  \Upsilon$ is canonical.
\par
By definition of a period map, $\tau$ 
is equivariant with respect to the action of $\pi_1(U)$ via
$$\rho_1: \pi_1(U) \>>> \Aut(\tilde{U}) \quad
\text{and} \quad \rho_2: \pi_1(U) \>>> \Aut(M') \cong G' $$
on domain and range. 
\par
The image of $\rho_2$ lies in $B$ by definition. The usual
argument with Schur's Lemma (e.g.\ Proposition~5.9 or 
\cite{VZ05} Proposition~3.3) implies that $\rho_2$ is a
tensor product of a unitary representation and of a representation
that factors through $\tilde{\tau_\iota}: G_\iota \to B \to G'$. 
\par
The last step in the proof of Proposition~\ref{kugaAB} is to
exploit that there are not many possibilities for 
$\tilde{\tau_1}$ that give rise to a holomorphic totally geodesic
embedding of hermitian symmetric domains.
\par
\begin{claim}\label{co.9}
The representation $G_\iota \to B \to G'$ is a wedge product of the 
standard representation.
\end{claim}
\noindent
{\it Proof.} In order to match the hypothesis of \cite{Sa80} precisely, we
should postcompose the map $\tilde{\tau_1}: G_\iota \to G'$ 
by a natural inclusion of $G'$ into the symplectic group. By the table 
p.~461 and Proposition~1 in \cite{Sa80}, ${\rm incl} \circ \tilde{\tau_1}$
is a direct sum of wedge products of the standard representations. This 
direct sum has only one summand, since $\V$ was reducible otherwise.
\end{proof}
In order to prove the missing part v) of Proposition~\ref{shimuraprop} we will use:
\begin{proposition}\label{stabAB}
Assume that $U$ is the quotient of a bounded symmetric domain by an arithmetic group.
Assume that $\Omega_\iota$ is of type A or B, that $M_\iota$ is the complex ball $\SU(1,n)/K$ and that $\V$ is the tensor product of a unitary representation with a wedge product of the standard representation of $\SU(1,n)$. 
\begin{enumerate}
\item[1.] Then $\V$ satisfies the Arakelov equality.
\item[2.] Let $Y'$ be a Mumford compactification. Writing $(E',\theta')$ for the
Higgs bundle of $\V$ on $Y'$, the sheaves $E'^{1,0}$ and $E'^{0,1}$ are $\mu$-stable
and ${E'^{1,0}}^\vee \otimes E'^{0,1}$ is $\mu$-polystable.
\end{enumerate}
\end{proposition}
\begin{proof}
Let $Y',S'$ be a Mumford compactification (see Section~\ref{verify_arakelov}). The bundles $E'^{1,0}$ and $E'^{0,1}$ are irreducible homogeneous bundles as in Lemma~\ref{DeligneSatake}, case $a_n$ and $q=1$, given by
the wedge products of the standard representation of $U(n)$. The same
arguments as in the proof of the first parts of Proposition~\ref{shimuraprop} now imply 1) and 2).
\par
By Lemma~\ref{AEindep} the Arakelov equality on a Mumford compactification implies
the one on any compactification, satisfying the positivity statement
in Assumption~\ref{compass-pos}.
\end{proof}
%
\subsection{Type C: $S^m(\Omega_\iota)$ is $\mu$-unstable for some $m >1$} 

Yau's Uniformization Theorem, recalled in \cite[Theorem 1.4]{VZ07}, implies that $M_\iota$ 
is a bounded symmetric domain of rank greater than one. Using the characteristic subvarieties, introduced by Mok
presumably one can write down an explicit formula for $\varsigma(\V)$. However we do not need this, since 
in this case the superrigidity theorems apply.
Recall the notations introduced at the beginning of the proof of Proposition~\ref{kugaAB}.
\par
\begin{proposition} \label{factor_type_C}
If $\V$ is pure of type $\iota$, 
The period map factors as the projection $\tilde{U} \to M_\iota$
and a totally geodesic embedding $M_\iota \to M'$.
\end{proposition}
\par
\begin{proof}
Purity of $\V$ implies that the period map factors through the
projection to $M_\iota$. In the case we treat, $M_\iota$
has rank greater than one, hence the metric rigidity theorems
of Mok and their generalizations due to To apply. More
precisely, let $h$ be the pullback the restriction of
the Bergman-metric on $M'$ to $M_\iota$. By purity
and since $M'$ is a bounded symmetric domain of non-compact
type, $h$ descends to a metric of semi-negative curvature
on the bundle $(\Omega_\iota)^\vee$ on $U$. Thus the
hypothesis of \cite[Theorem~4]{Mk87} are met, if one 
takes into account the arguments of To (\cite[Corollary~2]{To89}
and the subsequent remark) to extend from $U$ compact to
$U$ of finite volume. We conclude that up, to a constant
multiple, $h$ is the  Bergman-metric on $M_\iota$ and
$M_\iota \to M'$ a totally geodesic embedding. 
\end{proof}
\par
\begin{lemma} \label{lemma:geoprojections}
Let $U \to \sA_g$ be a generically finite map with $\tilde{U} = \prod M_i$.
Suppose that all for all irreducible summands $\V$ the period
map $\tau(\V)$ is either constant or the composition 
$\tau(\V) = \tau_{j} \circ p_{i(\V)}$ of a projection and
a totally geodesic embedding of $M_{i(\V)}$ to the period domain 
of $\V$. Then the universal covering
map $\tau: \tilde{U}\to \tilde{\sA}_g$ is a totally geodesic embedding.
\end{lemma}
\par
\begin{proof}
By Lemma~\ref{etale} the hypothesis `generically finite' implies that for each $i$ there
is at least one non-unitary summand $\V$ with $i=i(\V)$. The universal covering map
is, by definition, the product of the $\tau(\V)$ composed
with a block diagonal embedding $\prod_j M_j' \to \tilde{\sA}_g$.
Since the latter is totally geodesic for the Bergman metric,
the claim follows from the hypothesis on the $\tau(\V)$.
\end{proof}
\begin{proof}[Proof of Proposition~\ref{shimuraprop}]
Parts i)--iv) have been verified at the end of Section~\ref{verify_arakelov}. By assumption $f:A\to U$ is a Kuga fibre space, $\V$ is pure of type $i=i(\V)$, and $\Omega_{i}$ is of type B. In particular the assumption made in Proposition~\ref{stabAB} hold, and
on a suitable compactification the sheaf ${E^{1,0}}^\vee \otimes E^{0,1}$ is $\mu$-polystable.
So Proposition~\ref{numcondB} implies that $$
\varsigma((E,\theta))=\frac{\ell\cdot\ell'\cdot(n_i+1)}{(\ell + \ell')\cdot n_i}.
$$ 
Of course this equality is independent of the compactification.  
\end{proof}
\par
\begin{proof}[Proof of Addendum~\ref{characterizationadd}]
Part I) is Proposition~\ref{numcondA}, if one uses in addition the equivalence 
between f) and c) in Proposition~\ref{numcondB} and the 
Proposition~\ref{kugaAB}. Part II) is just repeating the conclusion of 
Proposition~\ref{factor_type_C}.

For Part III) remark first that the equivalence of the conditions $\beta$) and 
$\gamma$) is part of
Proposition~\ref{numcondB}. By Proposition~\ref{kugaAB} $\beta$) implies 
$\delta$) and $\eta$). 
\end{proof}
\par
\begin{proof}[Proof of Theorem~\ref{characterization}]
If for some \'etale covering $\tau:U'\to U$ the pullback family $f':A' \to U'$ is a Kuga fibre space, then the two conditions 1) and 2) on $U$ are equivalent to the same conditions on $U'$. So we may as well assume that
$f:A\to U$ is itself a Kuga fibre space. If $\V$ is a non-unitary 
irreducible subvariation of Hodge structures in $R^1f_*\C_A$, then part ii) of Proposition~\ref{shimuraprop} gives the Arakelov equality, and part i) implies that $\V$ is pure of type $i=i(\V)$.

If $\Omega_i$ is of type A or C, there is nothing to verify in 2). If $\Omega_i$ is of type B, Part v) of Proposition~\ref{shimuraprop} shows that 
$$
\varsigma(\V)=\varsigma((E,\theta))=\varsigma((E,\theta_i)) = \frac{\rk(E^{1,0})\cdot\rk(E^{0,1})\cdot(n_i+1)}{\rk(E) \cdot n_i}.
$$
\par
Assume now that the conditions 1) and 2) in Theorem~\ref{characterization},~b)
hold. Since $\varphi$ is generically finite, by the first part of Lemma~\ref{etale} one finds for each direct factor $\Omega^1_\iota$ of $\Omega_Y(\log S)$ some non-unitary subvariation of Hodge structures $\V$, which
is pure of type $\iota$.

If $\Omega_\iota$ is of type C, we find that the map
$\tilde{U} \to M'$ to the period domain $M'$ of $\V$ factors as the projection $\tilde{U} \to M_\iota$
and a totally geodesic embedding $M_\iota \to M'$. 

By Proposition~\ref{kugaAB} the same holds if $\Omega_\iota$ is of type A, 
or if it is of type B and if the condition 2) Theorem~\ref{characterization} holds.
\par
So all the hypothesis of the Lemma~\ref{lemma:geoprojections} are
met and $\tilde{U}\to \tilde{\sA}_g$ is a totally geodesic embedding, 
hence by Theorem~\ref{Moonen_product} there exists a Kuga fibre space $f':A'\to U'$ such that the image of
$U'$ in $\sA_g$ coincides with the image $\varphi(U)$. In particular this image is non singular.
By the second part of Lemma~\ref{etale} $\varphi:U\to \varphi(U)$ is \'etale, and replacing $U'$ by an \'etale covering we may assume that $U'$ dominates $U$.
\par
Finally the last statement in Theorem~\ref{characterization} follows from Corollary~\ref{Moonen_rigid}.
\end{proof}
\begin{proof}[Proof of Corollary~\ref{characterizationadd3}]
Since we assumed that $U\to \sA_g$ is generically finite and that the condition $\eta$) in 
Addendum~\ref{characterizationadd} holds for all irreducible non-unitary subvariations of Hodge structures,
the argument used at the end of the proof of Theorem~\ref{characterization} shows that the pullback $f':A'\to U'$
of $f:A\to U$ to some \'etale covering $U'$ of $U$ is a Kuga fibre space. Hence there exists a Mumford compactification $Y'$.  The condition $\delta$) allows to apply  Proposition~\ref{stabAB},~2) to obtain the 
conditions $\alpha$) and $\beta$) 
\end{proof}
\section{The Arakelov equality and the Mumford-Tate group}\label{aemt}

We keep the assumption that $U$ is the complement of a normal crossing divisor $S$ in 
a non-singular projective variety $Y$, that $\Omega_Y^1(\log S)$ is nef, and that $\omega_Y(S)$
is ample with respect to $U$. Let $f:A\to U$ be a family of polarized abelian varieties such that $R^1f_*\C_A$ has unipotent local monodromies at infinity, and such that the induced 
morphism $U\to \sA_g$ is generically finite.

If each irreducible subvariation of Hodge structures of $R^1f_*\C_A$ is either unitary or it satisfies the Arakelov equality and if in addition the second condition in Theorem~\ref{characterization} holds, we have 
shown in the last section that the induced morphism $U\to \sA_g$ is totally geodesic.
By Moonen's Theorem~\ref{Moonen_product} we know that $U$ is the base of a Kuga fibre space, and that it is the translate of a Shimura variety of Hodge type.
In particular this implies that the monodromy group $\Mon^0$ of $R^1f_*\C_A$ is normalized by the
complex structure, hence by the derived Hodge group $\MT(R^1f_*\C_A)^\der$. In this section we will verify the last property as a direct consequence of the Arakelov equality, without using the second condition in Theorem~\ref{characterization}, and we will determine the
invariant cycles under $\Mon^0$ explicitly. The final statement is given in 
Corollary~\ref{mtcor}.

In the first part of this section we will consider arbitrary complex polarized variations of Hodge structures $\V$ of weight $k$ on $U$, with unipotent local monodromy around the components of $S$, and we will write  $(E=\bigoplus_{m=0}^k E^{k-m,m},\theta)$ for the Higgs bundle. For $k>1$ and $\dim(U)>1$ there is not yet any concept of Arakelov inequality where maximality has as nice consequences as in weight one. We thus start with an ad hoc definition 
of what should be the maximal case and show that this condition is satisfied for
some variations of Hodge structures derived from  variations of Hodge structures
of weight one with Arakelov equality.
\par
\begin{definition}\label{aemt.1}
The Higgs bundle $(E,\theta)$ (or the variation of Hodge structures $\V$) satisfies the {\em Arakelov condition} if there exist integers $m_{\rm min} \leq m_{\rm max}$ with
\begin{enumerate}
\item[i.] $E^{k-m,m}\neq 0$ if and only if $m_{\rm min} \leq m \leq m_{\rm max}$.
\item[ii.] For $m_{\rm min} \leq m < m_{\rm max}$ the morphism
$\theta_{k-m,m}=\theta|_{E^{k-m,m}}$ is non-zero.
\item[iii.] For $m_{\rm min} \leq m \leq m_{\rm max}$ the sheaves $E^{k-m,m}$ are $\mu$-semistable of slope
\begin{equation}\label{aracond2}
\mu(E^{k-m,m})=
\mu(E^{k-m_{\rm min},m_{\rm min}})-
(m-m_{\rm min})\cdot \mu(\Omega^1_Y(\log S)).
\end{equation}
\end{enumerate}
\end{definition}
\begin{lemma}\label{aracondex} \
\begin{enumerate} 
\item[1.] If $\V$ is unitary and irreducible, it satisfies the Arakelov condition.
\item[2.] If $k=1$, if $\V$ is irreducible and if both, $E^{1,0}$ and $E^{0,1}$ are non-zero, then $\V$ satisfies the Arakelov condition if and only if the Arakelov equality
holds.
\item[3.] If $\V$ satisfies the Arakelov condition, then the same holds true for
its complex conjugate $\V^\vee$.
\end{enumerate}
\end{lemma}
\begin{proof}
Simpson correspondence implies in 1) that $E$ is concentrated in one bidegree,
whereas in 2) it implies that the Higgs field is non-zero.
Then 1) and 2) are just reformulations of the definition. 3) is obvious, since the polarization
(as indicated by the notation) allows to identify $\V^\vee$ with the dual local system. 
\end{proof}
\begin{lemma}\label{aemt.2} Consider for $i=1,\ldots,s$ polarized $\C$-variations of Hodge structures $\V_i$ with unipotent local monodromy at infinity and with Higgs bundles 
$$
\Big(E_i=\bigoplus_{m=0}^{k_i} E_i^{k_i-m,m},\theta_i\Big).
$$ 
If the Arakelov condition holds for all the $\V_i$, it holds for 
$\V=\V_1\otimes \cdots \otimes \V_s$ and for each irreducible direct factor
$\V'$ of $\V_1\otimes \cdots \otimes \V_s$.   
\end{lemma}
\begin{proof}
Let $(E,\theta)$ denote again the Higgs bundle of $\V$. In order to show that $\V$ satisfies
the Arakelov condition we may assume by induction that $s=2$.
Write $m^{(i)}_{\rm min}$ and $m^{(i)}_{\rm max}$ for the integers with
$E_i^{k_i-\ell_i,\ell_i}\neq 0$ for $m^{(i)}_{\rm min}\leq \ell_i \leq m^{(i)}_{\rm max}$.
Then for $k=k_1+k_2$
$$
E^{k-r,r} = \bigoplus_{\ell_1+\ell_2=r} E_1^{k_1-\ell_1,\ell_1}\otimes E_2^{k_2-\ell_2,\ell_2}\neq 0,
$$ 
if and only if $m_{\rm min}=m^{(1)}_{\rm min} +m^{(2)}_{\rm min} \leq r \leq m_{\rm max}=m^{(1)}_{\rm max} +m^{(2)}_{\rm max}$. In addition, if $m_{\rm min} \leq m_{\rm max}-1$ then
$r=\ell_1 + \ell_2$ for some $\ell_i$ with either $\ell_1 < m^{(1)}_{\rm max}$ or
$\ell_2 < m^{(2)}_{\rm max}$. In the first case, for example, part of the Higgs field is given by
the tensor product of the Higgs field $\theta_i|_{E_1^{k_1-\ell_1,\ell_1}}$ with the identity
on $E_2^{k_2-\ell_2,\ell_2}$, hence non-zero.

The equation~\ref{aracond2} tells us that as the tensor product of $\mu$-semistable sheaves $E_1^{k_1-\ell_1,\ell_1}\otimes E_2^{k_2-\ell_2,\ell_2}$ is $\mu$-semistable of slope
\begin{equation}\label{aracond3}
\mu(E_1^{k_1-m^{(1)}_{\rm min},m^{(1)}_{\rm min}}) + \mu(E_2^{k_2-m^{(2)}_{\rm min},m^{(2)}_{\rm min}}) - (\ell_1+ \ell_2-m^{(1)}_{\rm min}-m^{(2)}_{\rm min}) \cdot \mu(\Omega^1_Y(\log S)).
\end{equation}
So $E^{k-m,m}$ is $\mu$-semistable of slope $\mu(E^{k-m_{\rm min},m_{\rm min}}) - (m-m_{\rm min}) \cdot \mu(\Omega^1_Y(\log S))$, if non-zero. 

For the last part, let $(E',\theta')$ denote the Higgs bundle of the irreducible
subvariation of Hodge structures $\V'$. We choose
$m'_{\rm min}$ and $m'_{\rm max}$ to be the smallest and largest integer with
$E'^{k-m'_{\rm min},m'_{\rm min}}$ and $E'^{k-m'_{\rm max},m'_{\rm max}}$ non-zero.
By Simpson's correspondence $(E',\theta')$ can not be a direct sum of two Higgs bundles, hence
$\theta'_{E'^{k-m,m}}\neq 0$ for $m'_{\rm min}\leq m \leq m'_{\rm max}-1$. 
Finally, the $\mu$-semistability as well as the equation~\ref{aracond3} carry over 
to direct factors of $(E,\theta)$.
\end{proof}
\par
\begin{lemma}\label{aemt.4}
Let $\V$ be a complex polarized variation of Hodge structures $\V$ of weight $k$, with unipotent local monodromy around the components of $S$, and satisfying the Arakelov condition. 
\begin{enumerate}
\item[a.] There is a unique $m$ such that each unitary local subsystem $\U$ of $\V$
is concentrated in bidegree $(k-m,m)$. In particular
all global sections $s\in H^0(U, \V)$ are of bidegree $(k-m,m)$.
\item[b.] If $\V$ is defined over $\R$, hence of the form $\V_\R\otimes_\R\C$, then $k$ is even and $m=\frac{k}{2}$.
\end{enumerate}
\end{lemma}
\begin{proof}
Obviously b) follows from a) using Lemma~\ref{aracondex},~3).

By \cite{De87} in a) the local system $\U$ is a subvariation of Hodge structures, in particular the corresponding Higgs bundle $(F,0)$ is a direct factor of $(E,\theta)$. So $F^{p,q}$ has to be a direct factor of $E^{p,q}$, in particular it is again $\mu$-semistable of slope $\mu(E^{p,q})$. Since $\U$ is unitary $\mu(F^{p,q})=\mu(E^{p,q})=0$ and since $\mu(\Omega^1_Y(\log S))>0$ the equation~\ref{aracond2} implies that this is only possible for one tuple $(p,q)$.

For the last part of a) one takes for $\U$ the trivial sub-local system generated by $H^0(U, \V)$. 
\end{proof}
\par
In the sequel we consider a $\Q$-variation of Hodge structures $\W_\Q=R^1f_*\Q_A$
with unipotent monodromy at infinity, induced by a family of polarized abelian varieties $f:A\to U$. So $\W_\Q$ is polarized of weight $1$ and concentrated in bidegrees $(1,0)$ and $(0,1)$. For $\Q \subset K$ we will write $\W_K=\W_\Q\otimes K$ and $\W=\W_\C$.
\par
\begin{lemma}\label{decompass}
There exists a totally real number field $K$ such that:
\begin{enumerate}
\item[1.] One has a decomposition of variations of Hodge structures
$$
\W_K=\W_{1\,K}\oplus \cdots \oplus \W_{\ell\,K} 
\mbox{ \ \ with \ \ }
\W_{i\,K}=\V'_{i\,K}\otimes_K H_{i\,K},
$$ 
orthogonal with respect to the polarization.
\item[2.] $\V'_{i\,\R}=\V'_{i\,K}\otimes_K \R$ is irreducible for $i=1,\ldots,\ell$.
\item[3.] ${\rm Hom}(\V'_{i\,\R},\V'_{j\,\R})$ is a skew field for $i=j$ and is zero otherwise.
\item[4.]
For each $i$ the decomposition in 1) satisfies one of the following conditions:
\begin{enumerate}
\item[a.] $\V'_{i\,K}$ is a polarized $K$-variation of Hodge structures of weight $1$ and $H_{i\,K}$ a trivial $K$-Hodge structure, i.e.\ a $K$-vectorspace regarded as a Hodge structure concentrated in bidegree $(0,0)$. 
\item[b.] $H_{i\,K}$ a $K$-Hodge structure of weight $1$ and
$\V'_{i\,K}$ is a polarizable variation of Hodge structures concentrated in bidegree $(0,0)$
and unitary.
\end{enumerate}
\end{enumerate}
\end{lemma}
For subsequent use we label the direct factors in Lemma~\ref{decompass} such
that for some $\ell_2$ and for $1\leq i \leq \ell_2$ the condition a) holds true, whereas for $\ell_2 < i \leq \ell$ one has the condition b).  
\begin{proof}[Proof of Lemma~\ref{decompass}]
By \cite[Proposition~1.12]{De87} $\W$ decomposes as a direct sum of irreducible $\C$-subvariations of Hodge structures. Replacing the direct factors $\V$ which are not invariant under complex conjugation by $\V\oplus \V^\vee$, one obtains a decomposition 
of $\V_\R$ as a direct sum of irreducible polarized $\R$-subvariations of Hodge structures. 
As shown in \cite[Lemma~9.4]{VZ07}, for example, such a decomposition is induced by one which is defined over some totally real number field $K$, and it can be chosen to be orthogonal with respect to the polarization. The irreducibility implies that
${\rm Hom}(\V'_{i\,\R},\V'_{j\,\R})$ is a skew field if and only if  
$\V'_{i\,\R}\cong \V'_{j\,\R}$.

Of course we may write the direct sum of all direct factors, isomorphic to some $\V'_{i\,K}$
in the form $\V'_{i\,K}\otimes_K  H_{i\,K}$, for some $K$ vector space $H_{i\,K}$. 
As in \cite[Proposition~1.13]{De87} or in \cite[Theorem~4.4.8]{De71} one defines a Hodge structure on $H_{i\,K}$.

In 4) the bidegrees of $\V'_{i\,K}$ and $H_{i\,K}$ have to add up to
$(1,0)$ and $(0,1)$. If $H_{i\,K}$ is concentrated in bidegree $(0,0)$ the variation
of Hodge structures $\W_{i\,K}$ is just a direct sum of the $\V'_{i\,K}$, again
orthogonal with respect to the polarization, and one obtains case a).
Otherwise $\V'_{i\,K}$ has to be concentrated in bidegree $(0,0)$. Since it is polarizable, it has to be unitary.
\end{proof}
Beside of the totally real number field $K$ in Lemma~\ref{decompass}  
we fix as in Subsection~\ref{mt} a very general point $y\in U$. 
If a variation of Hodge structures is denoted by a boldface letter, the restriction to the base point $y\in U$ will be denoted by the same letter, not in boldface, so $W_{i\,K}$ and $V'_{i\,K}$ will denote the fibres at $y$ of $\W_{i\,K}$ and
$\V'_{i\,K}$, respectively.

As in \cite{An92}, one can extend the definition of the Hodge and Mumford-Tate group to  an arbitrary polarized $K$-Hodge structure $W_{K}$. Since the decomposition in Lemma~\ref{decompass} is defined over a real number field and orthogonal, the complex structure factors through
$$
\varphi_0: S^1 \>>> 
\bigtimes_{i=1}^\ell \Sp(W_{i\,K} \otimes_K \R,Q|_{W_{i\,K}}) \subset
\Sp(W_K \otimes_K \R,Q).
$$ 
In a similar way the morphism 
$h:{\rm Res}_{\C/\R}\G_m \to \Gl(W_K\otimes_K \R)$ factors through
$$
h:{\rm Res}_{\C/\R}\G_m \>>> \bigtimes_{i=1}^\ell \Gl(W_{i\,K} \otimes_K \R)
\subset \Gl(W_K\otimes_K \R).
$$
Hence for the {\em Mumford-Tate group $\MT(W_K)$}, defined as the smallest $K$-algebraic subgroup of $\Gl(W_K)$ whose extension to $\R$ contains the image of $h$, one has an inclusion
\begin{equation}\label{eqmt}
\MT(W_K)\subset \bigtimes_{i=1}^\ell \MT(W_{i\,K}). 
\end{equation}
By \cite{An92} and \cite{De82} the group $\MT(W_K)$ is reductive, and by \cite[Lemma 2, a)]{An92} it can again be defined as the largest $K$-algebraic subgroup of the linear group $\Gl(W_K)$, which leaves all $K$-Hodge cycles invariant, hence all elements
$$
\eta \in \big[W_K^{\otimes m} \otimes W_K^{\vee \otimes m'}\big]^{0,0}.
$$
The decomposition $\Gl(V'_{i\,K})\times \Gl(H_{i\,K}) \subset\Gl(W_{i\,K}) $ allows to define
$$
G^\mov_K=\bigtimes_{i=1}^\ell\Gl(V'_{i\,K})\times \{{\rm id}_{H_{i\,K}}\}\subset \bigtimes_{i=1}^\ell \Gl(W_{i\,K})
\subset \Gl(W_K). 
$$
\begin{addendum}\label{decompass5} Keeping the notations introduced in
Lemma~\ref{decompass} one has
\begin{enumerate}
\item[5.] There exists a $\Q$-algebraic subgroup $G^\mov_\Q$ of $\Gl(W_\Q)$ with
$G^\mov_\Q \otimes K=G^\mov_K$. Moreover $G^\mov_\Q$ is independent of $K$.
\end{enumerate}
\end{addendum}
\begin{proof} On may assume that $\W_\Q$ is irreducible over $\Q$.
Obviously, if $K'$ is a totally real extension of $K$, then
$G^\mov_{K'}=G^\mov_K\otimes K'$. So one may also assume that $K$ is a Galois extension of $\Q$ with Galois group $\Gamma$. 

Let $\Gamma'\subset \Gamma$ be the subgroup consisting of all $\gamma$ for which $\W_{i\,K}$ is isomorphic to the conjugate $\W_{i\,K}^\gamma$ under $\gamma$. In particular, $\V'_{i\,K}\cong \V'^\gamma_{i\,K}$. 
So the action of $\gamma$ on $\W_{i\,K}$ is trivial on the first factor $\V'_{i\,K}$, and it leaves $\Gl(V'_{i\,K})\times \{{\rm id}_{H_{i\,K}}\}$ invariant. Since
$\V'_{j\,K}=\V'^\delta_{i\,K}$ for some $\delta\in \Gamma$, unique up to
multiplication with $\Gamma'$, the group $G^\mov_K$ is invariant under conjugation by
$\Gamma$, hence it is defined over $\Q$ and, as said already, it is independent of $K$.
\end{proof}
As in \cite[Lemma~4.4.9]{De71} the polarization $Q|_{W_{i\,K}}$ is the tensor product of two forms $Q'_i$ and $Q_i$ on $\V'_{i\,K}$ and $H_{i\,K}$, respectively, one being antisymmetric, the other symmetric. This allows to distinguish in Lemma~\ref{decompass},~4,b)
two subcases:\\[.1cm] 
We say that $\W_{i\,K}$ is {\em of type b1} if $Q'_i$ is antisymmetric and {\em of type b2}
if $Q_i$ is antisymmetric. In the second case $H_{i\,K}$ is a polarized Hodge structure, and we can talk about its Mumford-Tate group.
\par
\begin{lemma}\label{mtdecomp} \
\begin{enumerate}
\item[a.] In case a) of Lemma~\ref{decompass},~4), i.e.\ for $i=1,\ldots,\ell_2$, 
one has
$$
\MT(W_{i\,K})=\MT(V_{i\,K}) \times \{{\rm id}_{H_{i\,K}}\}.
$$
\item[b.] In Lemma~\ref{decompass},~4.b) one finds: \\[.1cm]
1. In case b1, i.e.\ for symmetric $Q_i$, one has an inclusion
$$
\MT(W_{i\,K}) \subset \{{\rm id}_{V_{i\,K}}\} \times \SO(H_{i\,K}).$$
In particular, for $\dim(H_{i\, K})=2$, the group
$\MT(W_{i\,K})$ is commutative.\\[.1cm]
2. In case b2, i.e.\ if $Q_i$ is antisymmetric, one has
$$
\MT(W_{i\,K}) = \{{\rm id}_{V_{i\,K}}\} \times \MT(H_{i\,K}).$$
3. If $Q_i$ is antisymmetric or if $Q_i$ is symmetric and $\dim(H_{i\, K})>2$, there exists a non-zero antisymmetric endomorphism of $\W_{i\, K}$ of bidegree $(-1,1)$.
\end{enumerate}
\end{lemma}
\begin{proof}
Consider a non-trivial element
$$
\eta \in \big[W_{i\,K}^{\otimes m} \otimes_K  W_{i\,K}^{\vee \otimes m'}\big]^{0,0}=
\big[V'^{\otimes m}_{i\,K} \otimes_K  V'^{\vee \otimes m'}_{i\,K} \otimes_K 
H_{i\,K}^{\otimes m} \otimes_K  H_{i\,K}^{\vee \otimes m'}\big]^{0,0}.
$$   
So $m=m'$ and $\eta$ can be written as
$$
\eta=\sum_\iota \gamma_\iota \otimes h_\iota, \mbox{ \ \ with \ \ }
\gamma_\iota \in V'^{\otimes m}_{i\,K} \otimes \V'^{\vee \otimes m}_{i\,K}
\mbox{ \ \ and \ \ } h_\iota \in H_{i\,K}^{\otimes m} \otimes H_{i\,K}^{\vee \otimes m}.
$$
For $i \leq \ell_2$ all the $h_\iota$ are of bidegree $(0,0)$. Then $\eta$ is pure of bidegree $(0,0)$ if and only if this holds for $\gamma_\iota \otimes h_\iota$ for each $\iota$, or equivalently, if $\gamma_\iota$ is a Hodge cycle. Altogether $\MT(W_{i\,K})$ and $\MT(V'_{i\,K}) \times \{{\rm id}_{H_{i\,K}}\}$ are two reductive groups leaving the same cycles invariant. By \cite[Proposition~3.1~(c)]{De82} they coincide.

If $i > \ell_2$, the sections $\gamma_\iota$ are all of bidegree $(0,0)$. So again $\eta$ is of bidegree $(0,0)$ if and only if the same holds for the elements $h_\iota$. 
Let $\Gamma$ be the largest subgroup of $\Sp(H_{i\,K},Q_i)$ which leaves all tensors
$h$ of bidegree $(0,0)$ invariant. Then $\gamma_\iota\otimes h_\iota$ is invariant under $\{{\rm id}_{V_{i\,K}}\}\times \Gamma$ if and only if it is invariant under  $ \MT(W_{i\,K})$. Again by \cite[Proposition~3.1~(c)]{De82} both groups coincide.

In case b2 the vector space $H_{i\,K}$, together with $Q_i$, is a polarized variation of Hodge structures, and $\Gamma$ is the Mumford-Tate group of $H_{i\,K}$.

In case b1 one has $\Gamma \subset \SO(H_{i\,K})$. Since $\OO(2,K)$ is commutative
one obtains the second part of b.1).
\par
For the third part of b) assume first that $Q_i$ is symmetric. Then $\dim(H_{i\, K})=\mu$ is even and for $\mu\geq 4$ the elements of $\SO(\mu,K)$ generate the matrix algebra ${\rm M}(\mu,K)$.
So one finds an antisymmetric endomorphism of $V'_{i\, K}\otimes H_{i\,K}$ of bidegree
$(-1,1)$. 

For $Q_i$ antisymmetric there are obviously antisymmetric endomorphisms of $H_{i\,K}$ of bidegree $(-1,1)$. The product with the identity of $\V'_{i\,K}$ gives the endomorphism we are looking for. 
\end{proof}
To compare the Mumford-Tate group with the monodromy group in case a) of Lemma~\ref{decompass} one needs some additional hypothesis on the variation of Hodge structures, in our case the Arakelov equality. By \cite[Proposition~1.12]{De87} the variations of Hodge structures  $\V'_i=\V'_{i\,K}\otimes_K\C$ can be written as a direct sum of irreducible polarized $\C$-variations of Hodge structures. We distinguish two subcases.\\[.2cm]
\noindent
{\bf Type a1.} $\V'_i$ is an irreducible $\C$-variation of Hodge structures. This 
implies in particular that $\V'_i$ is isomorphic to its complex conjugate 
$\V'^\vee_i$, and that $\V'_i$ is not unitary.
In fact, if $\V'_i$ were unitary, it would decompose in two non-trivial subsystems, one of bidegree $(1,0)$ and the other of bidegree $(0,1)$, contradicting the irreducibility.

\begin{claim}\label{cyclesI}
Assume that $\W_{i\,K}$ is of type a1, and that it satisfies the Arakelov 
equality. Then all global sections
$$
\eta \in H^0\big(Y,\W_{i\,K}^{\otimes m} \otimes_K  \W_{i\,K}^{\vee \otimes m'}\big)
$$
are of bidegree $(m-m',m-m')$.
\end{claim}
\begin{proof}
The Arakelov equality implies that $\V'_i$ satisfies the Arakelov condition.
Since $H_i$ is a $K$-vector space concentrated in bidegree $(0,0)$, the same holds true for $\W_i=\V'_i\otimes H_i$. So the Claim follows from Lemma~\ref{aemt.4}.
\end{proof}
\par
\noindent
{\bf Type a2.} $\V'_i$ is the direct sum of two irreducible factors $\V_i$ and $\V^\vee_i$, dual to each other and interchanged by complex conjugation. 
Remark that we allow $\V_i$ and $\V^\vee_i$ to be unitary. If not, they satisfy the Arakelov equality. Hence by Lemma~\ref{aemt.2} the two variations of Hodge structures $\V_i$, $\V^\vee_i$ as well as their tensor product with $H_i$ will satisfy the Arakelov condition and Lemma~\ref{aemt.4} implies:
\begin{claim}\label{cyclesII} Assume that $\W_{i\,K}$ is of type a2, and either 
unitary or with Arakelov equality. Then there exist $p$ and $q$ such that 
all global sections
$$
\eta \in H^0\big(Y,(\V_{i\,K}\otimes_K  H_{i\,K})^{\otimes m} \otimes_K  (\V_{i\,K}\otimes_K H_{i\,K})^{\vee \otimes m'}\big)
$$
are of bidegree $(p,q)$, and all global sections
$$
\eta \in H^0\big(Y,(\V^\vee_{i\,K}\otimes_K  H_{i\,K})^{\otimes m} \otimes_K  (\V^\vee_{i\,K}\otimes_K  H_{i\,K})^{\vee \otimes m'}\big)
$$
are of bidegree $(q,p)$. Moreover one has $p+q=m-m'$. 
\end{claim}
\begin{claim}\label{mtII}
For $\W_{i\,K}$ of type a2 the Mumford-Tate group  
respects the decomposition of $\V'_i$, i.e.\ up to conjugation
$$\MT(W_{i\,K})\otimes_K \C \subset 
\Gl(V_i \otimes H_i)\times \Gl(V^\vee_i\otimes H_i).$$
\end{claim}
\begin{proof}
The decomposition in a direct sum can be defined over an imaginary quadratic extension
$K(\sqrt{b})$ of $K$, say with $\iota$ as a generator of the Galois group.
So the Mumford-Tate group acts trivially on $\iota$-invariant global sections of
${\rm End}(\W_i)$. Applying this to ${\rm id}_{V_i\otimes H_i}+{\rm id}_{V^\vee_i\otimes H_i}$
and to $\sqrt{b}\cdot({\rm id}_{V_i\otimes H_i}-{\rm id}_{V^\vee_i\otimes H_i})$
one obtains the claim.
\end{proof}
\begin{definition}\label{hgrestr}
Let $G^\mov_\Q$ be the group defined in Addendum~\ref{decompass5}. Then we define
the {\em moving part of the Mumford-Tate group} as 
$$
\MT^\mov(W_\Q)= \MT(W_\Q)\cap G^\mov_\Q \mbox{ \ \ and \ \ }
\MT^\mov(W_K)= \MT(W_K)\cap G^\mov_K.
$$
Correspondingly we write for any of the components $\W_{i\,K}$ in Lemma~\ref{decompass}
$$
\MT^\mov(W_{i\,K})=\MT(W_{i\,K})\cap \big(\Gl(V_{i\,K}) 
\times \{{\rm id}_{H_{i\,K}}\}\big).
$$
\end{definition}
Lemma~\ref{mtdecomp} allows to evaluate the moving part of the Mumford-Tate group.
In case a), i.e. for $i=1,\ldots,\ell_2$ one finds 
$$
\MT^\mov(W_{i\,K})= \MT(W_{i\,K})= \MT(V'_{i\,K})\times \{{\rm id}_{H_{i\,K}}\},
$$
whereas in case b) $\MT^\mov(W_{i\,K})$ is trivial. By~\ref{eqmt} 
\begin{equation}\label{mtprod}
\MT^\mov(W_K)=\MT(W_K)\cap \Big(\bigtimes_{i=1}^{\ell_2} \Gl(V'_{i\,K}) \times \{{\rm id}_{H_{i\,K}}\}\Big)\subset\bigtimes_{i=1}^{\ell_2}\MT^\mov(W_{i\,K}).
\end{equation}
To give a definition of $\MT^\mov(W_\Q)$ in terms of complex structures
we define 
\begin{equation}\label{restcs}
h^{\mov}:{\rm Res}_{\C/\R}\G_m \>>> \bigtimes_{i=1}^\ell \Gl(W_{i\,K} \otimes_K \R)
\>{\rm proj} >> \bigtimes_{i=1}^{\ell_2} \Gl(W_{i\,K} \otimes_K \R).
\end{equation}
\begin{lemma}\label{hgrest2} \ 
\begin{enumerate}
\item[i.] $\MT^\mov(W_\Q)$ is a normal subgroup of $\MT(W_\Q)$.
\item[ii.] $\MT^\mov(W_K)$ is the smallest $K$-algebraic subgroup $H_K$ of
$\Gl(W_K)$, for which $H_K\otimes_K \R$ contains the image of $h^{\mov}$.
\item[iii.] $\MT^\mov(W_\Q)$ is the smallest  $\Q$-algebraic subgroup $H_\Q$ of
$\Gl(W_\Q)$ with $$\MT^\mov(W_K) \subset H_\Q\otimes K.$$
\item[iv.] $\MT^\mov(W_\Q)$ is the smallest  $\Q$-algebraic subgroup $H_\Q$ of
$\Gl(W_\Q)$, for which $H_\Q\otimes \R$ contains the image of $h^{\mov}$. 
\end{enumerate}
\end{lemma}
\begin{proof}
We may assume again that $\W_\Q$ is irreducible an that $K$ is Galois over
$\Q$ with Galois group $\Gamma$.

Part ii) follows from~\ref{mtprod} and from the definition of $\MT(W_{i\,K})$, and
part iv) follows from ii) and iii).

To verify part iii) remark that $\MT(W_\Q)$ is the smallest $\Q$-algebraic
subgroup of $\Gl(W_\Q)$ whose extension to $K$ contains $\MT(W_K)$.
By~\ref{mtprod} 
$$
\MT^\mov(W_K)=\MT(W_K)\cap G_K = \MT(W_K)\cap ( G_\Q \otimes K).
$$
Taking conjugates with $\sigma\in \Gamma$ one finds that
$$
\MT^\mov(W_K)^\sigma =\MT(W_K)^\sigma \cap ( G_\Q \otimes K).
$$
For the smallest $\Q$ algebraic subgroup $H_\Q$ of $\Gl(W_\Q)$ with $\MT^\mov(W_K) \subset H_\Q\otimes K$ the extension $H_\Q\otimes K$ of scalars is the product over all conjugates of $\MT^\mov(W_K)$, hence it is equal to $(\MT(W_\Q)\otimes K) \cap ( G_\Q \otimes K)$ and one obtains iii).

Obviously $G^\mov_K$ is normal in $\displaystyle\bigtimes_{i=1}^{\ell} \Gl(W_{i\,K})$.
The latter contains $\MT(W_K)$ and all its conjugates under $\Gamma$. So
$\MT^\mov(W_\Q)\otimes K$ is a normal subgroup of $\MT(W_\Q)\otimes K$ and i) holds true.
\end{proof}
Lemma~\ref{mtdecomp} implies that $\MT^\mov(W_{i\,K})^\der=\MT(W_{i\,K})^\der$
in case a) and in case b1), provided $\dim(H_{i\,K}) =2$. In the remaining cases by Lemma~\ref{mtdecomp},~b.3) there exists a non-zero antisymmetric endomorphism
of $\W_{i\,K}$, which by \cite{Fa83} implies non-rigidity. So we can state:
\begin{lemma}\label{trivial}
Assume that $\W_\Q$ is a rigid polarized variation of Hodge structures of weight $1$.
Then for all $i$ one has $\MT^\mov(W_{i\,K})^\der=\MT(W_{i\,K})^\der$ and hence $\MT^\mov(W_\Q)^\der=\MT(W_\Q)^\der$.
\end{lemma}
Recall that $\W_\Q$ is the variation of Hodge structures of a polarized family of abelian varieties $f:A\to U$, and that $W_K$ and $W_\Q$ are the restrictions of $\W_K$ and $\W_\Q$ to a very general point $y\in U$. So $\MT(W_\Q)^\der$ is compatible with parallel transport and, following the usual convention, we write $\MT(\W_\Q)$ instead of $\MT(W_\Q)$ and
$\MT^\mov(\W_\Q)$ instead of $\MT^\mov(W_\Q)$
For $L=\Q$ or $L=K$ we consider the monodromy group $\Mon(\W_L)$, defined as the smallest $L$-algebraic subgroup of $\Gl(W_L)$ which contains the image of the monodromy representation. As usual the upper Index $0$ refers to the connected component of the identity.
By \cite{De82} (see also \cite{An92} or \cite{Mo98}) the connected component $\Mon^0(f)=\Mon^0(\W_\Q)$ is a normal subgroup of the derived subgroup $\MT(\W_\Q)^\der$.
\begin{proposition}\label{mtfinal}
Keeping the notations introduced in Lemma~\ref{decompass}, assume that each
irreducible direct factor of $\W=\W_\Q\otimes \C$ is either unitary or satisfies the Arakelov equality. Then 
$$
\MT^\mov(W_K)^\der \subset \Mon^0(\W_K).
$$
\end{proposition}
Before proving Proposition~\ref{mtfinal} let us state and prove the corollary we 
are heading for.
\par
\begin{corollary}\label{mtcor}
Let $Y$ be a non-singular projective variety, and let $U\subset Y$ be the complement of
a normal crossing divisor $S$. Assume that $\Omega^1_Y(\log S)$ is nef and that
$\omega_Y(S)$ is ample with respect to $U$.
Let $f:A\to U$ be a family of polarized abelian varieties with unipotent local monodromy at infinity and such that for $\W_\Q=R^1f_*\Q_A$ each non-unitary irreducible subvariation of Hodge structures of $\W=\W_\Q\otimes \C$ satisfies the Arakelov equality.
Then  
\begin{equation}\label{mtequality}
\MT^\mov(\W_\Q)^\der = \Mon^0(\W_\Q) = \MT(\W_\Q)^\der \cap G^\mov_\Q.
\end{equation}
In particular $\Mon^0(\W_\Q)=\MT^\mov(R^1f_*\Q_A)^\der$ is normalized by 
$\MT(\W_\Q)^\der$.

If $f:A\to U$ is rigid one finds that $\Mon^0(\W_\Q)=\MT(\W_\Q)^\der$.
\end{corollary}
\begin{proof}
Choose the totally real number field $K$ according to Lemma~\ref{decompass}.
Obviously $\Mon^0(\W_K)$ is contained in $\Mon^0(\W_\Q)\otimes K$, hence by Proposition~\ref{mtfinal} one has an inclusion
$$
\MT^\mov(W_K)^\der \subset \Mon^0(\W_\Q)\otimes K.
$$
Extending the coefficients to $\R$ one finds by Lemma~\ref{hgrest2},~ii) that
$\Mon^0(\W_\Q)\otimes \R$ contains the image of the moving part of the complex structure $h^\mov$, as defined in~\ref{restcs}. By part iv) of Lemma~\ref{hgrest2} one gets
$\MT^\mov(\W_\Q)^\der \subset \Mon^0(\W_\Q)$. 
By \cite{De82} one knows that $\Mon^0(\W_\Q)\subset\MT(\W_\Q)^\der$. 
Since obviously $\Mon^0(\W_\Q) \subset G^\mov_\Q$, one obtains~\ref{mtequality}.
The normality of $\MT^\mov(\W_\Q)^\der \subset \MT(\W_\Q)^\der$ follows from Lemma~\ref{hgrest2},~i). Finally the last part of  Corollary~\ref{mtcor} is a consequence of~\ref{mtequality}, using Lemma~\ref{trivial}.
\end{proof}
Using the notations from Section~\ref{Kugadef}, we choose $V=H^1(f^{-1}(y),\Q)$ for the very general point $y\in U$ and the induced symmetric bilinear form $Q$.

Since $\Mon^0(\W_\Q)=\MT^\mov(R^1f_*\Q_A)^\der$ is normalized by $\MT(R^1f_*\Q_A)^\der$, hence by the complex structure $\varphi_0$ as well, one obtains Kuga fibre spaces over
$$\sX^\mov=\sX(\MT^\mov(R^1f_*\Q_A)^\der,{\rm id},\varphi_0)\subset \sX=\sX(\MT(R^1f_*\Q_A)^\der,{\rm id},\varphi_0).$$ 
By \cite{Mu66} and \cite{Mu69} $\sX$ is the moduli space of abelian varieties whose Mumford-Tate group is contained in $\MT(R^1f_*\Q_A)$. So the family $f:A\to U$ induces a morphism $U \to \sX$, perhaps after replacing $U$ by an \'etale covering. Since $\varphi:U \to \sA_g$ is generically finite over its image, the morphism $U \to \sX$ has the same property.

Assume in Corollary~\ref{mtcor} that $f: A \to U$ is rigid, and that $\dim(U) \geq \dim{\sX}$. 
The rigidity implies by Corollary~\ref{mtcor} that $\MT^{\mov}(\W_\Q)^\der=\MT(\W_\Q)^\der$, and hence that $\sX^\mov=\sX$ is a Shimura variety of Hodge type.
Since $\varphi$ is generically finite over its image, $\varphi:U\to \sX$ is dominant, hence
$\sX=\varphi(U)$. By Lemma~\ref{etale},~(2) $\varphi:U\to \sX$ is \'etale.

The same argument applies for non-rigid families if one knows that $\varphi$
factors through $\sX^\mov$ and if $\dim(U) \geq \dim{\sX^\mov}$. So we can state:

\begin{lemma}\label{shimura}
Assume in Corollary~\ref{mtcor} that the induced morphism $\varphi:U\to \sA_g$
factors through $\sX^\mov$ and that $\dim(U) \geq \dim(\sX^\mov)$. Then (replacing $U$ by an \'etale covering, if necessary) $\varphi:U\to \sX^\mov$ is finite, \'etale, and surjective.

In particular this holds true if $f: A \to U$ is rigid, hence $\sX^\mov=\sX$ and if $\dim(U) \geq \dim(\sX)$.
\end{lemma}
\begin{example}\label{shimura2}
Assume in Corollary~\ref{mtcor} that the universal covering $\tilde{U}$ is a bounded symmetric domain, and that $\W_\Q$ is the uniformizing local system. So
$\tilde{U}$ is isomorphic to $\Mon^0(\W_\Q)\otimes\R$, divided by a maximal compact subgroup.

Assume either that $f:A\to U$ is rigid, or that the morphism $\tilde{\varphi}$ from $\tilde{U}$
to the Siegel upper halfspace $\tilde{\sA}_g$ is induced by a homomorphism 
$$
\Mon^0(\W_\Q)\otimes \R \to \Sp(2g,\R).
$$
Then the assumptions in Lemma~\ref{shimura} hold true.

In fact, in both cases we know that $\tilde{\varphi}:\tilde{U}\to \tilde{\sA}_g$
factors through $\sX^\mov$. Moreover the real dimension of
$\tilde{U}$ is equal to the dimension of the quotient of $\Mon^0(\W_\Q)\otimes \R$ by a maximal compact subgroup, hence equal to $2\cdot \dim(\sX^\mov)$.
\end{example}
\begin{remark}
Without any assumption on rigidity Theorem~\ref{Moonen_product} gives the existence
of a Shimura variety of Hodge type $\sX_1\times \sX_2$ such that $U=\sX_1\times \{b\}$.
Using the notations introduced above, $\sX=\sX_1\times \sX_2$ and $\sX^\mov=\sX_1\times \{b\}$.
By deforming $b$ to a point $a$ with complex multiplication one gets a Shimura variety of Hodge type $\sX_1\times \{a\}$.

As we have seen the non-rigidity comes from the existence of direct factors of type b1 with $\dim(H_{i\, K})\geq 4$ or of type b2. Passing from $b$ to $a$ corresponds to a modification of the Hodge structure $H_{i\, K}$ in such a way, that $\MT(W_{i\, K})/_{\MT^\mov(W_{i\, K})}$ becomes commutative.
\end{remark}

\begin{proof}[Proof of Proposition~\ref{mtfinal}]
We will apply arguments, similar to the ones used in the proof of 
\cite[Proposition 10.3]{VZ07}. 
By \cite[Lemma 4.4]{Si92} $\Mon^0(\W_K)$ is reductive, hence by \cite[Proposition~3.1~(c)]{De82} there is no larger subgroup of $\Gl(W_K)$ which leaves
all elements $\eta_y \in W_K^{\otimes m}\otimes_K W_K^{\vee \otimes m'}$ invariant, which are invariant under $\Mon^0(\W_K)$. If we verify that
all elements $\eta_y \in W_K^{\otimes m}\otimes_K W_K^{\vee \otimes m'}$
which are invariant under $\Mon^0(\W_K)$ are invariant under $\MT^\mov(W_K)^\der$,
we get the inclusion 
$$ \MT^\mov(W_K)^\der \subset \Mon^0(\W_K).
$$
If $\eta_y$ is invariant under $\Mon^0(\W_K)$, one may replace $U$ by an \'etale cover and assume that
$\eta_y$ is invariant under the monodromy representation, hence it is the restriction of a global section
$$
\eta\in H^0\big(Y,\W_K^{\otimes m}\otimes_K  \W_K^{\vee \otimes m'}\big).
$$
Since $K$ is a totally real number field, $\W_K^\vee$ is isomorphic to $\W_K$, hence
$\det(\W_K)^{2}$ is trivial. Up to a shift of the bigrading, $\W_K^\vee$ can be identified with 
$$
\bigwedge^{\rk(W_K)-1}\W_K\otimes_K  \det(\W_K)^{-1}=
\bigwedge^{\rk(W_K)-1}\W_K\otimes_K  \det(\W_K),
$$
so we may as well consider sections of 
$$
\eta \in H^0\big(Y,\W_K^{\otimes k}\big)= \bigoplus_{\sI'}H^0\Big(Y,\bigotimes_{i=1}^{\ell} \W^{\otimes \kappa_i}_{i\,K}\Big) = \bigoplus_{\sI'}
H^0\Big(Y,\bigotimes_{i=1}^{\ell} \V'^{\otimes \kappa_i}_{i\,K}\Big)\otimes_K  
\bigotimes_{i=1}^{\ell} H_{i\,K}^{\otimes \kappa_i},
$$
where $\sI'$ is the set of tuples $\underline{\kappa}=(\kappa_1,\ldots,\kappa_\ell)$ with $\sum_{i=1}^\ell \kappa_i=k$, so $\eta=\sum_{\sI'}\eta_{\underline{\kappa}}$.
Each component of $\eta$ in this direct sum decomposition is again invariant under $\Mon^0(\W_K)$.
So we may as well assume that $\eta=\eta_{\underline{\kappa}^0}$ for a fixed tuple $\underline{\kappa}^0=(\kappa^0_1,\ldots,\kappa^0_\ell)$ and that
$$
\eta_{\underline{\kappa}^0}= \gamma_{\underline{\kappa}^0}\otimes
h_{\underline{\kappa}^0} \mbox{ \ \ with \ \ }
\gamma_{\underline{\kappa}^0}\in H^0\Big(Y,\bigotimes_{i=1}^{\ell} \V'^{\otimes \kappa^0_i}_{i\,K}\Big) \mbox{ \ \ and \ \ }
h_{\underline{\kappa}^0}\in 
\bigotimes_{i=1}^{\ell} H_{i\,K}^{\otimes \kappa^0_i}.
$$
Recall that by our choice of the indices we are in case a) of Lemma~\ref{decompass},~4)
for $i=1,\ldots,\ell_2$. Let us rearrange the indices in such a way, that
$i=1,\ldots,\ell_1$ the local system $\V'_i=\V'_{i\, K}\otimes_K\C$ remains irreducible (type a1),
whereas for $i=\ell_1+1,\ldots,\ell_2$ it decomposes (type a2). 

Choose a Galois extension $L$ of $K$ with Galois group $\Gamma$, such that the local systems
$\V'_{i\,L}$ decompose as a direct sum of two subsystems $\V_{i\,L}$ and $\V^\vee_{i\,L}$ for $i=\ell_1+1,\ldots,\ell_2$. By abuse of notation we will drop the $L$, hence ${}_{i}$ stands for ${}_{i\,L}$.

Consider the set $\sI$ of tuples of natural numbers 
\begin{gather*}
\underline{k}=(k_1,\ldots,k_{\ell_1},k_{\ell_1+1},k'_{\ell_1+1},\ldots,
k_{\ell_2},k'_{\ell_2},k_{\ell_2+1},\ldots,k_\ell), 
\mbox{ \ \ with \ \ }\\
k_i=\kappa_i^0 \mbox{ \ \ for \ \ } i\in \{ 1,\ldots, \ell_1\}\cup\{\ell_2+1,\ldots,\ell\}\mbox{ \ \ and}\\
k_i+k'_i=\kappa_i^0 \mbox{ \ \ for \ \ } i\in \{\ell_1+1,\ldots, \ell_2\}.
\end{gather*}
Then $H^0\Big(Y,\bigotimes_{i=1}^{\ell} \V'^{\otimes \kappa^0_i}_{i\,K}\Big)\otimes_KL$ decomposes as
$$
\bigoplus_{\sI}H^0\Big(Y,\bigotimes_{i=1}^{\ell_1} \V'^{\otimes k_i}_{i}\otimes  \bigotimes_{i=\ell_1+1}^{\ell_2}\big(\V_i^{\otimes k_i}\otimes
\V_i^{\vee\otimes k'_i}\big)\bigotimes_{i=\ell_2+1}^{\ell} \V'^{\otimes k_i}_{i}\Big).
$$
Remark that the local systems $\V'_i$ and $\V_i$ occurring in this decomposition all satisfy the
Arakelov condition. Hence $\gamma=\gamma_{\underline{\kappa}^0}$ and $\eta=\eta_{\underline{\kappa}^0}$ decompose as 
$$
\gamma=\sum_{\sI}\gamma_{\underline{k}} \mbox{ \ \ and \ \ }
\eta=
\sum_{\sI}\gamma_{\underline{k}}\otimes h_{\underline{\kappa}^0}
$$ 
where by Lemma~\ref{aemt.4} 
$$
\gamma_{\underline{k}}\in \bigoplus_{\sI}H^0\Big(Y,\bigotimes_{i=1}^{\ell_1} \V'^{\otimes k_i}_{i}\otimes  \bigotimes_{i=\ell_1+1}^{\ell_2}\big(\V_i^{\otimes k_i}\otimes
\V_i^{\vee\otimes k'_i}\big)\otimes
\bigotimes_{i=\ell_2+1}^{\ell} \V'^{\otimes k_i}_{i}\Big)
$$ 
is pure of some bidegree $(p_{\underline{k}},q_{\underline{k}})$.

The Galois group $\Gamma$ acts on the decomposition, and since $\eta$ and 
$h=h_{\underline{\kappa}^0}$ are defined over $K$ the group $\Gamma$ permutes the components $\gamma_{\underline{k}}$.
The sum over the conjugates of a fixed $\gamma_{\underline{k}}$
will again be defined over $K$, and by abuse of notations, replacing $\sI$ by a subset, we can assume that $\sI$ consists of one $\Gamma$-orbit. 

If for some $\underline{k}\in \sI$ one has $p_{\underline{k}}\neq q_{\underline{k}}$
then $\gamma_{\underline{k}}$ is not defined over $\R$, and its complex conjugate
is of the form $\gamma_{\underline{k'}}$ for some $\underline{k'}\in \sI$. In particular
$
p=\sum_{\sI}p_{\underline{k}}=\sum_{\sI}q_{\underline{k}},
$
and hence the wedge product $\rho=\bigwedge_{\sI} \gamma_{\underline{k}}$
is pure of bidegree $(p,p)$ and defined over $L$. Since wedge products are direct factor of
some tensor product, $\rho$ is a section in
$$
H^0\Big(Y,\bigotimes^\nu \big( \bigotimes_{i=1}^{\ell_1} \V'^{\otimes k_i}_{i}\otimes  \bigotimes_{i=\ell_1+1}^{\ell_2}\big(\V_i^{\otimes k_i}\otimes
\V_i^{\vee\otimes k'_i}\big)\otimes
\bigotimes_{i=\ell_2+1}^{\ell} \V'^{\otimes k_i}_{i}\big)\Big).
$$
The Galois group $\Gamma$ of $L$ over $K$ permutes the different components $\gamma_{\underline{k}}$, hence it acts on $\rho$ by a character $\chi:\Gamma \to \{\pm 1 \}$. So for some $\beta\in L$ the cycle  $\beta\cdot\rho$ is invariant under $\Gamma$. Choosing 
$$
h' \in \bigotimes^\nu \bigotimes_{i=1}^{\ell} H_{i\,K}^{\otimes \kappa_i}
$$  
of bidegree $(p',p')$ one obtains a Hodge cycle
$$
\beta\cdot \rho\otimes h' \in H^0\big(Y,\W_K^{\otimes k\cdot \nu}\big).
$$
So $\beta\cdot \rho\otimes h'$ is invariant under $\MT(W_K)^\der$ hence under
the subgroup $\MT^\mov(W_K)^\der$ as well. This group acts trivially on
$h'$, hence $\beta\cdot\rho$ has to be invariant under $\MT^\mov(W_K)^\der$,
where we consider the identification
$$
\G^\mov_K= \bigtimes_{i=0}^{\ell} \Gl(V'_{i\,K}) \times \{{\rm id}_{H_{i\,K}}\}\cong
\bigtimes_{i=0}^{\ell} \Gl(V'_{i\,K}).
$$
This implies that the subspace 
$$
J=<\gamma_{\underline{k}}; \ \underline{k}\in \sI>_L \subset
\bigoplus_{\sI}H^0\Big(Y,\bigotimes_{i=1}^{\ell_1} \V'^{\otimes k_i}_{i}\otimes  \bigotimes_{i=\ell_1+1}^{\ell_2}\big(\V_i^{\otimes k_i}\otimes
\V_i^{\vee\otimes k'_i}\big)\otimes 
\bigotimes_{i=\ell_2+1}^{\ell} \V'^{\otimes k_i}_{i}\Big)
$$
is invariant under the action of $\MT^\mov(W_K)^\der\otimes L$.
Since 
$$
\MT^\mov(W_K)^\der \subset \big(\bigtimes_{i=0}^{\ell} \Gl(V'_{i\,K}) \times \{{\rm id}_{H_{i\,K}}\}\big)
$$
and since we have seen in Claim~\ref{mtII} that $\MT(W_{i\,K})^\der\otimes_K  \C$
respects the decomposition $\V'_{i\,K}\otimes_K  \C=\V_{i}\oplus \V^\vee_i$,
the action of $\MT^\mov(W_K)^\der\otimes_K  L$ leaves for each $\underline{k}\in \sI$
the subspaces
$$
<\gamma_{\underline{k}}>_L=
J \cap
H^0\Big(Y,\bigotimes_{i=1}^{\ell_1} \V'^{\otimes k_i}_{i}\otimes  \bigotimes_{i=\ell_1+1}^{\ell_2}\big(\V_i^{\otimes k_i}\otimes
\V_i^{\vee\otimes k'_i}\big)\otimes\bigotimes_{i=\ell_2+1}^{\ell} \V'^{\otimes k_i}_{i}\Big)
$$
invariant. So one obtains a homomorphism
$$
\MT^\mov(W_K)^\der\otimes_K  L \>>> \Gl(<\gamma_{\underline{k}}>_L)= L^*,
$$
necessarily trivial. In particular $\gamma_{\underline{k}}$ is invariant under $\MT^\mov(W_K)^\der \otimes_K  L$. 
\par
Since both $\sum_{\underline{k}}\gamma_{\underline{k}}$ and
$\eta=\sum_{\sI}\gamma_{\underline{k}}\otimes h_{\underline{\kappa}^0}$
are defined over $K$, they are invariant under $\MT^\mov(W_K)^\der$, as claimed.
\end{proof}

\section{Variations of Hodge structures of low rank}\label{ex}

In this section we will discuss the `complexity condition' 2) in Theorem ~\ref{characterization}, b) for 
$\C$-variations of Hodge structures of low rank.
\begin{assumptions}\label{excludeass}
The $\C$-variation of Hodge structures $\V$ is non unitary, irreducible with unipotent monodromy at infinity and it satisfies the Arakelov equality. By Theorem~\ref{purity_Thm} $\V$ is pure for some i,
and we assume that $\Omega_i$ is of type A or B.
We write $\Omega$, $T$, and $n$ for $\Omega_i$, its dual, and 
its rank and $M$ for the corresponding factor of the universal covering $\tilde{U}$. 
As usual $(E=E^{1,0}\oplus E^{0,1},\theta)$ denotes the Higgs bundle of $\V$, the 
Hodge numbers are $\ell=\rk(E^{1,0})$ and $\ell'=\rk(E^{0,1})$, hence
the period map is given by a morphism $M\to {\rm SU}(\ell,\ell')$. 

We will assume moreover, that $\omega_Y(S)$ is ample or that the following
strengthening of the condition ($\star$) in Lemma~\ref{poly} holds. 
\end{assumptions}
\par
\begin{condition}\label{poly2} \ 
\begin{itemize}
\item[i.]
If $\sF$ and $\sG$ are two $\mu$-stable torsion free coherent sheaves, then $\sF\otimes \sG$ is $\mu$-polystable.
\item[ii.]
If $\sF$ is a $\mu$-stable torsion free coherent sheaf, then $\sF$ admits an admissible
Hermite-Einstein metric, as defined in \cite{BS94}.
\end{itemize}
\end{condition}
The Condition~\ref{poly2} will allow to apply \cite[Lemma 2.7]{VZ07}, saying that the Higgs field $\theta$
respects the socle filtration. In particular, the $\mu$-polystability of $E^{1,0}$ will imply the $\mu$-polystability of $E^{1,0}\otimes T$, hence the
$\mu$-polystability of $E^{0,1}$. 
\par
\begin{lemma}\label{poly3}
If $\omega_Y(S)$ is ample, then the Condition~\ref{poly2} hold true.
\end{lemma}
\begin{proof}
In \cite{BS94} it is shown, that a reflexive sheaf on a compact 
K\"ahler manifold admits
an admissible Hermite-Einstein metric if and only if it is $\mu$-polystable. 
Part i) follows from the fact, that a tensor product of two 
admissible Hermite-Einstein metrics is again admissible Hermite-Einstein.
In fact, in \cite{BS94} admissibility of metrics $h_i$ on bundles
$\sV_i$ asks for two conditions. First, the curvatures $F_i$ should
be square integrable and second their traces $\Lambda F_i$  should
be uniformly bounded. The curvature of $h_1 \otimes h_2$ is
$F_1\otimes {\rm Id}_2+ {\rm Id}_1\otimes F_2$. Thus, if $h_i$ are
admissible, so is $h_1 \otimes h_2$, and
the claim follows.
\end{proof}
\par

Recall that by~\ref{equpper} the length $\varsigma(\V)=\varsigma((E,\theta))$ of the Higgs subbundle $\bigwedge^\ell(E,\theta)$ satisfies
\begin{equation}\label{eqex.1}
{\rm Min}\{\ell,\ell'\} \geq
\varsigma(\V) \geq \frac{\ell\cdot\ell'\cdot(n+1)}{(\ell+\ell') \cdot n}.
\end{equation}
Since $\V$ irreducible, by Addendum~\ref{characterizationadd},~III) the bundle $E^{1,0}$ is $\mu$-stable if and only if the right hand side of~\ref{eqex.1} is an equality. Since
\eqref{eqex.1} is symmetric in $\ell$ and $\ell'$, in order to verify the equality in certain cases, we are allowed to replace $\V$ by $\V^\vee$ and assume that $\ell\leq \ell'$.
One obtains:
\begin{property}\label{exlenght}
The irreducibility of $\V$ implies that $n\cdot \ell \geq \ell'\geq \ell$. 
If $\ell'=n\cdot \ell$ the numerical condition 2) in Theorem~\ref{characterization} holds, 
hence the right hand side of~\ref{eqex.1} is an equality. In particular this is the case for $n=1$, as
said already in Lemma~\ref{numcondA}.
\end{property}
\begin{example}\label{ex.1}
Assume $\ell=1$. Since $E^{1,0}$ is invertible,
$E^{0,1}$ is the saturated hull of the $\mu$-stable sheaf 
$E^{1,0}\otimes T$, hence of rank
$\ell'=n$, and (\ref{eqex.1}) is an equality.
\end{example}
\par
\begin{lemma}\label{invquot} \ 
\begin{enumerate}
\item[i.] The Hodge bundle $E^{1,0}$ can not have a torsion free $\mu$-stable quotient sheaf $\sV$ with $\mu(\sV)=\mu(E^{1,0})$, such that $\sV\otimes T$ is $\mu$-stable.
\item[ii.] In particular $E^{1,0}$ can not have a torsionfree rank one quotient sheaf $\sN$ with $\mu(\sN)=\mu(E^{1,0})$.
\end{enumerate}
\end{lemma}
\begin{proof}
Obviously ii) is a special case of i). Assume there exists a torsion free $\mu$-stable quotient sheaf $\sV$ with $\mu(\sV)=\mu(E^{1,0})$, such that $\sV\otimes T$ is $\mu$-stable.
To be allowed to replace $\sV$ by its reflexive hull, we only assume that there is a morphism $E^{1,0} \to \sV$ which is surjective on some open dense subscheme and that $\mu(\sV)=\mu(E^{1,0})$. 

In order to keep notations consistent with \cite[Section 2]{VZ07}, we will first study the dual situation, hence a subbundle $\sV'$ of $E^{0,1}$. 
Recall that the socle $\sS_1(\sF)$ of a coherent sheaf $\sF$ is the smallest saturated subsheaf
containing all $\mu$-polystable subsheaves of $\sF$ of slope $\mu(\sF)$. 
By \cite[Lemma 2.7]{VZ07} the Property~\ref{poly2},~i) implies that the Higgs field $\theta$ respects the socle,
in particular for $\sV'\subset \sS_1(E^{0,1})$ the preimage $\theta^{-1}(\sV'\otimes \Omega)$ is contained in $\sS(E^{1,0})$. Since $(E,\theta)$ is the Higgs bundle of an irreducible variation of Hodge structures,
$\theta^{-1}(\sV'\otimes \Omega)\neq 0$. In fact, 
$\theta^\vee: E^{1,0} \otimes T \to E^{0,1}$ is surjective, since the cokernel
would be a Higgs subbundle of $(E,\theta)$ of degree zero. 

So $\theta^{-1}(\sV'\otimes \Omega)$ is a non-trivial subsheaf of the socle, hence $\mu$-polystable.
The $\mu$-stability of $\sV'\otimes \Omega$ implies that $\theta^{-1}(\sV'\otimes \Omega)$ contains a direct factor which is $\mu$-equivalent to $\sV'\otimes \Omega$.
 
Applying this to the cosocle $\sS'(E^{1,0})$, i.e.\ to the dual of $\sS({E^{1,0}}^\vee)$ 
one finds a quotient sheaf of $E^{0,1}$ which is $\mu$-equivalent to $\sV\otimes T$.
So $(E,\theta)$ has a quotient Higgs bundle whose reflexive hull is isomorphic to $\sQ=\sV\oplus \sV\otimes T$.
Lemma~\ref{spl.2},~ii), applied to $\sQ=\sV\oplus \sV\otimes T$, and the Arakelov equality imply that 
\begin{multline*}
0\leq \mu(\sQ)\rk(\sQ)=\rk(\sV)\cdot \mu(\sV) + \rk(\sV)\cdot n\cdot (\mu(\sV)-\mu(\Omega)) = \\
\rk(\sV)\cdot(\mu(E^{1,0}) + n\cdot (\mu(E^{1,0})-\mu(\Omega)))=\rk(\sV)\cdot(\mu(E^{1,0})+ 
n \cdot \mu(E^{0,1})).
\end{multline*}
On the other hand, the property~\ref{exlenght} implies that 
$$
0= \ell\cdot \mu(E^{1,0}) + \ell'\cdot \mu(E^{1,0})\geq \ell\cdot( \mu(E^{1,0}) + n\cdot \mu(E^{1,0})),
$$
hence that $\mu(\sQ)=0$. Since $\V$ is irreducible, $(E,\theta)$ can not have a Higgs subbundle of degree zero, a contradiction.
\end{proof}
\begin{example}\label{ex.2}
If $\ell=2$ and if the $\mu$-semistable sheaf $E^{1,0}$ was not $\mu$-stable, 
one would find an invertible quotient, contradicting Lemma~\ref{invquot},~ii). 

Hence $E^{1,0}$ is $\mu$-stable, and the right hand side of (\ref{eqex.1}) is an equality.
Since ${\rm Min}\{\ell,\ell'\}=2$ the only solution is $\ell'=2\cdot n$ and $\varsigma(\V)=2$.
\end{example}

Next we will consider the case of a rank two quotient of $E^{1,0}$. 
To this aim, we have to analyze the holonomy group:
\begin{lemma}\label{ranktwo}
Let $\sV$ be a $\mu$-stable torsion free quotient sheaf of $E^{1,0}$ of rank two with $\mu(\sV)=\mu(E^{1,0})$. Then $n=2$ and for some invertible sheaf $\sN$
one has an isomorphism $\sV^{\vee\vee}\cong T\otimes \sN$. 
\end{lemma}
\par
\begin{proof}
By Lemma~\ref{invquot},~ii) $\sV$ has to be $\mu$-stable. Moreover, since the assumptions
are compatible with replacing $U$ by an \'etale covering, $\sV$ remains $\mu$-stable under
pullback to such a covering. By Lemma~\ref{invquot},~i) the sheaf $\sV\otimes T$ can not be $\mu$-stable.
So in order to finish the proof of the Lemma~\ref{ranktwo} it just remains to verify:
\begin{claim}\label{tangent}
Let $\sV$ be a rank $2$ torsion free sheaf on $Y$, whose pullback to any \'etale covering
remains $\mu$-stable. If $\sV\otimes T$ is not $\mu$-stable, then $n= 2$ and 
$\sV^{\vee\vee}\cong T\otimes \sN$.
\end{claim}
\begin{proof} For a sheaf $\sV$ of rank two, the only irreducible Schur functors are
of the form $\{k-a,a\}$, for $a\leq\frac{k}{2}$. By \cite{FH91}, 6.9 on p. 79, one has
$$
{\mathbb S}_{\{k-a,a\}}(\sV)=\left\{ \begin{array}{lll} {\mathbb
S}_{\{k-2a\}}(\sV)=S^{k-2a}(\sV)\otimes\det(\sV)^a & \mbox{ if } & 2a < k\\
{\mathbb S}_{\{a,a\}}(\sV)=\det(\sV)^a & \mbox{ if } & 2a=k
\end{array} \right. .
$$
\begin{claim}\label{holonomy}
The sheaves $S^{m}(\sV)$ (and $S^{m}(T)$) are $\mu$-stable, for all $m$.
Moreover, the holonomy group of $S^{m}(T)$ with respect to the 
Hermite-Einstein metric is the full group ${\rm U}(n)$.
\end{claim}
\noindent
{\it Proof.}\,\,Otherwise, the holonomy group with respect to the Hermitian-Einstein metric
on $S^m(\sV)$ (or on $S^m(T)$) is not irreducible. Note that the holonomy group of the tensor product of Hermitian vector bundles is just the tensor product of the holonomy groups of the different factors. 

Consequently, a non-trivial splitting of $S^m(\sV)$ (resp.\ of $S^m(T)$) 
forces the holonomy group of $\sV$ (resp.\ of $T$)
with respect to the Hermite-Einstein metric to be strictly smaller than
${\rm U}(2)$ (resp.\ smaller than ${\rm U}(n)$).

It is known that a proper subgroup of ${\rm U}(2)$ is a semi-product of the torus with $\mathbb Z_2$. So one obtains a splitting of $\sV$ on some \'etale double cover.
\par
For $T$ we use instead \cite{Ya93} (see also \cite[Section 1]{VZ07}), saying that 
the holonomy group of $T$ is ${\rm U}(n)$. 
\end{proof}
Let us continue the proof of Claim~\ref{tangent}. Assume that $\sV\otimes T$  contains a subsheaf $\sN$ of the same slope and of rank $r< 2\cdot \rk(T)=2\cdot n$. Since $\sV\otimes T$ is $\mu$-polystable, $\sN$ is a direct factor. Replacing $\sN$ by its complement in
$\sV\otimes T$, if necessary, we may assume that $r \leq n$.

By taking the $r$-th wedge product one obtains an inclusion of $\sL=\bigwedge^r\sN$
into $\bigwedge^r(\sV\otimes T)$, and both sheaves have the same slope. Here and later on, the wedge products of 
a torsion free sheaf is the reflexive hull of the corresponding wedge product on the open set, where the sheaf is locally free.

By \cite[p.~80]{FH91}, for example, one has a decomposition
$$
\bigwedge^r(\sV\otimes T)=\bigoplus {\mathbb S}_{\lambda}(\sV)\otimes {\mathbb S}_{\lambda'}(T)
$$
where the sum is taken over all partitions $\lambda$ of $r$ with at most $2$ rows and $n$ columns
and where $\lambda'$ is the partition complementary to $\lambda$. The rank one subsheaf $\sL$ of $\bigwedge^r(\sV\otimes T)$ must inject to 
${\mathbb S}_{\lambda}(\sV)\otimes {\mathbb S}_{\lambda'}(T)$
for some $\lambda$. Again both sheaves are $\mu$-semistable of slope $\mu(\sL)$. Moreover,
for $\lambda=\{a,a\}$ the rank of ${\mathbb S}_{\lambda'}(T)$ is strictly larger than one, and the Claim~\ref{holonomy} implies that neither ${\mathbb S}_{\lambda}(\sV)$ nor
${\mathbb S}_{\lambda'}(T)$ can be invertible.

Let us assume that $n=2$. If $r=2$, the only possibilities for $\lambda$ are $\{2,0\}$ or $\{1,1\}$. In the first case ${\mathbb S}_{\lambda}(\sV)=\det(\sV)$, and in the second case
${\mathbb S}_{\lambda'}(T)=\det(T)$. So both are excluded.

If $\sN$ is a subbundle of rank one, we obtain a non-trivial map $\sN\otimes \Omega \to \sV$.
Since both sheaves are $\mu$-stable of the same slope this must be an isomorphism on some dense open subset, and since $\Omega=T\otimes \det(\Omega)$ we are done.

So assume from now on that $n\geq 3$. A non-zero projection of $\sL$ to some Schur functor  
$\sL\to {\mathbb S}_{\lambda}(\sV)\otimes {\mathbb S}_{\lambda'}(T)$
gives again rise to a non-zero map
$$
{\mathbb S}_{\lambda}(\sV)^\vee\otimes \sL\>>> {\mathbb S}_{\lambda'}(T)
$$
between $\mu$-polystable bundles of rank strictly larger than $1$ and of the same slope. Claim~\ref{holonomy} implies that this is an isomorphism.

Hence the holonomy group of ${\mathbb S}_{\lambda'}(T)$ with respect to the Hermitian-Yang-Mills
connection is isomorphic to the holonomy group of ${\mathbb S}_{\lambda}(\sV)^\vee$,
up to twisting by scalars. Holonomy groups are compatible with Schur functors, so the ${\mathbb S}_\lambda$-representation of the holonomy group of $\sV$ is isomorphic to ${\mathbb S}_{\lambda'}$ applied to the holonomy group of $T_Y$, which by Claim~\ref{holonomy} is ${\rm U}(n)$.

Since ${\mathbb S}_\lambda'$ is not the determinant representation,
this representation is almost faithful (with the kernel contained in the subgroup of scalar matrices). Since the holonomy group of $\sV$ is  ${\rm U}(2),$ it is too small to contain an almost faithful representation of ${\rm U}(n)$  for $n\geq 3$ one obtains 
a contradiction. So $n$ must be two, and we handled this case already.
\end{proof}
\par
\begin{example}\label{ex.3}
If $\ell=3$ and if $n\geq 3$, then the right hand side of (\ref{eqex.1}) is an equality, hence
$$
3\geq \varsigma(\V) = \frac{3\cdot \ell' \cdot(n+1)}{(3+\ell')\cdot n}> 1.
$$
For $\varsigma(\V)=3$ one finds $\ell'=n\cdot \ell$. For $\varsigma(\V)=2$
the only possibility is $n=\ell'=3$.
\end{example}
\begin{proof}
If $E^{1,0}$ is not $\mu$-stable, it has a torsion free quotient sheaf $\sV$ of slope $\mu(E^{1,0})$, either of rank one or of rank two. Both cases have been excluded, by the
Lemmata~\ref{invquot} and~\ref{ranktwo}.

For $\varsigma=\varsigma(\V)$ the equality implies that $\ell' = \frac{\varsigma\cdot 3 \cdot n}{(3-\varsigma)\cdot n + 3}$. For $\varsigma=1$ there is no solution in $\Z_{\geq 3}$, and for 
$\varsigma=2$ the only solutions are $(\ell',n)=(3,3), \ (4,6)$ or $(5,15)$.
To exclude the last two cases, consider the non-trivial map
$$
S^2(T)\otimes \det(E^{1,0}) \>\tau^{(2)}>> E^{1,0}\otimes \bigwedge^2(E^{0,1}). 
$$
Since both sides have the same slope, $\tau^{(2)}$ must be injective. However the inequality
$$
\frac{(n+1)\cdot n}{2} \leq \ell \cdot \frac{\ell'\cdot(\ell'-1)}{2}. 
$$
is violated for $(\ell',n)=(4,6)$ or $(5,15)$.
\end{proof}
\par
\begin{example}\label{lowcase}
For $n=2$ the right hand side of (\ref{eqex.1}) is an equality, except possibly for $\ell'=5$.
\end{example} 
\begin{proof} The inequality (\ref{eqex.1}) says that
$$
3\geq \varsigma(\V) \geq \frac{3\cdot \ell' \cdot 3}{(3+\ell')\cdot 2}.
$$
Since $\ell' \geq 3$ the right hand side is strictly larger than $2$, hence 
$\varsigma(\V)=3$, and the morphism
$$
\det(E^{1,0})\otimes S^3(T) \>\tau^{(3)}>> \bigwedge^{3}(E^{0,1})
$$
is non-zero. Since both sides have the same slope, for $\ell'=3$ this
contradicts the stability of $S^3(T)$. For $\ell'=4$ the saturated image of
$\tau^{(3)}$ is $\bigwedge^{3}(E^{0,1})$. Hence the latter and $E^{0,1}$
are both $\mu$-stable.
 The compatibility of the Higgs field with the socle filtration
implies that $E^{1,0}$ is $\mu$-stable, and hence the right hand side of (\ref{eqex.1}) must be an equality. Obviously this is a contradiction. 
\end{proof}
Altogether we verified:
\begin{proposition}\label{exclude}
Under the Assumptions~\ref{excludeass} the numerical condition 2) in Theorem~\ref{characterization}, b) holds in the following cases:
\begin{enumerate}
\item[1.] $n=1$.
\item[2.] $n=2$, $\ell \leq 3$, $\ell \leq \ell'$ and $\ell'\neq 5$.
\item[3.] $n\geq 3$, $\ell \leq 3$, and $\ell \leq \ell'$.
\end{enumerate}
\end{proposition}


\end{document}